\documentclass[a4paper,10pt]{article}   
\usepackage[utf8]{inputenc}    
\usepackage{graphicx}
\usepackage{subcaption} 
\usepackage{float}   
\usepackage{fancybox}		  
\usepackage{makeidx} 
\usepackage{verbatim}
\usepackage{listings} 
\usepackage{fancyvrb}
\usepackage{amsfonts}
\usepackage{color}
\usepackage{colortbl}
\usepackage{color}
\usepackage{dsfont}
\usepackage{epsfig}
\usepackage{bbm}
\usepackage{tikz}
\definecolor{dgreen}{rgb}{0.0, 0.5, 0.0}
\definecolor{byzantium}{rgb}{0.44, 0.16, 0.39}

\usepackage{amssymb}
\usepackage{amsbsy}
\usepackage{amsmath}
\usepackage{enumerate}
\usepackage{stmaryrd}
\usepackage{mathtools}
  \usepackage{ulem}
\usepackage{tikz}
\usepackage{bbm}

\usepackage{fancyhdr}

\usepackage{amsthm}

\newtheorem{prop}{Proposition}
\newtheorem{theorem}{Theorem}
\newtheorem{hyp}{Hypothesis}
\newtheorem{lemma}{Lemma}
\newtheorem{rem}{Remark}[section]
\newtheorem{definition}{Definition}

\newtheorem{taggedlemma}{Lemma}
\newenvironment{lemma_appendix}[1]
 {\renewcommand\thetaggedlemma{#1}\taggedlemma}
 {\endtaggedlemma}

\usepackage[margin=1in]{geometry}
\usepackage{hyperref}
\hypersetup{                    
    colorlinks=true,                
    breaklinks=true,                
    urlcolor= blue,                 
    linkcolor= red,                
    citecolor= blue                
    }

\makeatletter
\DeclareRobustCommand\widecheck[1]{{\mathpalette\@widecheck{#1}}}
\def\@widecheck#1#2{%
    \setbox\z@\hbox{\m@th$#1#2$}%
    \setbox\tw@\hbox{\m@th$#1%
       \widehat{%
          \vrule\@width\z@\@height\ht\z@
          \vrule\@height\z@\@width\wd\z@}$}%
    \dp\tw@-\ht\z@
    \@tempdima\ht\z@ \advance\@tempdima2\ht\tw@ \divide\@tempdima\thr@@
    \setbox\tw@\hbox{%
       \raise\@tempdima\hbox{\scalebox{1}[-1]{\lower\@tempdima\box
\tw@}}}%
    {\ooalign{\box\tw@ \cr \box\z@}}}

\def\N {\mathbb{N}}

\def\R {\mathbb{R}}

\newcommand\pare[1]{\left(#1\right)}

\newcommand\croch[1]{\left[#1\right]}

\newcommand{\sumi}{\sum_{i=1}^N}
\newcommand{\sumj}{\sum_{j=1}^N}

\newcommand{\elts}{\{1,\cdots,N\}}

\renewcommand{\P}{\mathcal{P}}

\newcommand{\E}{\mathbb{E}}

\newcommand{\one}{\mathbb{1}}
\newcommand{\schema}[1]{\b{\sc #1}}

\newcommand{\xij}{\xi_{ij}}
\newcommand{\uin}{u_i^N}
\newcommand{\ujn}{u_j^N}
\newcommand{\uine}{u_i^{N,\varepsilon}}
\newcommand{\ujne}{u_j^{N,\varepsilon}}
\newcommand{\uinei}{\tilde{u}_i^{N,\varepsilon,k}}
\newcommand{\ujnei}{\tilde{u}_j^{N,\varepsilon,k}}
\newcommand{\uina}{u_i^{N,\text{Av}}}
\newcommand{\ujna}{u_j^{N,\text{Av}}}
\newcommand{\vin}{v_i^N}
\newcommand{\vjn}{v_j^N}

\newcommand{\dsp}{\displaystyle}
\newcommand{\wijn}{W_{ij}^N}
\newcommand{\fijn}{f_{ij}^N}
\newcommand{\wijk}{W_{ik}^N}
\newcommand{\fijk}{f_{ik}^N}

\newcommand{\bw}{\bar{w}}
\newcommand{\V}{\mathbb{V}}
\newcommand{\PP}{\mathbb{P}}

\newcommand{\sz}{\sigma_Y}

\newcommand{\talpha}{\tilde{\alpha}_N}
\newcommand{\tgamma}{\tilde{\gamma}_N}

\newcommand{\tZ}{\tilde{Z}}
\newcommand{\tY}{\tilde{Y}}

\definecolor{carnelian}{rgb}{0.7, 0.11, 0.11}

\definecolor{crimsonglory}{rgb}{0.75, 0.0, 0.2}
\definecolor{darkraspberry}{rgb}{0.53, 0.15, 0.34}
\definecolor{darkmagenta}{rgb}{0.55, 0.0, 0.55}

\begin{document}

\title{Graph Limit for Interacting Particle Systems on Weighted Random Graphs}

\author{
Nathalie Ayi\thanks{Sorbonne Universit\'e, CNRS, Universit\'e Paris Cit\'e, Inria, Laboratoire Jacques-Louis Lions (LJLL), F-75005 Paris, France} 
 \and
Nastassia Pouradier Duteil\thanks{Sorbonne Universit\'e, Inria, CNRS, Universit\'e Paris Cit\'e, Laboratoire Jacques-Louis Lions (LJLL), F-75005 Paris, France} }

\maketitle

\abstract{In this article, we study the large-population limit of interacting particle systems posed on weighted random graphs. In that aim, we introduce a general framework for the construction of weighted random graphs, generalizing the concept of graphons. We prove that as the number of particles tends to infinity, the finite-dimensional particle system converges in probability to the solution of a deterministic graph-limit equation, in which the graphon prescribing the interaction is given by the first moment of the weighted random graph law. 
We also study interacting particle systems posed on switching weighted random graphs, which are obtained by resetting the weighted random graph at regular time intervals. We show that these systems converge to the same graph-limit equation, in which the interaction is prescribed by a constant-in-time graphon. 
}

\section*{Introduction}

Models for interacting particle systems provide a general framework to study any population of agents (also referred to as ``particles'') interacting pairwise, and generally giving rise to collective behavior without any centralized intelligence. These models can be refined and adapted to many applications, as varied as animal behavior \cite{CKFL05,CS07}, opinion formation \cite{HK,BST22}, cell movement \cite{BDPZ20}, crowd motion \cite{MV11}, etc.

Interacting particle systems can broadly be grouped into two categories: models for indistinguishable particles, and models for non-indistinguishable (or non-exchangeable) particles \cite{AyiPou}.
In the former, the interaction between particles is exclusively based on their position in the state space, so that exchanging two particles has no effect on the overall dynamics. In the latter, particles are tagged with a specific identity, and their interaction also depends on an underlying interaction graph: they are non-exchangeable.

In this article, we aim to derive the large-population limit of coupled dynamical systems posed on \textit{random weighted graphs}, of the form 
\begin{equation} \label{eq:gensyst}
\frac{d}{dt} \uin(t) = \frac{1}{N} \sum_{j=1}^N  \xi_{ij} D(\ujn(t) - \uin(t)),~~~~ i \in \{1, \dots, N\}.
\end{equation}
Here, the time-evolving variables $(\uin)_{i\in\elts}\in \R^N$ can represent oscillator frequencies as in the Kuramoto model \cite{Kuramoto84, MedvedevR}, opinions as in the Hegselmann-Krause model \cite{HK}, neuron potentials \cite{Hoppensteadt97}, positions, or any other quantity likely to evolve based on an underlying graph. 
The interaction function $D$ regulates the interaction due to the values of $\uin$ and $\ujn$ at time $t$, \textit{independently of the underlying network}. The network's role, on the other hand, is encoded in the weights $(\xi_{ij})_{i,j\in\elts}\in(\R_+)^N$, which depend exclusively on the nodes $i$ and $j$, and not on the values of the evolving variable. 


The relevance of such dynamical systems posed on random interaction networks has been pointed out in many applications. 
For instance, Watts and Strogatz have introduced a model for a ``small-world'' network, that they apply to systems as varied as the \textit{C. elegans} neural network, the US power grid or the social structure of the movie actors community  \cite{WS98}.
 The construction procedure involves connecting each node with its closest neighbors to form a ring lattice, and then rewiring each edge at random with probability $p$.
The constructed network reflects the well-known ``small-world'' property according to which each individual has a small probability to be connected with another individual supposedly outside its circle. 

Because of their large size, dynamical systems of the form \eqref{eq:gensyst} are generally difficult to analyze. However, it has been shown in many cases that their continuum limit (as $N$ goes to infinity) often provides a good approximation of their behavior \cite{WileyStrogatzGirvan06}. 
The limit of such a system posed on an underlying deterministic graph was derived in \cite{Medvedev14} in the graph limit framework, and in  \cite{JabinPoyatoSoler21, KuehnXu22} in the mean-field framework.
In \cite{MedvedevR}, Medvedev derived the limit of \eqref{eq:gensyst} when $N$ goes to infinity, in the case of a $W-$random graph in which the weights take random values in $\{0,1\}$, signifying the presence ($\xi_{ij}=1$) or absence ($\xi_{ij}=0$) of edges. Given a graphon $W\in L^\infty([0,1]^2;[0,1])$, the probability for the presence or absence of each edge is assumed to satisfy $\PP[\xi_{ij}=1]=1-\PP[\xi_{ij}=0]=W(x_i,x_j)$, where $x_i, x_j$ are deterministic or random variables that depend respectively on $i$ and $j$.
The limiting equation is the nonlinear heat equation
\begin{equation}\label{eq:graph-limit_W_intro}
\partial_t u(t,x) = \int_I W(x,y) D(u(t,y)-u(t,x)) \, dy ,
\end{equation}
in which the graphon $W$ now plays the role of \textit{deterministic} edge weights between the (continuous) vertices indexed by $x$ and $y$. Thus, the limit of a dynamical system posed on a \textit{random unweighted} graph is a nonlinear heat equation posed on a \textit{deterministic weighted} graphon: the stochasticity of the system is lost as $N$ goes to infinity.

However, as remarked in \cite{BarratBarthelemyPastorVespignani04}, binary (or unweighted) graphs, in which each edge is either present or absent, lack essential features needed to fully capture the complexity of a real network. This motivated the introduction of \textit{weighted graphs} in which each edge is assigned a weight proportional to the intensity of the connection between the nodes it links. The worldwide airport network is an example of a weighted graph in which the nodes model the cities containing an airport, the edges model the existing direct-flight connections between any two cities, and the edge weights model the intensity of these connections (for instance in number of seats per year) \cite{BarratBarthelemyPastorVespignani04}. Many real-life networks are in fact better modeled by weighted graphs, where the weights can represent communication frequencies or intensities. It then seems natural to couple the notions of \textit{random graphs} and \textit{weighted graphs} to define \textit{weighted random graphs}.

A \textit{weighted graph} of size $N$, denoted by $G_N = \langle V(G_N), E(G_N), W(G_N)\rangle$, is composed of $N$ vertices and of a maximum of $N^2$ edges. With this notation, $V(G_N)=\elts$ denotes the set of vertices, $E(G_N)\subset\elts^2$ the set of edges, and $W(G_N)$ the set of edge weights, considered to be non-negative. 
For weighted graphs, one can always consider that the number of edges is exactly equal to the maximal number $N^2$, by attributing zero-weights to some edges, as follows: 
$$(W(G_N))_{ij} := \left\{
\begin{array}{ll}
\dsp{\xi_{ij}}\in \R^*_+ ,&  \text{ if } (i,j) \in E(G_N)\\
~\\
\dsp{0} &\text{ otherwise.}
\end{array} \right. $$
A weighted graph is said to be a \textit{weighted random graph} if the edge weights are attributed randomly. 

In \cite{Garlaschelli_2009}, Garlaschelli introduced a weighted random graph model in which the probability of drawing an edge of discrete weight $w\in \N$ between vertices $i$ and $j$ is given by $\PP[\xi_{ij}=w]:=q_{ij}(w) = (y_i y_j)^w (1-y_i y_j)$, where $(y_i)_{i\in\elts}$ tune the expected stength of the vertices. In \cite{FossKonstantopoulos18}, a general model was introduced permitting to assign random weights to the egdes and to the vertices, so that all edge weights are i.i.d. random variables.
In \cite{AminiLelarge15}, the authors introduce and study a weighted random graph in which all edge weights are i.i.d., following the law of an exponential random variable.

Although specific examples of \textit{weighted random graphs} such as the few above can be found in the literature, to the best of our knowledge there lacks a general framework encompassing all models into one.
In the present article, we propose a general framework for \textit{weighted random graphs}, able to encompass specific examples such as those given in \cite{AminiLelarge15, FossKonstantopoulos18,Garlaschelli_2009, WS98} as well as unweighted random graphs (which can be seen as weighted random graphs with weights belonging to $\{0,1\}$).
Notably, our framework allows for non identically distributed edges.

In this aim, we define a \textit{weighted random graph law} $q:[0,1]\times [0,1] \rightarrow \mathcal{P}(\R_+)$, i.e. a function taking values in $[0,1]^2$, such that for every $(x,y)\in [0,1]^2$, $q(x,y;\cdot)$ is a probability measure.
Similarly to the approach in \cite{MedvedevR}, we can then define a $q-$weighted random graph of size $N$ in two ways, that differ from one another in their level of randomness:
\begin{enumerate}
\item[(r-r)] Given a sequence of i.i.d. random variables $(X_i)_{i\in \N}$ uniformly distributed in $[0,1]$, for each edge $(i,j)$, we randomly select the weight $\xi^N_{ij}$ with probability $q(X_i,X_j;\cdot)$.
\item[(r-d)] Given a deterministic sequence $(x^N_i)_{i\in\elts}$ satisfying $x^N_i\in [\frac{i-1}{N}, \frac{i}{N})$ for all $i$, we randomly select the value of $\xi^N_{ij}$, with probability $q(x^N_i,x^N_j;\cdot)$.
\end{enumerate}
The approach (r-r) (standing for random-random) requires randomness in both steps: that of randomly attributing the random variables $X_i$ to each node $i$, and that of randomly selecting the weights $\xi_{ij}$ from the law $q(X_i,X_j;\cdot)$. It is the approach presented in \cite{MR2274085}.
The approach (r-d) (for random-deterministic) requires randomness only in the second step, and has been presented in \cite{MedvedevR}. The advantage of the second one is that it is convenient for applications, since as mentioned in \cite{MedvedevR}, it can often easily be linked  to  existing random graph models (as for instance the interpretation of the ``small-world'' graph as a $W$-random graph generated by a deterministic sequence).

Figure \ref{Fig:Intro_SmallWorld_matrices} shows pixel representations of random matrices $(\xi_{ij})_{i,j\in\elts}$ corresponding to a weighted version of the small-world network (in the case (r-d)). The right-most plot shows the limit graphon $(x,y)\mapsto \bw(x,y):= \int_{\R_+} w \, q(x,y; dw) $.

\begin{figure}[h!]
\centering
\includegraphics[width = 0.32\textwidth, trim = 0cm 0cm 0cm 0cm, clip=true]{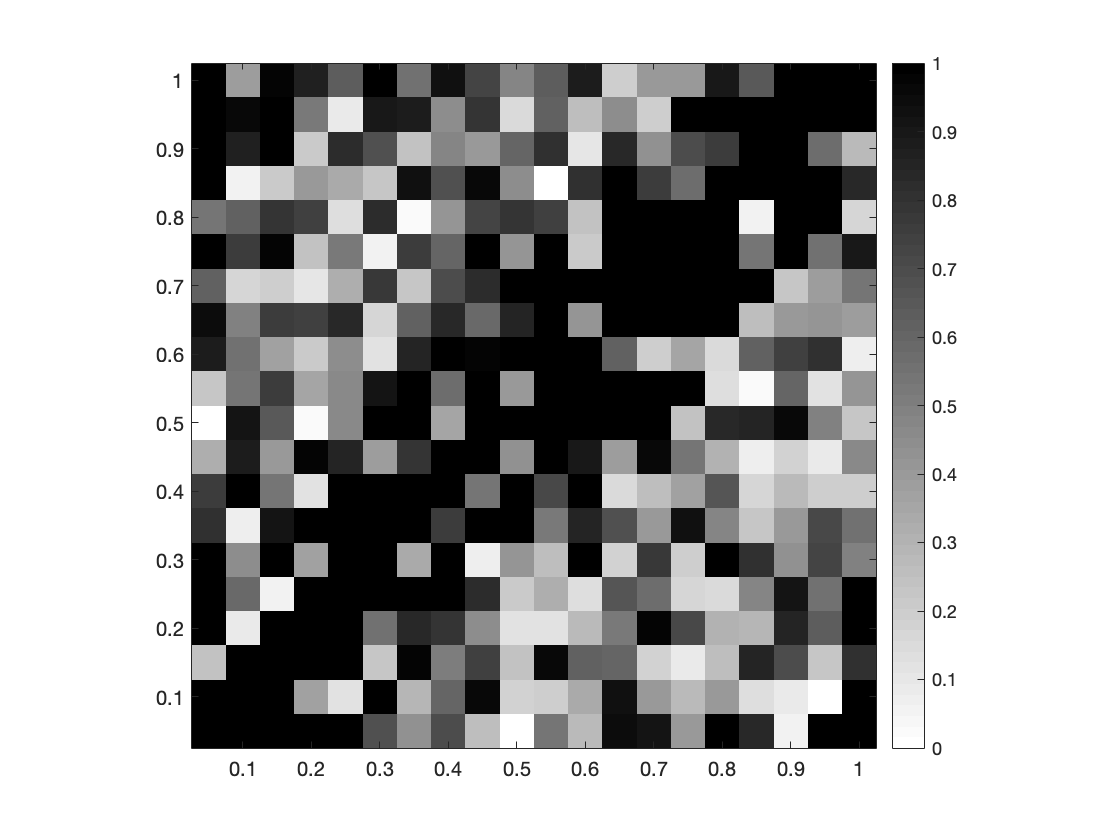}
\includegraphics[width = 0.32\textwidth, trim = 0cm 0cm 0cm 0cm, clip=true]{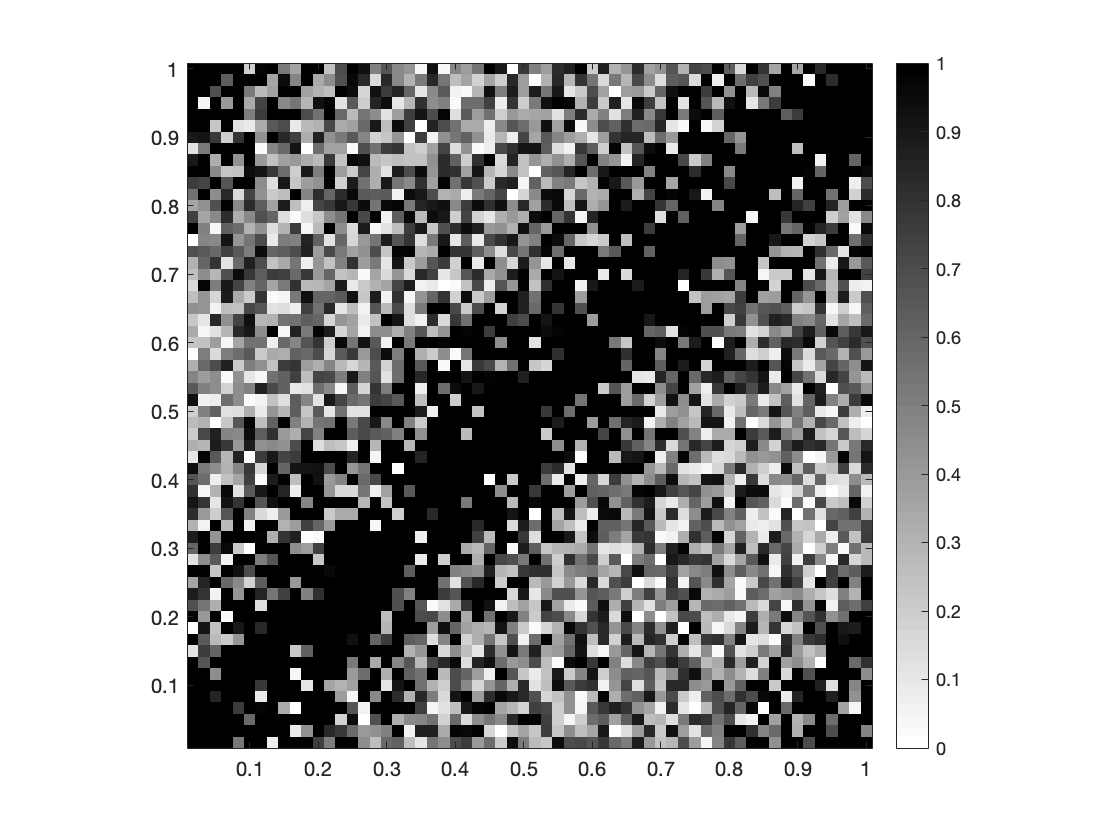}
\includegraphics[width = 0.32\textwidth, trim = 0cm 0cm 0cm 0cm, clip=true]{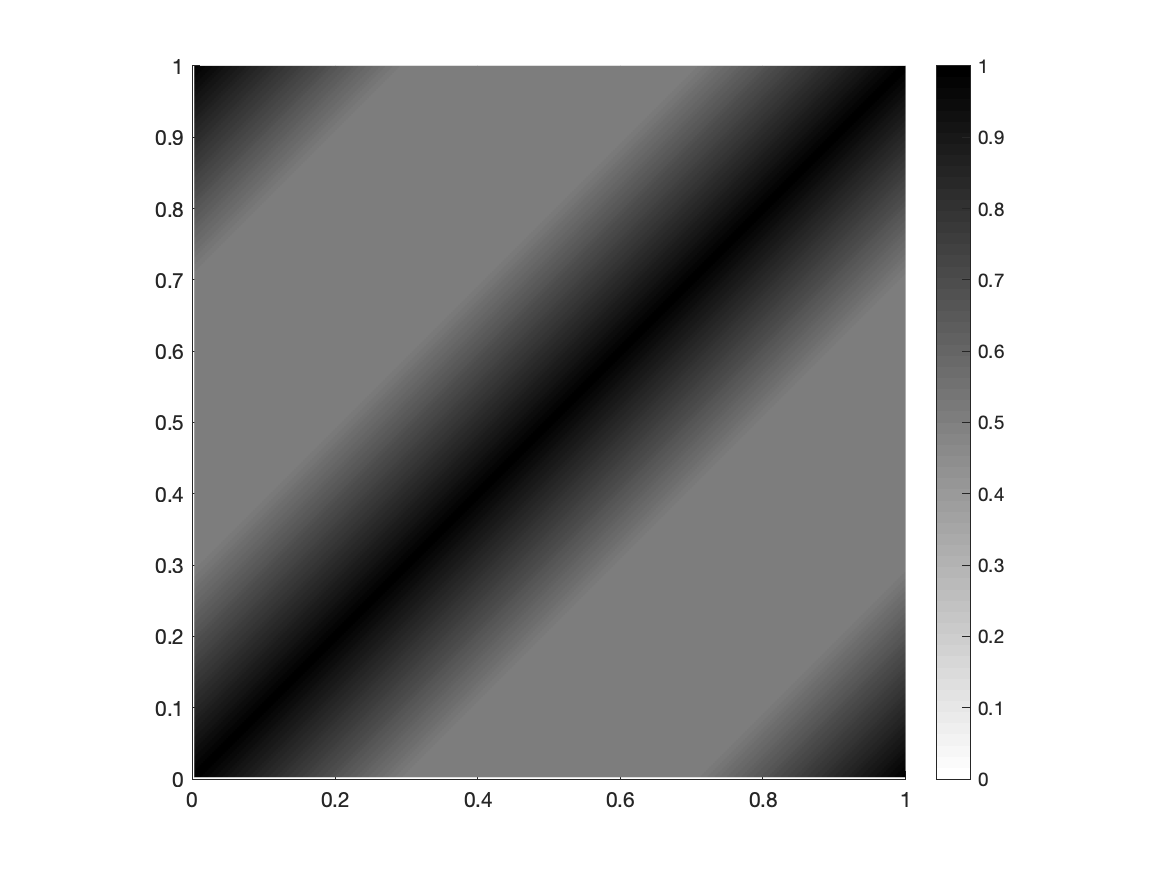}
\caption{Values of the random interaction matrices generated from a deterministic sequence according to the Small-World weighted random graph law \eqref{eq:Example_SmallWorld} for $N=20$ and $N=60$. Right: Corresponding continuous graphon $(x,y)\mapsto \bw(x,y):= \int_{\R_+} w \, q(x,y; dw) $. }
\label{Fig:Intro_SmallWorld_matrices}
\end{figure} 

In each of these cases, we will prove that as $N$ tends to $+\infty$, the microscopic system \eqref{eq:gensyst} converges in probability to the solution of the following graph limit equation: 
\begin{equation}\label{eq:graph-limit_intro}
\partial_t u(t,x) = \int_I \left(\int_{\R_+} w \, q(x,y; dw) \right) D(u(t,y)-u(t,x)) \, dy ,
\end{equation}
which is a \textit{deterministic} integro-differential equation, in which the weight of the edge $(x,y)$ is given by the expected value of the weighted random graph law $q(x,y;\cdot)$. 
More precisely, if the probability measures $q(x,y;\cdot)$ have uniformly bounded first four moments, we prove that on any finite time interval $[0,T]$, the solutions $u^N$ to the discrete system \eqref{eq:gensyst} and the solution $u$ to the integro-differential equation~\eqref{eq:graph-limit_intro} satisfy
\begin{equation*}
 \mathbb{P} \croch{ \sup_{t \in [0,T]}  \|u^N(t) - \mathbf{P}_{\tilde{X}^N} u(\cdot,t) \|_{2,N} \geq \frac{C(T)}{\sqrt{N}}} \leq \frac{\tilde{C}}{N},
\end{equation*}
where the constants can be computed explicitely. This quantitative result is obtained in both cases (r-r) (Theorem \ref{th:graph_limit}) and (r-d) (Theorem \ref{th:deterministic}), but in the latter, additional regularity assumptions are required on the initial data, the weighted random graph law $q$ and the interaction function $D$.

We then focus our attention on time-dependent weighted random networks. 
Again, this consideration is motivated by practical applications \cite{belykh_SW}. In neural networks, neurons interact via electric spikes that take place only intermittently. In packet switching technology (such as the Internet), data channels are occupied only during the transmission of information packets. Both are examples of ``blinking networks'', in which connections are sporadic.
In \cite{belykh_SW}, Belykh et al introduced the so-called ``blinking model'' in the context of a small-world network composed of a regular locally coupled lattice, to which is superimposed a time-varying random small-world network. The varying small-world network is obtained by randomly switching on or off some new shortcuts on the graph every lapse of time $\varepsilon$.
Since then, it has been studied for many dynamical systems on graphs (as in \cite{belykh_kuramoto,FaggianGinelliRosasLevnajic19} for the Kuramoto oscillators models for instance). 
It has been proven that when the ``blinking'' time is small compared to the characteristic synchronization time, a few random shortcut additions significantly lower the synchronization threshold of the system. 

We extend this concept to the framework of weighted random graphs. 
Applying results from Averaging theory, we show that for a given $N$, when the ``blinking'' time tends to zero, the blinking system tends to a system in which the communication weights are constant, and given by the expected values of the blinking weights.
We then prove that the continuous limit (as $N \varepsilon$ goes to infinity) of the finite-dimensional blinking system is the same graph limit equation \eqref{eq:graph-limit_intro}, obtained as the limit of non-blinking systems.

The article is organized as follows. After defining the framework of weighted random graphs in Section \ref{Sec:Prelim}, we prove the convergence of the coupled dynamical system \eqref{eq:gensyst} towards its graph limit \eqref{eq:graph-limit_intro} in the context of random graphs generated by random (Section \ref{Sec:Random}) or deterministic (Section \ref{Sec:Deter}) sequences.
Section \ref{Sec:blinking} is devoted to the study of the interplay between the blinking time and the number of agents in dynamical systems evolving on blinking weighted random graphs. We illustrate our results with numerical simulations in Section \ref{Sec:Sim}.


\section{Notations and preliminary concepts}\label{Sec:Prelim}

In this paper, we will use the notation $I:=[0,1]$. $\P(\R_+)$ denotes the set of probability measures with support in $\R_+$. We will denote by $\E$ and $\V$ the expectation and variance of random variables.

In everything that follows, $D:\R\rightarrow\R$ will denote the interaction function, and $g:I\rightarrow\R$ will provide the initial condition of both microscopic and macroscopic systems. 
We will make the following technical assumptions:
\begin{hyp}\label{hyp:D}
Let $D\in L^\infty(\R)$ be bounded 
and Lipschitz continuous, with $\|D\|_{L^\infty(\R)}:=K$ and $\|D'\|_{L^\infty(\R)}:=L$. 
\end{hyp}

In the seminal paper \cite{LS}, a procedure was introduced to construct \textit{unweighted} random graphs from a limiting object named graphon. Formally, a graphon is a measurable function 
\[
\begin{array}{rccl}
W :& I\times I &\rightarrow& [0,1]\\
  &(x,y)&\mapsto& W(x,y).
\end{array}
\]
Using the graphon $W$, a $W$-random graph can then be constructed in two ways \cite{VGF21}: 
\begin{enumerate}
\item[(r-r)] $W$-random graph generated by a random sequence: given a sequence of i.i.d. random variables $(X_i)_{i\in \N}$ uniformly distributed in $[0,1]$, an edge between the nodes $i$ and $j$ is attributed with probability $W(X_i,X_j)$.
\item[(r-d)] $W$-random graph generated by a deterministic sequence: given a sequence $(x_i)_{i\in \N}$ satisfying for all $i\in\elts$, $x_i\in [\frac{i-1}{N}, \frac{i}{N})$, an edge between the nodes $i$ and $j$ is attributed with probability $W(x_i,x_j)$.
\end{enumerate}

We introduce the concept of \textit{weighted random graph law}, which will allow us to generalize $W$-random graphs to $q$-weighted random graphs. Since it will be central to all that follows, we define it here.
\begin{definition}
A \textbf{weighted random graph law} is a function
$$
\begin{array}{llll}
q :&I \times I &  \to & \mathcal{P}(\mathbb{R_+})\\
 & (x,y) &\mapsto & q(x,y;\cdot).
\end{array}$$
\end{definition}
In this article, we will only consider weighted random graph laws with uniformly bounded first four moments, requiring the following:
\begin{hyp}
\label{hypo_moment}
There exists $M>0$ such that for all $(x,y) \in I^2$, for all $k\in\{1,\cdots,4\},$ 
\begin{equation}\label{eq:moments}
\left( \int_{\mathbb{R}_+} w^k q(x,y;dw)\right)^{1/k} \leq M,
\end{equation}
i.e. the first four moments of the probability measure $q(x,y;\cdot)$ are bounded uniformly in $x$ and $y$.
\end{hyp} 
From here onwards, for all $x,y\in I$, we denote by $ \bw(x,y) := \int_{\mathbb{R_+}}  w q(x,y;dw)$ the first moment of the probability distribution $q(x,y;\cdot)$.

Similarly to the construction of $W$-random graphs, we will also consider two sampling methods to construct a $q$-weighted random graph. 
The first one, in which the graph is generated from a sequence of independent uniformly distributed random variables $(X_i)_{i\in \N}$, is introduced in Section \ref{Sec:Random}.
The second one, in which the graph is generated from a deterministic sequence of  points $(x_i)_{i\in \N}$, is introduced in Section \ref{Sec:Deter}.

Given $\dsp (x^N)_{i\in\elts} :=\{x_1^N, x_2^N, \dots, x_N^N\}$ a set of distincts points in $I$, for all $t\in [0,T]$, we denote the projection of a function $f\in C(I)$ onto $x^N$ by the evaluation of $f$ at these points, i.e.
\begin{equation}
\mathbf{P}_{x^N} f = (f(x_1^N), f(x_2^N), \dots, f(x_N^N)) \subset \R^N,
\end{equation}
which is a vector of $\R^N$.
For such vectors $u,v\in\R^N$, we use the weighted Euclidean inner product
\begin{equation*}
(u,v)_N := \frac{1}{N} \sum_{i=1}^N u_i v_i, \quad u=(u_1,u_2, \dots,u_N)^T,  v=(v_1,v_2, \dots,v_N)^T
\end{equation*}
and the corresponding norm $\|u\|_{2,N} := \sqrt{(u,u)_N}$. 


\section{Networks on weighted random graphs generated by random sequences}\label{Sec:Random}

Denote $X = (X_1, X_2, X_3, \dots )$ and $X^N = (X_1, X_2, \dots, X_N)$ where $X_i, i \in \mathbb{N}$ are independent identically distributed (i.i.d.) random variables. We suppose that $X_1$ follows the law of a uniform random variable on $I$, i.e. $\mathcal{L}(X_1) = \mathcal{U}(I)$. We define general \textit{weighted random graphs}, which generalize the study of particular weighted (or unweighted) random graph models like in \cite{AminiLelarge15,Garlaschelli_2009,LS, WS98}.

\begin{definition}
Let
$$
\begin{array}{llll}
q :&I \times I &  \to & \mathcal{P}(\mathbb{R_+})\\
 & (x,y) &\mapsto & q(x,y;\cdot)
\end{array}$$
be a weighted random graph law.
A \textbf{$q$-weighted random graph} on $N$ nodes \textbf{generated by the random sequence $X$}, denoted $G_N$, is such that the weight of each edge $(i,j), i \neq j,$ of $G_N$ is randomly attributed, and its law is $q(X_i,X_j;\cdot)$. \\
The decision of the attribution of the weight of a pair $(i,j) \in \{1, \dots, N\}^2, i \neq j$ is made independently from the decision for other pairs.
\end{definition}

\begin{rem}\label{rem:examples}
This general definition of $q$-weighted random graphs encompasses many specific examples, including unweighted random graphs.
\begin{itemize}
\item \textbf{$W$-random graphs} \cite{LS}: given a graphon $W:I^2\rightarrow [0,1]$, an (unweighted) $W$-random graph can be defined with the Bernoulli weighted random graph law: $q(x,y,\cdot)=W(x,y)\delta_1 + (1-W(x,y))\delta_0$
\item Weighted Erdös-Rényi random graph: the model of Garlaschelli \cite{Garlaschelli_2009} can be recovered with the weighted random graph law : $q(x,y;\cdot) =  \dsp (1-\frac{xy}{2}) \sum_{i=0}^{+\infty}(\frac{xy}{2})^i\delta_i$.
\item Weighted random graph with i.i.d. exponential weights \cite{AminiLelarge15}:
$q(x,y;dw) =  \index{w\geq 0} \lambda e^{-\lambda w} dw$
\end{itemize} 
\end{rem}

Given $(X_i)_{i\in\elts}$ a sequence of i.i.d. random variables satisfying $\mathcal{L}(X_1) = \mathcal{U}(I)$, we consider the system of differential equations
\begin{equation}\tag{$\mathcal{S}_N^\mathrm{r-r}$}
\label{eq:ODEs}
\begin{cases}
\displaystyle \frac{d}{dt} \uin(t) = \frac{1}{N} \sum_{j=1}^N  \xi_{ij} D(\ujn(t) - \uin(t)),\\ 
\uin(0)  =g(X_i^N),~~~~   i \in \{1, \dots, N\}
\end{cases}
\end{equation}
where for all $(i,j)\in\elts^2$, $\mathcal{L}(\xi_{ij}| {X}) = q(X_i,X_j;\cdot)$.
We denote this system \eqref{eq:ODEs} due to its twice random nature: for each pair $(i,j)\in\elts^2$, the weight $\xi_{ij}$ is \textit{randomly} attributed from the law $q(X_i,X_j;\cdot)$, generated by the \textit{random} variables $X_i$ and $X_j$.
Our goal will be to prove the convergence of the microscopic system \eqref{eq:ODEs} towards the continuum limit
\begin{equation}\tag{C}
\begin{cases}
\label{eq:graph_limit}
\partial_t u(x,t)  = \dsp \int_I \left( \int_{\mathbb{R_+}}  w q(x,y;dw)\right) D(u(y,t)-u(x,t))dy\\
u(x,0)   =  g(x),  ~~~x \in I,
\end{cases}
\end{equation}
in a sense that we will clarify.

We start by recalling the well-posedness of the limit equation \eqref{eq:graph_limit}. 
\begin{theorem}
Suppose that $D$ is Lipschitz continuous, that there exists $M>0$ such that for all $ (x,y) \in I^2$, $\displaystyle \int_{\R_+} w q(x,y;dw) \leq M$ and that $g \in L^\infty(I)$. 
Then, for any $T>0$, there exists a unique solution $u \in \mathcal{C}^1([0,T];L^\infty(I))$ to \eqref{eq:graph_limit}.
\end{theorem}
We refer the reader to \cite{Medvedev14} for a proof in a more general framework.
Our main result can then be stated as follows:

\begin{theorem}
\label{th:graph_limit}
Let $D$ satisfy Hyp. \ref{hyp:D}, let $g \in L^\infty(I)$ and let $q$ be a weighted random graph law satisfying Hyp.~\ref{hypo_moment}.
Then, as $N$ goes to infinity, the solution $u^N$ to the discrete system \eqref{eq:ODEs} converges to the solution $u$ of the continuous model \eqref{eq:graph_limit}. More precisely, 
\begin{equation}\label{eq:conv1}
 \mathbb{P} \croch{ \sup_{t \in [0,T]}  \|u^N(t) - \mathbf{P}_{\tilde{X}^N} u(\cdot,t) \|_{2,N} \geq \frac{C_1(T)}{\sqrt{N}}} \leq \frac{\tilde{C}_1}{N}
\end{equation}
where the constants $C_1(T)$ and $\tilde{C}_1$ are respectively defined by
$C_1(T):= \sqrt{T} \sqrt{1 + M^2 K^2 }e^{(\frac{1}{2}+4ML)T}$ and
$\tilde{C}_1 := 3 M^4 K^4+ 6$.
\end{theorem}

\begin{rem}
This theorem generalizes and improves previous results concerning the convergence on $W$-random graphs. In \cite{MedvedevR}, dynamical systems on $W$-random graphs were shown to converge in probability to the solution to \eqref{eq:graph-limit_W_intro}. With the $q$-weighted random graph formulation, this result is covered by Theorem~\ref{th:graph_limit}, since as seen in Remark~\ref{rem:examples}, writing $q(x,y;dw)=W(x,y)\delta_1 + (1-W(x,y))\delta_0$ yields:
\[
\int_{\R_+} w q(x,y;dw) = W(x,y).
\]
Moreover, unlike in \cite{MedvedevR}, in which proofs are applications of the Central Limit Theorem, our proof of this result will rely on multiple applications of the Bienaymé-Chebyshev inequality. This allows us to obtain an explicit rate for the convergence in probability, as stated in \eqref{eq:conv1}.
\end{rem}

The proof of Theorem \ref{th:graph_limit} relies on two technical lemmas.
The first one quantifies the asymptotic bounds of two random variables $\alpha_N$ and $\gamma_N$ defined from the family of i.i.d. random variables $(\xi_{ij})_{i,j\in\elts}$ as follows: 
\begin{equation}\label{def:alphagamma}
\alpha_N := \left( \frac{1}{N^2} \sumi \sumj \xi_{ij}^2\right)^\frac{1}{2} \qquad \text{ and } \qquad \gamma_N := \max_{i\in\elts} \frac{1}{N} \sumj \xi_{ij}.
\end{equation}
As can be expected from the Central Limit Theorem, although there is no bound on the individual values of $(\xi_{ij}^2)_{i,j\in\elts}$, their sum can be shown to be bounded with probability tending to $1$ as $N$ goes to infinity. Moreover, the bound on the first four moments of $q$ (Hyp.~\ref{hypo_moment}) allows us to use the Bienaymé-Chebyshev inequality, and we can even quantify the rate of convergence to be $N^{-2}$.
The bound on $\gamma_N$ is slightly less trivial to obtain, due to the fact that it is defined as the maximum of a sum. However, using a generalization of the Bienaymé-Chebyshev inequality, we also show that it is bounded with probability going to $1$ as $N$ tends to infinity, but with a rate only proportional to $N^{-1}$. 

\begin{lemma}\label{Lemma:alphagamma}
Let $X=(X_i)_{i\in\elts}$ be a sequence of i.i.d. random variables satisfying $\mathcal{L}(X_1) = \mathcal{U}(I)$, and $(\xi_{ij})_{i,j\in\elts}$ be the weights of a $q$-weighted random graph generated by the random sequence $X$.
Then the two random variables $\alpha_N$ and $\gamma_N$ defined in \eqref{def:alphagamma} satisfy for $N$ large enough: 
\[
\PP\croch{ \alpha_N \geq 2M} \leq \frac{1}{N^2} \qquad \text{ and } \qquad 
\PP\croch{ \gamma_N \geq 2M} \leq \frac{5}{N},
\]
where $M$ is the uniform bound on the moments of $q$ defined in \eqref{eq:moments}.
\end{lemma}

\begin{proof}
We aim to apply the Bienaymé-Chebyshev inequality to the random variable $\alpha_N$.
Recall that $(\xi_{ij})_{i,j\in\elts}$ are i.i.d. random variables whose expectations and variances satisfy for all $(i,j)\in\elts^2$:
\begin{equation*}
E[\xi_{ij}^2]= \iint_{I^2}\int_{\R_+} w^2 q(x,y;dw) dx dy  \leq M^2
\end{equation*}
and 
\begin{equation}\label{eq:boundxicarre}
\V[\xi_{ij}^2]= \iint_{I^2}\int_{\R_+} w^4 q(x,y;dw) dx dy - \left(\iint_{I^2}\int_{\R_+} w^2 q(x,y;dw) dx dy\right)^2 \leq M^4,
\end{equation}
so the Bienaymé-Chebyshev inequality implies:
\[\begin{array}{rcl}
\dsp \PP\left[\frac{1}{N^2}\sum_{i=1}^N \sum_{j=1}^N \xi_{ij} ^2 - \frac{1}{N^2} \sum_{i=1}^N \sum_{j=1}^N \E[\xi_{ij}^2] \geq M^2 \right] & = &  \dsp \PP\left[\alpha_N^2 - \E[\alpha_N^2] \geq M^2 \right] 
  \leq  \dsp \PP\left[\left| \alpha_N^2 - \E[\alpha_N^2] \right| \geq M^2 \right] \\
 & \leq &\dsp  \frac{1}{M^4} \V[\alpha_N^2] = \frac{1}{M^4 N^4} \sum_{i=1}^N \sum_{j=1}^N \V[\xi_{ij}^2] \leq \frac{1}{N^2}.
\end{array}
\]
Thus, 
\[
 \PP\left[\alpha_N^2 \geq M^2 +  M^2 \right] \leq \PP\left[\alpha_N^2 \leq  \frac{1}{N^2} \sum_{i=1}^N \sum_{j=1}^N \E[\xi_{ij}^2] +  M^2 \right]  \leq \frac{1}{N^2},
\]
which implies
\[
 \PP\left[\alpha_N \geq 2  M \right]  \leq \PP\left[\alpha_N^2 \geq 2  M^2 \right] \leq \frac{1}{N^2},
\]
concluding the first part of the statement. 

To study the term $\gamma_N$, we apply a generalization of the Bienaymé-Chebyshev inequality to the random variables $(\frac{1}{N}\sum_{j=1}^N\xi_{ij})_{i\in\elts}$ and obtain
\[\begin{array}{rcl}
\dsp  \PP\croch{{\frac{1}{N}}\sum_{j=1}^N\xi_{ij} - \frac{1}{N}\sum_{j=1}^N\E[\xi_{ij}]   \geq M} & \leq & \dsp   \PP\croch{\left|{\frac{1}{N}}\sum_{j=1}^N\xi_{ij} - \frac{1}{N}\sum_{j=1}^N\E[\xi_{ij}] \right|   \geq M} \\
 & \leq & \dsp \frac{1}{M^4} \E\left[\left|\frac{1}{N} \sum_{j=1}^N \xi_{ij} - \E\left[\sum_{j=1}^N \xi_{ij}\right]\right|^4\right].
\end{array}
\]
Besides, 
\[\begin{array}{rcl}
\dsp \E\left[\left|\frac{1}{N} \sum_{j=1}^N \xi_{ij} - \E\left[\sum_{j=1}^N \xi_{ij}\right]\right|^4\right]& = & \dsp \frac{1}{N^4}  \sum_{j=1}^N  \E\left[\left|\xi_{ij} - \E\left[\xi_{ij}\right]\right|^4\right]  + \frac{6}{N^4} \sum_{1\leq j <  k \leq N} \V[\xi_{ij}]  \V[\xi_{ik}] \\
 & \leq & \dsp \frac{4M^4}{N^3} + \frac{3 N (N-1)}{N^4}M^4 \leq \frac{4M^4}{N^2}
\end{array}
\]
for $N$ large enough.
This implies that 
\[ \dsp  \PP\croch{{\frac{1}{N}}\sum_{j=1}^N\xi_{ij} - \frac{1}{N}\sum_{j=1}^N\E[\xi_{ij}]   \geq M} \leq \frac{4}{N^2}.
\]
Since the random variables $(\frac{1}{N}\sum_{j=1}^N\xi_{ij})_{i\in\elts}$ are i.i.d.,
\[
\begin{array}{rcl}
 \dsp\PP\croch{\max_{i\in\{1,\cdots,N\}}{\frac{1}{N}}\sum_{j=1}^N\xi_{ij} < \frac{1}{N}\sum_{j=1}^N \E[\xi_{ij}] + M} &  =&  \dsp  \PP\croch{\bigcap_{i=1}^N\left\{{\frac{1}{N}}\sum_{j=1}^N\xi_{ij} < \frac{1}{N}\sum_{j=1}^N \E[\xi_{ij}] + M \right\}}\\
 &  = & \dsp \prod_{i=1}^N \PP\croch{{\frac{1}{N}}\sum_{j=1}^N\xi_{ij} < \frac{1}{N}\sum_{j=1}^N \E[\xi_{ij}] + M} \\
  &  = & \dsp \prod_{i=1}^N\left( 1 -  \PP\croch{{\frac{1}{N}}\sum_{j=1}^N\xi_{ij} \geq  \frac{1}{N}\sum_{j=1}^N \E[\xi_{ij}] + M} \right) \\
 & \geq  & \dsp \left(1- \frac{4}{N^2} \right)^N 
= \exp \left(N \ln\left( 1- \frac{4}{N^2}\right) \right) .
\end{array}
\]
Hence, for $N$ large enough,
\[  
\begin{array}{rcl}
\dsp\PP\croch{\max_{i\in\{1,\cdots,N\}}{\frac{1}{N}}\sum_{j=1}^N\xi_{ij} \geq 2 M}  
& \leq & \dsp \PP\croch{\max_{i\in\{1,\cdots,N\}}{\frac{1}{N}}\sum_{j=1}^N\xi_{ij} \geq \frac{1}{N}\sum_{j=1}^N \E[\xi_{ij}] + M} \\
& \leq & \dsp1 - \exp \left(N \ln\left( 1- \frac{4}{N^2}\right) \right) 
 = \dsp \frac{4}{N} + o\left(\frac{4}{N}\right)\leq \frac{5}{N}.
\end{array}
\]

\end{proof}

In a second technical lemma, we study the time-dependent random variable $Y_i^N(t)$ defined from the solution $u$ to \eqref{eq:graph_limit} for all $i\in\elts$ and $t\in [0,T]$ as $Y_i^N(t) := \sqrt{N}Z_i^N(t)$, where
\begin{equation}\label{eq:Z}
 Z_i^N(t) := \frac{1}{N} \sum_{j=1}^N  \xi_{ij} D(u(X_j,t)-u(X_i,t))  - \int_I   \bw(X_i,y) D(u(y,t)-u(X_i,t))dy.
\end{equation}
Notice that the expected value of $Z_i^N(t)$ satisfies for all $t\in [0,T]$:
\[ 
\E\croch{Z_i^N(t)} = \int_{I^2} \int_{\R_+} \frac{1}{N} \sum_{j=1}^N  w  D(u(y,t)-u(x,t)) q(x,y;dw) dy dx - \int_{I^2}  \bw(x,y) D(u(y,t)-u(x,t))dy dx = 0,
\]
recalling that $\bw(x,y)$ denotes the first moment of $q(x,y;\cdot)$. Hence, it also holds $\E\croch{Y_i^N(t)}=0$.
In the following lemma, we compute the (time-dependent) variance of $Y_i^N(t)$ that we denote by $\sigma_Y^2(t):=\V[Y_i(t)]$, and we estimate the expected value and variance of $(Y_i^N(t))^2$, expressing them as functions of $\sigma_Y^2(t)$.

\begin{lemma}\label{Lemma:sigmaY}
Given a solution $u$ to the integro-differential equation \eqref{eq:graph_limit}, we consider the collection of random variables $(Y^N_i)_{i\in\elts}$, defined for all $i\in\elts$ by $Y^N_i(t):=\sqrt{N} Z^N_i(t)$, where $Z^N_i(t)$  is given by \eqref{eq:Z}.
Then for all $i\in\elts$, its variance satisfies:
$
\V\croch{(Y_i^N)} = \E\croch{(Y_i^N)^2} = \sz^2(t),
$
where
\[
\sz^2(t) := \iint_{I^2} \bw(x,y) D(u(y,t)-u(x,t))^2 dx dy 
-  \int_I \left( \int_I  \bw(x,y) D(u(y,t) - u(x,t)) \right)^2 dx.
\]
For the random variables $(Y^N_i)^2$, it holds for all $i\in\elts$,
\[
\E\croch{(Y_i^N)^4} = 3\sz^4(t) + O(\frac{1}{N}) \qquad \text{ and }  \qquad
\V\croch{(Y_i^N)^2} = 2\sz^4(t) + O(\frac{1}{N}).
\]
\end{lemma}

\begin{proof}
The proof of Lemma \ref{Lemma:sigmaY} can be found in the Appendix \ref{App:1}.
\end{proof}

We are now ready to prove Theorem \ref{th:graph_limit}.

\begin{proof}[Proof of Theorem \ref{th:graph_limit}]
For all $i \in \{1, \dots, N\}$ and $t\in [0,T]$, we denote $\zeta_i^N(t) := u(X_i,t) - u_i^N(t)$ and $\zeta^N(t) = (\zeta^N_1(t), \dots, \zeta^N_N(t))$. 
For conciseness, when context is clear, we omit the explicit time dependence and write $u(x,t)$ as $u(x)$.
We substract \eqref{eq:ODEs} from \eqref{eq:graph_limit} evaluated at $x=X_i$ and obtain 
\begin{equation*}
\begin{split}
 \frac{d}{dt} \zeta_i^N(t)   = & \dsp \int_I   \bw(X_i,y) D(u(y,t)-u(X_i,t))dy - \frac{1}{N} \sum_{j=1}^N  \xi_{ij} D(\ujn(t) - \uin(t)) \\
  = & \dsp \int_I   \bw(X_i,y) D(u(y,t)-u(X_i,t))dy - \frac{1}{N} \sum_{j=1}^N  \xi_{ij} D(u(X_j,t)-u(X_i,t))\\
 & + \frac{1}{N} \sum_{j=1}^N  \xi_{ij} \left[  D(u(X_j,t)-u(X_i,t)) - D(\ujn(t) - \uin(t))\right].
\end{split}
\end{equation*} 
Recognizing the random variable $Z_i^N(t)$ from its definition in \eqref{eq:Z}, 
we multiply by $\dsp \frac{1}{N} \zeta_i^N$ and sum over $i$, which yields:
\begin{equation}\label{eq:diffzeta}
\frac{1}{2} \frac{d}{dt} \|\zeta^N\|_{2,N}^2 = -(Z^N, \zeta^N)_N + \frac{1}{N^2} \sum_{i=1}^N\sum_{j=1}^N \xi_{ij} [D(u(X_j)-u(X_i)) - D(\ujn - \uin)] \zeta^N_i.
\end{equation}
Using the Cauchy-Schwarz inequality for the inner product $(\cdot,\cdot)_N$,
the second term of \eqref{eq:diffzeta} can be bounded above as follows: 
\begin{equation}\label{eq:alphaNgammaN}
\begin{split}
& \left|\frac{1}{N^2} \sum_{i=1}^N\sum_{j=1}^N \xi_{ij} [D(u(X_j)-u(X_i)) - D(\ujn - \uin)] \zeta^N_i \right| \leq \frac{1}{N^2} \sum_{i=1}^N\sum_{j=1}^N \xi_{ij} L \left(|\zeta^N_j|+|\zeta^N_i|\right) |\zeta^N_i| \\
\leq \; & \frac{L}{N^2} \sum_{i=1}^N|\zeta^N_i| \sum_{j=1}^N \xi_{ij} |\zeta^N_j| + \frac{L}{N^2} \sum_{i=1}^N|\zeta^N_i| \sum_{j=1}^N \xi_{ij} |\zeta^N_i| \\
\leq \; & \frac{L}{N} \sum_{i=1}^N|\zeta^N_i| \left( \frac{1}{N}\sum_{j=1}^N \xi_{ij} ^2\right)^{\frac{1}{2}} \left( \frac{1}{N}\sum_{j=1}^N |\zeta^N_j| ^2\right)^{\frac{1}{2}} + \frac{L}{N} \left( \frac{1}{N}\sum_{i=1}^N|\zeta^N_i|^2\right)^{\frac{1}{2}} \left( \frac{1}{N} \sum_{i=1}^N \left( \sum_{j=1}^N \xi_{ij} |\zeta^N_i|\right)^2\right)^{\frac{1}{2}} \\
\leq \; &  L \left( \frac{1}{N}\sum_{i=1}^N |\zeta^N_i| ^2\right) \left( \frac{1}{N}\sum_{i=1}^N \frac{1}{N}\sum_{j=1}^N \xi_{ij} ^2 \right)^{\frac{1}{2}} + \frac{L}{N} \left( \frac{1}{N}\sum_{i=1}^N|\zeta^N_i|^2\right) \left(\max_{i\in\{1,\cdots,N\}} \left(\sum_{j=1}^N \xi_{ij}\right)^2\right)^{\frac{1}{2}} \\
= \; &  L \|\zeta^N\|_{2,N}^2 \left( \frac{1}{N^2}\sum_{i=1}^N \sum_{j=1}^N \xi_{ij} ^2 \right)^{\frac{1}{2}} + \frac{L}{N} \|\zeta^N\|_{2,N}^2 \left(\max_{i\in\{1,\cdots,N\}} \left(\sum_{j=1}^N \xi_{ij}\right)^2\right)^{\frac{1}{2}}\\
= \; &  L \|\zeta^N\|_{2,N}^2 (\alpha_N+\gamma_N), 
\end{split}
\end{equation}
with $ \alpha_N :=  \left( \frac{1}{N^2}\sum_{i=1}^N \sum_{j=1}^N \xi_{ij} ^2 \right)^{\frac{1}{2}}$ and $\gamma_N = \max_{i\in\{1,\cdots,N\}} \frac{1}{N}\sum_{j=1}^N \xi_{ij}.$
Coming back to \eqref{eq:diffzeta}, it holds
\[
\frac{d}{dt} \|\zeta^N\|_{2,N}^2 \leq 2\|Z^N\|_{2,N} \|\zeta^N\|_{2,N}  + 2 L \|\zeta^N\|_{2,N}^2 (\alpha_N+\gamma_N)
\leq  \|Z^N\|_{2,N} ^2 + \|\zeta^N\|_{2,N}^2(1+2 L (\alpha_N+\gamma_N)).
\]
From Gronwall's lemma and using the fact that $\|\zeta^N(0)\|_{2,N}^2=0$, we obtain: 
\begin{equation}\label{eq:zeta}
\|\zeta^N(t)\|_{2,N}^2 \leq \int_0^t \|Z^N(s)\|_{2,N}^2 ds \; e^{(1+2 L (\alpha_N+\gamma_N))t} \leq {T}\sup_{s\in[0,T]} \|Z^N(s)\|_{2,N}^2  \; e^{(1+2 L (\alpha_N+\gamma_N))T} .
\end{equation}
The asymptotic behavior of $\alpha_N$ and $\gamma_N$ was studied in Lemma \ref{Lemma:alphagamma}.
We now focus on the asymptotic behavior of $Z^N$.

Denonting for all $i\in\elts$ $\sigma^2_i(t):=\V[Y_i^N(t)]$, notice from Lemma \ref{Lemma:sigmaY} that
for all $t\in[0,T]$,
\[
\begin{split}
\sigma^2_i(t) & = \sigma_Y^2(t) \\
& = \iint_{I^2}\bw(x,y) D(u(y,t)-u(x,t))^2  dx dy - \int_I \left( \int_{\R_+} \int_I \bw(x,y) D(u(y,t) - u(x,t)) q(x,y;dw)dy \right)^2 dx \\
 & \leq \; c_1 := M^2 K^2, 
\end{split}
\]
where $K$ and $M$ are the constants of Hyp.~\ref{hyp:D} and \ref{hypo_moment}.
Therefore, applying again the Bienaymé-Chebyshev inequality, it holds
$$\begin{array}{rcl}
\mathbb{P}\left(\sqrt{N} \|Z^N(t) \|_{2,N} e^{(\frac{1}{2}+4ML)T} \geq \sqrt{1 + c_1}e^{(\frac{1}{2}+4ML)T}\right) & = &  \dsp \mathbb{P}(N \|Z^N(t) \|_{2,N}^2 \geq 1 +c_1) \\
 & \leq &\dsp \mathbb{P}\left(N \|Z^N(t) \|_{2,N}^2 \geq 1 +\frac{1}{N}\sum_{i=1}^N (\sigma_i(t))^2\right)  \\
  & \leq & \dsp \mathbb{P}\left( \left|N \|Z^N(t) \|_{2,N}^2 - \frac{1}{N}\sum_{i=1}^N (\sigma_i(t))^2 \right| \geq 1 \right) \\
  &=  & \dsp \mathbb{P}\left( \left| \frac{1}{N} \sum_{i=1}^N (Y_i^N(t))^2 - \E\left[ \frac{1}{N}  \sum_{i=1}^N (Y_i^N(t))^2  \right]  \right| \geq 1 \right) \\
  & \leq & \dsp \frac{1}{N^2} \sum_{i=1}^N  \V\left[ (Y_i^N(t))^2\right]  \\
  & \leq & \dsp
 \frac{1}{N^2}  \sum_{i=1}^N (2 \sz^4(t)+O(\frac{1}{N})),
 \end{array}$$
where we used Lemma \ref{Lemma:sigmaY} for the last inequality. \\
Thus, for $N$ large enough,  
 $$\begin{array}{rcl}
\dsp \mathbb{P}\left(\sqrt{N} \|Z^N(t) \|_{2,N} e^{(\frac{1}{2}+4ML)T} \geq \sqrt{1 + c_1}e^{(\frac{1}{2}+4ML)T)}\right) & \leq &\dsp  \frac{3 }{N}M^4 K^4.
 \end{array}$$ 
Hence, for all $t \in [0,T]$,
\[\begin{split}
& \mathbb{P}\croch{\sqrt{N} \|Z^N(t) \|_{2,N} e^{(\frac{1}{2}+L(\gamma_N+\alpha_N))T} \geq \sqrt{1 + c_1}e^{(\frac{1}{2}+4ML)T}} \\
= \;&  \mathbb{P}\croch{\left\{\sqrt{N} \|Z^N(t) \|_{2,N} e^{(\frac{1}{2}+L\gamma_N)T} \geq \sqrt{1 + c_1}e^{(\frac{1}{2}+4ML)T} e^{-L\alpha_N T}\right\} \bigcap \left\{ \alpha_N < 2M \right\}}\\
&  ~~~~~~ +  \mathbb{P}\croch{\left\{\sqrt{N} \|Z^N(t) \|_{2,N} e^{(\frac{1}{2}+L\gamma_N)T} \geq \sqrt{1 + c_1}e^{(\frac{1}{2}+4ML)T)} e^{-L\alpha_N T}\right\} \bigcap \left\{ \alpha_N \geq 2M \right\}} \\
 \leq \; &   \mathbb{P}\croch{\left\{\sqrt{N} \|Z^N(t) \|_{2,N} e^{(\frac{1}{2}+L\gamma_N)T} \geq \sqrt{1 + c_1}e^{(\frac{1}{2}+2ML)T)}\right\} \right] +  \mathbb{P}\left[ \alpha_N \geq 2M } \\
=  \; &\mathbb{P}\croch{\left\{\sqrt{N} \|Z^N(t) \|_{2,N} e^{\frac{1}{2}T} \geq \sqrt{1 + c_1}e^{(\frac{1}{2}+2ML)T)} e^{-L\gamma_N T}\right\} \bigcap \left\{ \gamma_N < 2M \right\}}\\
&  ~~~~~~+   \mathbb{P}\croch{\left\{\sqrt{N} \|Z^N(t) \|_{2,N} e^{\frac{1}{2}T} \geq \sqrt{1 + c_1}e^{(\frac{1}{2}+2ML)T)} e^{-L\gamma_N T}\right\} \bigcap \left\{ \gamma_N \geq 2M \right\} } 
 +  \mathbb{P}\croch{ \alpha_N \geq 2M } \\
   \leq \; &  \mathbb{P}\croch{\sqrt{N} \|Z^N(t) \|_{2,N} e^{\frac{1}{2}T} \geq \sqrt{1 + c_1}e^{\frac{1}{2}T} } +   \mathbb{P} \croch{\gamma_N \geq 2M} +  \mathbb{P}\croch{ \alpha_N \geq 2M } \\
  \leq  \; & \frac{3 M^4 K^4}{N} + \frac{1}{N^2} + \frac{5}{N} \leq \frac{\tilde{C}_1}{N}
 \end{split}
 \]
for $N$ large enough, where the penultimate inequality comes from Lemma \ref{Lemma:alphagamma}, and where $\tilde{C}_1 := 3 M^4 K^4+ 6$.
Recaling from \eqref{eq:zeta} that 
$\|\zeta^N(t)\|_{2,N} \leq \sqrt{T}\sup_{s\in[0,T]} \|Z^N(s)\|_{2,N}  \; e^{(\frac{1}{2}+ L (\alpha_N+\gamma_N))T}$, we obtain:
\[
\mathbb{P}\croch{\|\zeta^N(t)\|_{2,N} \leq \frac{C_1(T)}{\sqrt{N}} } \leq \frac{\tilde{C}_1}{N},
\]
where $C_1(T):= \sqrt{T} \sqrt{1 + M^2 K^2 }e^{(\frac{1}{2}+4ML)T}$ uniformly for $t \in [0,T]$, which concludes the proof by continuity in $t$.
\end{proof}


\section{Networks on weighted random graphs generated by deterministic sequences} \label{Sec:Deter}
{In this section, we are now interested in weighted random graphs generated by deterministic sequences. 
Denoting for all $i\in\elts$, $I_i^N:= [\frac{i-1}{N}, \frac{i}{N})$, we consider the deterministic sequence $X^N= \{ x_1^N, \dots , x_N^N\}$, where for all $i \in \{1, \dots, N \}$, $x_i^N \in I_i^N$. 

\begin{definition}
Let
$$
\begin{array}{llll}
q :&I \times I &  \to & \mathcal{P}(\mathbb{R_+})\\
 & (x,y) &\mapsto & q(x,y;\cdot)
\end{array}$$
be a weighted random graph law.
A \textbf{$q$-weighted random graph} on $N$ nodes \textbf{generated by the deterministic sequence $X^N$}, denoted $\overline{G}_N$, is such that the weight of each edge $(i,j), i \neq j,$ of $\overline{G}_N$ is randomly attributed, and its law is $q(x_i^N,x_j^N;\cdot)$. \\
{The decision of the attribution of the weight of a pair $(i,j) \in \{1, \dots, N\}^2, i \neq j,$ is made independently from the decision for other pairs.} 
\end{definition}
}

Thus, in this part, we consider the system of differential equation 
\begin{equation}     
\tag{$\mathcal{S}_N^\mathrm{r-d}$}
\label{ODE22}
\begin{cases}     
\displaystyle  \frac{d}{dt} \uin(t)= \frac{1}{N} \sum_{j=1}^N \xi_{ij} D(\ujn - \uin),~~~~ i \in \{1, \dots, N\}\\
\uin(0) = g(x_i^N),~~~~   i \in \{1, \dots, N\}
\end{cases}
\end{equation}
where  $\mathcal{L}(\xi_{ij}| {X}^N) = q(x_i^N,x_j^N,\cdot)$. 
As opposed to Section \ref{Sec:Random}, system \eqref{ODE22} is only simply random: the weights $\xi_{ij}$ are \textit{randomly} attributed from the distribution $q(x_i^N,x_j^N,\cdot)$, where $(x_i^N,x_j^N)$ are \textit{deterministic}.
        
As in Section \ref{Sec:Random}, we will prove that the solution to \eqref{ODE22} converges as $N$ goes to infinity to the solution to the Graph Limit equation 
\begin{equation}
\begin{cases}\tag{C}
\partial_t u(x,t)  = \dsp \int_I \left( \int_{\mathbb{R_+}}  w q(x,y;dw)\right) D(u(y,t)-u(x,t))dy\\
u(x,0)   =  g(x),  ~~~x \in I,
\end{cases}
\end{equation}
in a sense to be determined.

From the particle system $(u_i^N(t))_{i\in\elts}$, we construct a piecewise-constant bounded function $u_N\in L^\infty(I\times\R)$ defined by 
\begin{equation} \label{def:un}
\forall t \in \mathbb{R},\; \forall x\in I, \qquad u_N(x,t) = \sum_{i=1}^N u_i^N(t) \one_{I^N_i}(x),
\end{equation}
where $\one_{I^N_i}$ denotes the indicator function of the interval $I^N_i$.
We can then proove the following:
\begin{theorem}
\label{th:deterministic}
Let $D$ satisfy Hyp. \ref{hyp:D}, and let $g \in C^{0,\frac{1}{2}}(I)$. Suppose that the  weighted random graph law satisfies Hypothesis \ref{hypo_moment} and that $(x,y) \mapsto \int_{\R_+} w q(x,y;dw)$ is $\frac{1}{2}-$Hölder on $I^2$.
Then, the function $u_N$ dedfined in \eqref{def:un}  converges as $N$ goes to infinity to the solution $u$ of the continuous model \eqref{eq:graph_limit}. More precisely, 
\begin{equation}\label{eq:conv2}
 \mathbb{P} \croch{ \|u_N - u \|_{\mathcal{C}(0,T;L^2(I))} \geq \frac{C_2}{\sqrt{N}}} \leq \frac{\tilde{C}_2}{N}
\end{equation}
for some explicit constants $C_2, \tilde{C}_2>0$.
\end{theorem}
        
        We start by introducing an auxiliary problem, the heat equation on the following deterministic weighted graph $\tilde{G}_N$ : for all $(i,j)\in\elts^2$, the edge $(i,j)$ of $\tilde{G}_N$ is supplied with the weight 
        $$W_{ij}^N = \E\croch{\xi_{ij}} = \bw(x_i^N, x_j^N) = \int_{\mathbb{R}_+} w q(x_i^N, x_j^N;dw), $$
         and we study the following problem
        \begin{equation}
                \begin{cases}     
        \label{ODE2}       
        \displaystyle \frac{d}{dt} \vin(t)= \frac{1}{N} \sum_{j=1}^N \bw(x_i^N, x_j^N) D(\vjn(t) - \vin(t)),~~~~ i \in \{1, \dots, N\}\\                
       \vin(0)  =g(x_i^N),~~~~   i \in \{1, \dots, N\}.
        \end{cases}
        \end{equation}
        
Let us denote $v^N(t) = (v_1^N(t), \dots, v_N^N(t))$ the solution of \eqref{ODE2}. Let $v_N:I\times\R\rightarrow\R$ be the piecewise-constant function defined by
\[
\forall t \in \R,  \; \forall x\in I, \qquad v_N(x,t)= \sum_{i=1}^N v_i^N(t) \one_{I^N_i}(x) .
\]
In the same spirit, we define a piecewise-constant function $W_N$ on $I^2$ such that 
\[ 
\; \forall (x,y)\in I^2,\quad W_N(x,y) = \sum_{i=1}^N \sum_{i=1}^N W_{ij}^N \one_{I^N_i}(x)\one_{I^N_j}(y) . 
\]      
Then, by construction, $v_N(x,t)$ solves the following system
 \begin{equation}
\begin{cases}
\label{graph_limit2}
\partial_t v_N(x,t)  = \dsp \int_I W_N(x,y)   D(v_N(y,t)-v_N(x,t))dy\\
v_N(\cdot,0)   =  \dsp \sum_{i=1}^N g(x_i^N) \one_{I^N_i}.
 \end{cases}
\end{equation}

The convergence of the deterministic system \eqref{ODE2} can then be obtained as a direct consequence of Theorem 4 in \cite{PT23}, with additional regularity assumptions on $g$, and the first moment of the weighted random graph law, $\bw$.

\begin{theorem}{\cite{PT23}}\label{th:PT}
Suppose that $D$ is a Lipschitz continuous function satisfying Hyp. \ref{hyp:D}. Moreover, suppose that there exists $\alpha\in (0,1]$ such that $g\in C^{0,\alpha}(I)$ and that the first moment $(x,y)\mapsto\bw(x,y)$ of the weighted random graph law is also $\alpha-$Hölder with respect to $x$ and $y$.
Then, the solution $v_N$ to \eqref{graph_limit2} converges to the solution $u$ to \eqref{eq:graph_limit} as $N$ goes to $+\infty$, and
\begin{equation}
\|v_N(\cdot,t)-u(\cdot,t)\|_{L^\infty(I)} \leq \frac{2}{N^{\alpha}}(1+H(g)) e^{2tL_G},
\end{equation}
where $H(g)$ denotes the Hölder constant of $g$, $H(\bw)$ the Hölder constant of $\bw$, and $L_G:= \max(H(\bw),L)$.
\end{theorem}


Using this intermediate result, we can now prove our main Theorem \ref{th:deterministic} by multiple uses of the Bienaymé-Chebyshev inequality (or of its generalization).

\begin{proof}[Proof of Theorem \ref{th:deterministic}]
For all $i\in\elts$, we denote $\eta_i^N(t) = u_i^N(t) - v_i^N(t)$, and $\eta^N(t) = (\eta_1^N(t), \dots, \eta_N^N(t))$. We substract \eqref{ODE22} and \eqref{ODE2} and obtain
\[
\begin{split}
\dsp \frac{d}{dt} \eta_i^N(t)  & = \dsp  \frac{1}{N}  \pare{\sum_{j=1}^N \xi_{ij} D(u_j^N - u_i^N) - \sum_{j=1}^N  W_{ij}^N D(\vjn - \vin)} \\
&=\dsp{\frac{1}{N}   \sum_{j=1}^N \xi_{ij} \pare{D(u_j^N - u_i^N) -  D(\vjn - \vin)}} + \frac{1}{N}   \sum_{j=1}^N (\xi_{ij}- W_{ij}^N) D(v_j^N - v_i^N) .
\end{split}
\]

We multiply the previous equation by $\frac{1}{N} \eta_i^N$ and we sum over $i$ to obtain 
\begin{equation}
\frac{1}{2} \frac{d}{dt} \| \eta^N\|_{2,N}^2 = \frac{1}{N^2} \sum_{i,j=1}^N \xi_{ij} \pare{D(\ujn - \uin) - D(\vjn -\vin)} \eta_i^N + (Z^N, \eta^N)_N, 
\end{equation}
where we denote $\tZ^N= (\tZ_1^N, \dots, \tZ_N^N)$ with $\dsp \tZ_i^N := \frac{1}{N}   \sum_{j=1}^N \pare{ \xi_{ij}  - W_{ij}^N} D(\vjn - \vin).$\\
Let us deal with the first term. As in the proof of Theorem \ref{th:graph_limit}, 
$$\begin{array}{rcl}
\dsp \left| \frac{1}{N^2}  \sum_{i=1}^N  \pare{ \sum_{j=1}^N \xi_{ij} \pare{D(u_j^N - u_i^N) -  D(\vjn - \vin)} } \eta_i^N\right|& \leq & \dsp \frac{ML}{N^2} \sum_{i,j=1}^N  |\eta_j^N(t) - \eta_i^N(t)| |\eta_i^N(t)| \\
& \leq &\dsp  L(\talpha+\tgamma ) \| \eta^N(t)\|_{2,N}^2,
\end{array}$$
with $ \talpha :=  \left( \frac{1}{N^2}\sum_{i=1}^N \sum_{j=1}^N \xi_{ij} ^2 \right)^{\frac{1}{2}}$ and $\tgamma = \max_{i\in\{1,\cdots,N\}} \frac{1}{N}\sum_{j=1}^N \xi_{ij}.$

From the previous inequalities, we deduce that 
\[
 \dsp \frac{d}{dt} \| \eta^N\|_{2,N}^2 \leq  \| \tZ^N\|_{2,N} ^2 + (1 + 2L(\talpha+\tgamma)) \| \eta^N\|_{2,N}^2, 
\]
and Gronwall's lemma yields
\begin{equation}
\label{ineg_gronwall2}
\| \eta^N(t)\|_{2,N}^2 \leq T \sup_{t\in [0,T]} \| {\tZ}^N(t)\|_{2,N}^2 e^{ (1 + 2L(\talpha+\tgamma))T}.
\end{equation}
As in the proof of Theorem \eqref{th:graph_limit}, we aim to use the Bienaymé-Chebyshev inequality to bound  $\tilde{\alpha}_N$ and $\tilde{\gamma}_N$.  Notice however the key difference introduced by the deterministic choice of the points $(x_i^N)_{i\in \elts}$. The random variables $({\xi_{ij}^2})_{i,j\in\elts}$ no longer have the same law, and for all $i,j\in\elts$,
\[
{\E\croch{\xi_{ij}^2} = \int_I w^2 q(x_i^N,x_j^N,dw), \qquad  \V\croch{\xi_{ij}^2} = \int_I w^4 q(x_i^N,x_j^N,dw) - \left(\int_I w^2 q(x_i^N,x_j^N,dw)\right)^2.}
\]
However, using the uniform bounds on the moments of $q$ (Hyp. \ref{hypo_moment}), we can follow the same arguments as in Lemma \ref{Lemma:alphagamma}, and prove
\[
\dsp\PP\croch{\tilde{\alpha}_N \geq 2 M}  \leq \frac{1}{N^2}, \qquad \text{ and }\qquad
 \PP\croch{\tilde{\gamma}_N \geq 2 M}  \leq \frac{5}{N}.
\]
This allows us to obtain for $N$ large enough, as in the proof of Theorem \ref{th:graph_limit}:
\[
   \mathbb{P}\left(\sqrt{N} \|\tZ^N(t) \|_{2,N} e^{(1+2L(\tilde{\gamma}_N+\tilde{\alpha}_N))T} \geq \sqrt{1 + c_1}e^{(\frac{1}{2}+4ML)T)}\right)  \leq \frac{\tilde{C}_1}{N},
\]
where $c_1=M^2K^2$ and $\tilde{C}_1 = 3M^4K^4+6$
(see Appendix \ref{App:2} for some details of the computations).
Moreover, it holds
$$\begin{array}{rcl}
\dsp \| u_N - u \|_{\mathcal{C}(0,T;L^2(I))} & = & \dsp \sup_{t \in [0,T]} \pare{\int_I |u_N(t,x) - u(t,x)|^2 dx} ^{1/2}\\
 & \leq &  \dsp \sup_{t \in [0,T]} \| u_N - v_N \|_{L^2(I)} +  \sup_{t \in [0,T]} \| v_N - u \|_{L^2(I)}\\
  & = & \dsp \sup_{t \in [0,T]} \pare{\sum_{i=1}^N \int_{I_i} | u_i^N(t) - v_i^N(t)|^2 dx}^{1/2} + \| v_N - u \|_{\mathcal{C}(0,T;L^2(I))}\\
   & =& \dsp \sup_{t \in [0,T]}  \| \eta^N \|_{2,N} + \|v_N - u \|_{\mathcal{C}(0,T;L^2(I))}.
\end{array}$$
The second term is deterministic, and from Theorem \ref{th:PT}, converges to zero with a rate of $N^{-\frac{1}{2}}$:
More precisely, denoting $c_4 :=2(1+H(g)) e^{2TL_G}$, it holds
\[ \PP \croch{ \|v_N-u\|_{\mathcal{C}(0,T;L^2(I))} \geq \frac{c_4}{\sqrt{N}}} = 0.\]
Finally, denoting $\dsp C_2 := \max \left(2c_4, 2 \sqrt{T}\sqrt{1+ c_1} e^{(\frac{1}{2}+4ML)T} \right)$, we obtain 
\[\begin{array}{l}
\dsp  \PP \croch{ \|u_N-u\|_{\mathcal{C}(0,T;L^2(I))} \geq \frac{C_2}{\sqrt{N}}} \\
 \leq  \dsp \PP \croch{\left\{ \sup_{t \in [0,T]} \|\eta\|_{2,N} \geq \frac{C_2}{2\sqrt{N}}\right\}\bigcup \left\{ \|v_N-u\|_{\mathcal{C}(0,T;L^2(I))} \geq \frac{C_2}{2\sqrt{N}}\right\}} \\
  \leq  \dsp \PP \croch{ \sup_{t \in [0,T]} \|\eta\|_{2,N} \geq \frac{C_2}{2\sqrt{N}}} + \PP \croch{  \|v_N-u\|_{\mathcal{C}(0,T;L^2(I))} \geq \frac{C_2}{2\sqrt{N}}} \\
   \leq \dsp \PP \croch{ \sup_{t \in [0,T]} \|\tZ^N\|_{2,N} \geq \sqrt{T}\frac{\sqrt{1+ c_1}}{\sqrt{N}} e^{(\frac{1}{2}+4ML)T}} + \PP \croch{  \|v_N-u\|_{\mathcal{C}(0,T;L^2(I))} \geq \frac{C_2}{2\sqrt{N}}} \\
   \leq \dsp \frac{\tilde{C}_1}{N},
\end{array}
\]
 which concludes the proof.
\end{proof}


\section{Blinking systems on weighted random graphs}\label{Sec:blinking}

In what we have done previously, the weighted random graph is fixed at $t=0$ and stays constant with time, even if the edges' weights are randomly chosen. In this section, we will now be interested in time-dependent random graphs. 

We will focus on blinking systems in the context of a weighted random graph generated by a random sequence. As in Section \ref{Sec:Random}, let $(X_i)_{i\in\N}$ be a sequence of i.i.d. random variables uniformly distributed on $I$. \\
We now consider a time-dependent piecewise-constant random variable $\xij(t)$ defined as follows : for all $k \in \N$, for all $t \in [k, k+1)$, $\xij(t)=\xij^k$ with $\mathcal{L}(\xi_{ij}^k| \tilde{X}) = q(X_i,X_j,\cdot)$. We then study the following blinking system : given $T>0$, $n \in \N^*$, and $ \varepsilon = \frac{T}{n}$,
 \begin{equation}\label{eq:ODEbl}
 \tag{$S_{N,\varepsilon}$}
\begin{cases}
\displaystyle \frac{d}{dt} \uine(t) = \frac{1}{N} \sum_{j=1}^N  \xij\left(\frac{t}{\varepsilon}\right)  D(\ujne(t) - \uine(t)),~~~~ i \in \{1, \dots, N\}\\ 
\uine(0)  =g(X_i^N),~~~~   i \in \{1, \dots, N\}
\end{cases}
\end{equation}
Hence, given $\varepsilon >0$, the graph associated with \eqref{eq:ODEbl} is redefined on each interval $[k\varepsilon, (k+1) \varepsilon)$, $k \in \{0, \dots, n-1\}$.
Applying results from Averaging theory (see \cite[Section 3.2]{Skorokhod2002}), we can show that as $\varepsilon \to 0$, the time-dependent system \eqref{eq:ODEbl} converges to the averaged problem 
 \begin{equation}\label{eq:ODEav}
\tag{$S_{N,\text{Av}}$}
\begin{cases}
\displaystyle \frac{d}{dt} \uina(t) = \frac{1}{N} \sum_{j=1}^N  \left(\int_{\R_+} w q(X_i,X_j;dw)\right) D(\ujna(t) - \uina(t)),~~~~ i \in \{1, \dots, N\}\\ 
\uina(0)  =g(X_i^N),~~~~   i \in \{1, \dots, N\},
\end{cases}
\end{equation}
where $\int_{\R_+} w q(X_i,X_j;dw) = \E\croch{\xi_{ij}^k| X_i,X_j}$ for all $k\in\{1,\cdots, n-1\}$.
The exact statement writes 
\begin{prop}\label{Prop:limepsilon}
Let $g\in L^\infty(I)$, $D$ satisfying Hyp.\ref{hyp:D}. Let $N\in\N$, $(X_i)_{i\in\elts}$ be a sequence of i.i.d. random variables, and $u^{N,\varepsilon}$ be the solution to \eqref{eq:ODEbl}. Then as $\varepsilon$ goes to zero, it satisfies
\[\mathbb{P} \left\{\lim_{\varepsilon \to 0} \sup_{s \leq T} \left| u^{N,\varepsilon}(s) - u^{N,\text{Av}}(s)\right| = 0\right\} = 1,\]
where  $u^{N,\text{Av}}$ is the solution to \eqref{eq:ODEav}.
\end{prop}
\begin{proof}
The proof of Prop. \ref{Prop:limepsilon} is a direct consequence of the averaging results presented in \cite{Skorokhod2002}.
\end{proof}

\begin{rem}
When directly applying the results of averaging theory from \cite[Section 3.2]{Skorokhod2002} , it does not provide a rate for the convergence of \eqref{eq:ODEbl} to \eqref{eq:ODEav}. A rate can be achieved in the case of blinking systems taking a finite number of values (see \cite{belykh_rate}).
\end{rem}

Notice that the limit system \eqref{eq:ODEav} is no longer defined on a time-varying network. 
Moreover, it would be deterministic if the sequence $(X_i)_{i\in\elts}$ were deterministic.
Thus, so far we have considered four models : 
\begin{itemize}
\item the system on a fixed weighted random graph \eqref{eq:ODEs} 
\item its limit as $N$ goes to infinity, i.e. the graph limit equation \eqref{eq:graph_limit}
\item the blinking system \eqref{eq:ODEbl} 
\item its limit as $\varepsilon$ goes to zero, i.e. the associated averaged system \eqref{eq:ODEav}.
\end{itemize}
Naturally, we are interested in the relation between them. 
Interestingly, we can prove that \eqref{eq:ODEav} converges as $N \to + \infty$ to the same limit \eqref{eq:graph_limit} as \eqref{eq:ODEs} by seeing  \eqref{eq:ODEav}  as a specific case of \eqref{eq:ODEs}  where $q(x,y;dw) = \delta (w=\bw(x,y))$. 
The known relations between all these various systems are summarized in Figure \ref{fig:Ergodicity?}. The question of linking \eqref{eq:ODEbl} to \eqref{eq:graph_limit} by taking first the limit in $N$ remains open, and the goal of this section is to answer it.

\begin{figure}[h!]
\centering
\includegraphics[scale=1]{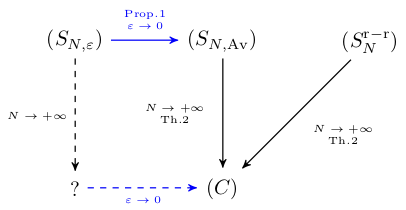}
\caption{Existing and missing links between systems \eqref{eq:ODEbl}, \eqref{eq:ODEav}, \eqref{eq:ODEs} and \eqref{eq:graph_limit}}
\label{fig:Ergodicity?}
\end{figure}

To study the convergence of \eqref{eq:ODEbl} as $N$ goes to infinity, a natural approach would be to apply the convergence method from \eqref{eq:ODEs} to \eqref{eq:graph_limit} on each time interval $[k\varepsilon, (k+1) \varepsilon)$.
On the first interval $[0,\varepsilon)$, the convergence \eqref{eq:ODEbl} to \eqref{eq:graph_limit} as $N$ goes to infinity is indeed a direct application of Theorem \ref{th:graph_limit}. However, in subsequent intervals, the initial condition $\uine(k\varepsilon)$ does not match $u(k\varepsilon,X_i^N)$, as required to apply Theorem \ref{th:graph_limit} again. 
Thus, our approach consists in resetting the initial conditions at the beginning of each time interval, hoping that the error that we commit by doing so is small enough. 
In that purpose, we introduce the following intermediate system : $\forall t \in (k\varepsilon, (k+1)\varepsilon]$,
 \begin{equation}\label{eq:ODENek}
\tag{$\tilde{{S}}_{N,\varepsilon,k}$}
\begin{cases}
\displaystyle \frac{d}{dt} \uinei(t) = \frac{1}{N} \sum_{j=1}^N  \xij\left(\frac{t}{\varepsilon}\right)  D(\ujnei(t) - \uinei(t)),~~~~ i \in \{1, \dots, N\}\\ 
\uinei(k\varepsilon)  =u(k\varepsilon,X_i^N),~~~~   i \in \{1, \dots, N\}.
\end{cases}
\end{equation}
We consider that on each time interval $[k\varepsilon, (k+1)\varepsilon)$, the systems  \eqref{eq:ODEbl} and  \eqref{eq:ODENek} are constructed using the same random variables. 
We can then show  that the solutions to the two systems are indeed close, in the following sense: 

\begin{lemma}\label{gronwall}
For all $k \in \{0, \dots, n-1\}$, for all $t \in [k\varepsilon, (k+1)\varepsilon)$, the solutions $u^{N,\varepsilon}$ and $\tilde{u}^{N,\varepsilon,k}$ to \eqref{eq:ODEbl} and \eqref{eq:ODENek} satisfy:
\[
\PP\croch{  \|u^{N,\varepsilon}(t)-\tilde{u}^{N,\varepsilon,k}(t)\|_{2,N} \leq e^{ 4L M \varepsilon} \| u^{N,\varepsilon}(k\varepsilon)-\tilde{u}^{N,\varepsilon,k}(k\varepsilon) \|_{2,N} } \geq 1 - \frac{6}{N},
\]
where $L$ is the  Lipschitz constant of $D$ and $M$ is the uniform bound on the moments of $q$, as defined in \eqref{eq:moments}. 
\end{lemma}

\begin{proof}
Let $k \in \{0, \dots, n-1\}$ and $t \in [k\varepsilon, (k+1)\varepsilon)$. Let $U_N(t):=u^{N,\varepsilon}-\tilde{u}^{N,\varepsilon,k}$.
Since $u^{N,\varepsilon}$ and $\tilde{u}^{N,\varepsilon,k}$ satisfy the same differential equation, following the computation \eqref{eq:alphaNgammaN} in the proof of Theorem \ref{th:graph_limit}, it holds with probability 1:
\[
\frac{d}{dt} \| u^{N,\varepsilon}(t)-\tilde{u}^{N,\varepsilon,k}(t) \|_{2,N}^2 
\leq 2L (\alpha^k_N+\gamma^k_N)  \| u^{N,\varepsilon}(t)-\tilde{u}^{N,\varepsilon,k}(t) \|_{2,N}^2 ,
\]
where $ \alpha^k_N := \displaystyle \left( \frac{1}{N^2}\sum_{i=1}^N \sum_{j=1}^N (\xi^k_{ij}) ^2 \right)^{\frac{1}{2}}$ and $\displaystyle \gamma^k_N = \max_{i\in\{1,\cdots,N\}} \frac{1}{N}\sum_{j=1}^N \xi^k_{ij}.$
From Gronwall's lemma, for all $t \in [k\varepsilon, (k+1)\varepsilon)$,
\[
\| u^{N,\varepsilon}(t)-\tilde{u}^{N,\varepsilon,k}(t) \|_{2,N}^2 
\leq e^{2 L (\alpha^k_N+\gamma^k_N) \varepsilon} \| u^{N,\varepsilon}(k\varepsilon)-\tilde{u}^{N,\varepsilon,k}(k\varepsilon) \|_{2,N}^2.
\]
Since $\alpha^k_N+\gamma^k_N\leq 4M$ implies 
\[
\| u^{N,\varepsilon}(t)-\tilde{u}^{N,\varepsilon,k}(t) \|_{2,N}^2 
\leq e^{ 8LM \varepsilon} \| u^{N,\varepsilon}(k\varepsilon)-\tilde{u}^{N,\varepsilon,k}(k\varepsilon) \|_{2,N}^2,
\]
it holds
\[
\PP\croch{  \|u^{N,\varepsilon}(t)-\tilde{u}^{N,\varepsilon,k}(t)\|_{2,N}^2 \leq e^{ 8L M \varepsilon} \| u^{N,\varepsilon}(k\varepsilon)-\tilde{u}^{N,\varepsilon,k}(k\varepsilon) \|_{2,N}^2 }
\geq 
\PP\croch{\alpha^k_N+\gamma^k_N\leq 4M}.
\]
Moreover, 
\[
\PP\croch{\alpha^k_N+\gamma^k_N > 4M} \leq \PP\croch{\{\alpha^k_N >2M\}\cup \{ \gamma^k_N > 2M\} } \leq \PP\croch{\alpha^k_N >2M} +  \PP\croch{ \gamma^k_N > 2M} \leq \frac{6}{N}
\]
from Lemma \ref{Lemma:alphagamma}.
Thus, 
\[
\PP\croch{  \|u^{N,\varepsilon}(t)-\tilde{u}^{N,\varepsilon,k}(t)\|_{2,N}^2 \leq e^{ 8L M \varepsilon} \| u^{N,\varepsilon}(k\varepsilon)-\tilde{u}^{N,\varepsilon,k}(k\varepsilon) \|_{2,N}^2 }
\geq 
\PP\croch{\alpha^k_N+\gamma^k_N\leq 4M}
\geq 
1-\frac{6}{N}.
\]
\end{proof}

Using this intermediate system, we are able to prove the main result of this Section - as $N\varepsilon$ converges to infinity, the solution to \eqref{eq:ODEbl} converges to the solution to \eqref{eq:graph_limit}, in the following sense:

\begin{theorem}\label{th:conv_blinking}
Let $T>0$, $\varepsilon>0$ be given. Let $X =(X_i)_{i\in\N}$ be a sequence of i.i.d. random variables and for all $N \in \N$, let $X^N =(X_i)_{1 \leq i \leq N}$. Let $\xij(t) = \xij^k$ for all $t \in [k\varepsilon, (k+1)\varepsilon)$, $k \in \{0, \dots, n-1\}$ where $\mathcal{L}(\xi_{ij}^k| {X^N}) = q(X_i,X_j,\cdot)$.
Let $u^{N,\varepsilon}$ be the solution to \eqref{eq:ODEbl} and let $u$ be the solution to \eqref{eq:graph_limit}.
Then, 
\[
\mathbb{P}\croch{ \sup_{t \in [0,T]} \| u^{N,\varepsilon}(t) - P_{X^N}u(t,\cdot) \|_{2,N} \leq \frac{C_3(T)}{\sqrt{ N \varepsilon}} }\geq 1- \frac{\tilde{C}_3(T)}{N\varepsilon}.
\]
where 
$C_3(T):= \sqrt{1+M^2 K^2}e^{(\frac{1}{2}+4ML)T}\frac{e^{4MLT}-1}{4ML}$
and
$\tilde{C}_3(T) := (12+3M^4 K^4)T$ .
\end{theorem}

\begin{proof}[Proof of Theorem \ref{th:conv_blinking}]
Let $u^{N,\varepsilon}$ be the solution to \eqref{eq:ODEbl} and $\tilde{u}^{N,\varepsilon,k}$ the solution to \eqref{eq:ODEav} on each interval $[k\varepsilon,(k+1)\varepsilon)$.
We start by proving by induction that 
\begin{equation}\label{prop_rec}
\mathbb{P} \croch{ \sup_{t \in [k\varepsilon, (k+1)\varepsilon]} \| u^{N,\varepsilon}(t) - P_{X^N}u(t,\cdot) \|_{2,N} > \frac{C_1(\varepsilon)}{\sqrt{N}} \sum_{\ell=0}^k e^{4LM\varepsilon \ell}} \leq \frac{(k+1)}{N}(6+\tilde{C}_1),
\end{equation}
where $C_1$ and $\tilde{C}_1$ are the constants that appear in Theorem \ref{th:graph_limit}.
For $k=0$, 
Theorem \ref{th:graph_limit} implies that 
$$\mathbb{P} \left( \sup_{t \in [0,\varepsilon]} \| u^{N,\varepsilon}(t) - P_{X^N}u(t,\cdot) \|_{2,N} \geq \frac{C_1(\varepsilon) }{\sqrt{N}} \right) \leq \frac{\tilde{C}_1}{N}\leq \frac{\tilde{C}_1+6}{N}.$$
Suppose that \eqref{prop_rec} holds for some $k \in \{0, \dots, n-2\} $. 
We begin by taking $t$ in the half-open interval $[(k+1)\varepsilon, (k+2) \varepsilon)$. Denoting $U_N(t):=\left\|u^{N,\varepsilon}(t) - \tilde{u}^{N,\varepsilon,k+1}(t) \right\|_{2,N}$, it holds
\begin{equation}\label{eq:blink1}
\begin{split}  
& \mathbb{P}\croch{\| u^{N,\varepsilon}(t) - P_{X^N}u(t,\cdot) \|_{2,N} \geq \frac{C_1(\varepsilon) }{\sqrt{N}} \sum_{\ell=0}^{k+1} e^{4LM\varepsilon \ell}} \\
\leq & \;
\mathbb{P}\croch{U_N(t) +\|  \tilde{u}^{N,\varepsilon,k+1}(t)- P_{X^N}u(t,\cdot) \|_{2,N} \geq \frac{C_1(\varepsilon) }{\sqrt{N}} \sum_{\ell=0}^{k+1} e^{4LM\varepsilon \ell}}  \\
\leq & \;
\mathbb{P}\croch{\{ U_N(t)\geq \frac{C_1(\varepsilon) }{\sqrt{N}}\sum_{\ell=1}^{k+1} e^{4LM\varepsilon \ell} \} \cup  \{\|  \tilde{u}^{N,\varepsilon,k+1}(t)- P_{X^N}u(t,\cdot) \|_{2,N} \geq \frac{C_1(\varepsilon) }{\sqrt{N}} \}} \\
\leq & \;
\mathbb{P}\croch{\{ U_N(t)\geq \frac{C_1(\varepsilon) }{\sqrt{N}}\sum_{\ell=1}^{k+1} e^{4LM\varepsilon \ell} \} } + \PP\croch{  \{\|  \tilde{u}^{N,\varepsilon,k+1}(t)- P_{X^N}u(t,\cdot) \|_{2,N} \geq \frac{C_1(\varepsilon)}{\sqrt{N}} \}}.
\end{split}
\end{equation}
We begin by examining the first term:
\[
\begin{split}  
& \PP\croch{ U_N(t) \geq\frac{C_1(\varepsilon)}{\sqrt{N}}\sum_{\ell=1}^{k+1} e^{4LM\varepsilon \ell}} \\
= & \PP\croch{ \{U_N(t)\geq \frac{C_1(\varepsilon)}{\sqrt{N}}\sum_{\ell=1}^{k+1} e^{4LM\varepsilon \ell}\} \cap \{ U_N(t)\leq U_N((k+1)\varepsilon)e^{4LM\varepsilon}\}} \\
& + \PP\croch{ \{U_N(t) \geq\frac{C_1(\varepsilon)}{\sqrt{N}}\sum_{\ell=1}^{k+1} e^{4LM\varepsilon \ell}\} \cap \{ U_N(t)> U_N((k+1)\varepsilon)e^{4LM\varepsilon}\}}.
\end{split}
\] 
From Lemma \ref{gronwall}, since for all $t\in[(k+1)\varepsilon, (k+2) \varepsilon)$,  $\PP\croch{ U_N(t)\leq U_N((k+1)\varepsilon)e^{4LM\varepsilon}} \geq 1-\frac{6}{N}$,
\[
\begin{split}  
 &\PP\croch{ \{U_N(t) \geq\frac{C_1(\varepsilon)}{\sqrt{N}}\sum_{\ell=1}^{k+1} e^{4LM\varepsilon \ell}\} \cap \{ U_N(t)\leq U_N((k+1)\varepsilon)e^{4LM\varepsilon}\}} \\
& \leq  \PP\croch{ \{U_N((k+1)\varepsilon) \geq \frac{C_1(\varepsilon)}{\sqrt{N}}\sum_{\ell=1}^{k+1} e^{4LM\varepsilon (\ell-1)}\} \cap \{ U_N(t)\leq U_N((k+1)\varepsilon)e^{4LM\varepsilon}\}}\\
& \leq \PP\croch{ U_N((k+1)\varepsilon) \geq \frac{C_1(\varepsilon)}{\sqrt{N}}\sum_{\ell=0}^{k} e^{4LM\varepsilon \ell}\} } \leq \frac{(k+1)}{N}(6+\tilde{C}_1),
\end{split}
\]
where the last equality comes from the induction hypothesis, noticing that by definition of $\tilde{u}^{N,\varepsilon,k+1}$,
\[ U_N((k+1)\varepsilon) = \left\|u^{N,\varepsilon}((k+1)\varepsilon) - \tilde{u}^{N,\varepsilon,k+1}((k+1)\varepsilon) \right\|_{2,N} = \left\|u^{N,\varepsilon}((k+1)\varepsilon) - P_{X^N}u(\cdot,(k+1)\varepsilon) \right\|_{2,N}.\]
Moreover, still from Lemma \ref{gronwall},
\[
\begin{split}  
& \PP\croch{ \{U_N(t) \geq \frac{C_1(\varepsilon)}{\sqrt{N}}\sum_{\ell=1}^{k+1} e^{4LM\varepsilon \ell}\} \cap \{ U_N(t)> U_N((k+1)\varepsilon)e^{4LM\varepsilon}\}} \\
\leq \;& \PP\croch{ U_N(t)> U_N((k+1)\varepsilon)e^{4LM\varepsilon}} \leq \frac{6}{N}.
\end{split}
\]
Thus, for all $t\in [(k+1)\varepsilon, (k+2) \varepsilon)$,
\[
 \PP\croch{ U_N(t) \geq\frac{C_1(\varepsilon)}{\sqrt{N}}\sum_{\ell=1}^{k+1} e^{4LM\varepsilon \ell} } \leq \frac{(k+1)}{N}(6+\tilde{C}_1) + \frac{6}{N}.
\]
Moreover, coming back to the second term of \eqref{eq:blink1}, from Theorem \ref{th:graph_limit}, for all $t\in [(k+1)\varepsilon, (k+2) \varepsilon)$,
\[
\PP\croch{\| \tilde  u^{N,\varepsilon,k}(t) - P_{X^N}u(t,\cdot) \|_{2,N}\geq  \frac{C_1(\varepsilon)}{\sqrt{N}}} \leq \frac{\tilde{C}_1}{N} 
\]
Putting these last two  convergence results together, and coming  back to \eqref{eq:blink1}, for  all $t\in [(k+1)\varepsilon, (k+2) \varepsilon)$,
\[
\mathbb{P} \croch{ \| u^{N,\varepsilon}(t) - P_{X^N}u(t,\cdot) \|_{2,N} \geq \frac{C_1(\varepsilon)}{\sqrt{N}} \sum_{\ell=0}^k e^{4LM\varepsilon \ell}} \leq \frac{(k+1)}{N}(6+\tilde{C}_1) + \frac{6}{N}+ \frac{\tilde{C}_1}{N} = \frac{(k+2)}{N}(6+\tilde{C}_1) .
\]
By continuity of $u^{N,\varepsilon}$ and $P_{X^N}u$, we can extend this result to the full interval $[(k+1)\varepsilon, (k+2)\varepsilon]$, and finally
obtain the desired result:
\[
\mathbb{P} \croch{ \sup_{t\in[(k+1)\varepsilon, (k+2)\varepsilon]}\| u^{N,\varepsilon}(t) - P_{X^N}u(t,\cdot) \|_{2,N} \geq \frac{C_1(\varepsilon)}{\sqrt{N}} \sum_{\ell=0}^{k+1} e^{4LM \varepsilon \ell} } \leq \frac{(k+2)}{N}(6+\tilde{C}_1)
\]
By induction, \eqref{prop_rec} holds for all $k\in \{0,\cdots,n-1\}$.

Moreover, recalling that $\varepsilon=\frac{T}{n}$, and that $C_1(\varepsilon) = \sqrt{\varepsilon} \sqrt{1 + M^2\| D\|_{L^\infty(\R^d)}^2 }e^{(\frac{1}{2}+4ML)\varepsilon}$,
 it holds
 $$\sum_{\ell=0}^{k} e^{4LM \varepsilon \ell} = \frac{ e^{4LM\varepsilon(k+1)}-1}{e^{4LM\varepsilon}-1} \leq \frac{e^{4LMT}-1}{4LM\varepsilon}$$
 and 
 \[
 \frac{C_1(\varepsilon)}{\sqrt{N}} \sum_{\ell=0}^k e^{4LM\varepsilon \ell} \leq \sqrt{1 + M^2\| D\|_{L^\infty(\R^d)}^2 }e^{(\frac{1}{2}+4ML)T} \frac{e^{4LMT}-1}{4LM\sqrt{N\varepsilon}} = \frac{C_3(T)}{\sqrt{N\varepsilon}}.
 \]
 
Thus, we get for all $k\in \{0,\cdots,n-1\}$ and all $t \in [k\varepsilon, (k+1)\varepsilon)$:
\[\PP\croch{  \sup_{t\in[k\varepsilon, (k+1)\varepsilon]}\left\|{u}^{N,\varepsilon}(t) - P_{X^N}u(t,\cdot) \right\|_{2,N} \geq \frac{C_3(T)}{\sqrt{N\varepsilon}}  }\leq \frac{\tilde{C}_3(T)}{N\varepsilon},\]
where $\tilde{C}_3(T) = T(\tilde{C}_1+6)$.
Since the constants are independent of $k$,  we obtain 
\[  \PP\croch{ \sup_{t \in [0,T]}  \left\|{u}^{N,\varepsilon}(t) - P_{X^N}u(t,\cdot) \right\|_{2,N} \geq \frac{C_3(T)}{\sqrt{N\varepsilon}} } \leq \frac{\tilde{C}_3(T)}{N\varepsilon}
\] which concludes the proof.
\end{proof}

{\begin{rem}
As in section \ref{Sec:Deter},  we can construct a constant-by-part bounded function $\tilde{u}_N^{\varepsilon,k}\in L^\infty(I\times\R)$ defined by 
        \[\text{for } k \in \mathbb{N},\forall  t \in [k\varepsilon,(k+1)\varepsilon], \; \forall x\in I_i^N, \qquad \tilde{u}_N^{\varepsilon,k}(x,t) = \tilde{u}_i^{N,\varepsilon,k}(t).\]
     Noticing that $\displaystyle \| u_N \|_{\mathcal{C}(0,T;L^2(I))}  =   \sup_{t \in [0,T]}  \left\|{u}^{N}(t)  \right\|_{2,N} $ and  $\displaystyle \|\tilde{u}_N^{\varepsilon,k} \|_{\mathcal{C}(k\varepsilon,(k+1)\varepsilon;L^2(I))}  =   \sup_{t \in [k\varepsilon,(k+1)\varepsilon]}  \left\|{u}^{N,\varepsilon,k}(t)  \right\|_{2,N}$, we can obtain an equivalent result for a blinking system on weighted random graphs generated by a deterministic sequence with a straightforward adaptation of the proof using this time Theorem \ref{th:deterministic}  instead of Theorem \ref{th:graph_limit}. 
\end{rem}}

We can then complete the schematic linking the various systems of interest as shown in Fig. \ref{fig:Ergodicity}. 
Thus, the graph limit equation \eqref{eq:graph_limit} is the limit of the three systems \eqref{eq:ODEs}, \eqref{eq:ODEav} and \eqref{eq:ODEbl} as $N$ goes to infinity.
 
\begin{figure}[h!]
\centering
\includegraphics[scale=1]{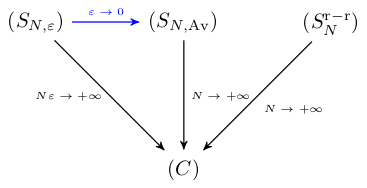}
\caption{Ergodicity}
\label{fig:Ergodicity}
\end{figure}

\section{Numerical Simulations}\label{Sec:Sim}

\subsection{Application to the Weighted Random Graph model of Garlaschelli}

The Erd\"os-R\'enyi random graph is an unweighted graph constructed by randomly linking any two nodes with a given probability $p$.
In \cite{Garlaschelli_2009}, Garlaschelli introduced a weighted version of the Erd\"os-R\'enyi random graph as follows. 
Let $p\in (0,1)$. We generate between every pair of edges $(i,j)$ an edge with an integer weight $w\in \N$, with probability $ p^w (1-p)$.
Notice that this enters our framework, defining a weighted random graph law $q$ with support in $\N$ by
\begin{equation}
q(x,y;\cdot)= (1-p) \sum_{i=0}^{+\infty} p^i \delta_i, \qquad \text{ for all } x,y\in\R.
\label{eq:example_Garaschelli}
\end{equation}
We can easily check that for all $x,y\in\R$, $q(x,y;\cdot)$ is a probability distribution, since 
\[
\int_{\R^+}q(x,y;dw) = (1-p)\sum_{i=1}^{+\infty}  p^i = (1-p)\frac{1}{1-p} = 1.
\]
Furthermore, its first moment (by definition constant in $x,y$) is given by
\[
\bw(x,y)=\int_{\R^+}w q(x,y;dw) = (1-p)\sum_{i=1}^{+\infty} i p^i 
= \frac{p}{1-p},
\]
and all moments of higher order are bounded as well. In particular, 
\[
\int_{\R^+}(w-\bw(x,y))^2 q(x,y;dw) = \int_{\R^+}w^2 q(x,y;dw)- \left(\int_{\R^+}w q(x,y;dw) \right)^2 
= \frac{p}{(1-p)^2}.
\]
This model implies that all edges are statistically equivalent, as $q(x,y;\cdot)$ does not depend on $x,y$. The probability that no edge is drawn between a given pair of vertices  given by $(1-p) \sum_{i=0}^{+\infty} p^i \delta_i(0) = (1-p)$, and consequently, the probability that an edge has a non-zero weight is given by $p$, as in the Erd\"os-R\'enyi random graph. 

In a refined version of the model, the parameter $p$ is now allowed to depend on the nodes' indices $x$ and $y$. For instance, setting $p(x,y)=\frac{xy}{2}$, the weighted graph law becomes for all $(x,y)\in I^2$
\begin{equation}
q(x,y;dw)= (1-\frac{xy}{2}) \sum_{i=0}^{+\infty} (\frac{xy}{2})^w \delta_i(w),
\label{eq:example_Garaschelli_xy}
\end{equation}
and its expected value and variance become 
\[
\bar{w}(x,y) =  \frac{xy}{2-xy}; \qquad \int_{\R^+} (w-\bw(x,y))^2 q(x,y;dw) = \frac{2xy}{(2-xy)^2}.
\]
Fig.~\ref{Fig:Garlaschelli_xy_matrices_Rand} and \ref{Fig:Garlaschelli_xy_matrices_Det} show the random matrices $(\xi_{ij})_{i,j\in\elts}$ generated respectively by a random or a deterministic sequence, for various values of $N$.

\begin{figure}[h!]
\centering
\includegraphics[width = 0.32\textwidth, trim = 0cm 0cm 0cm 0cm, clip=true]{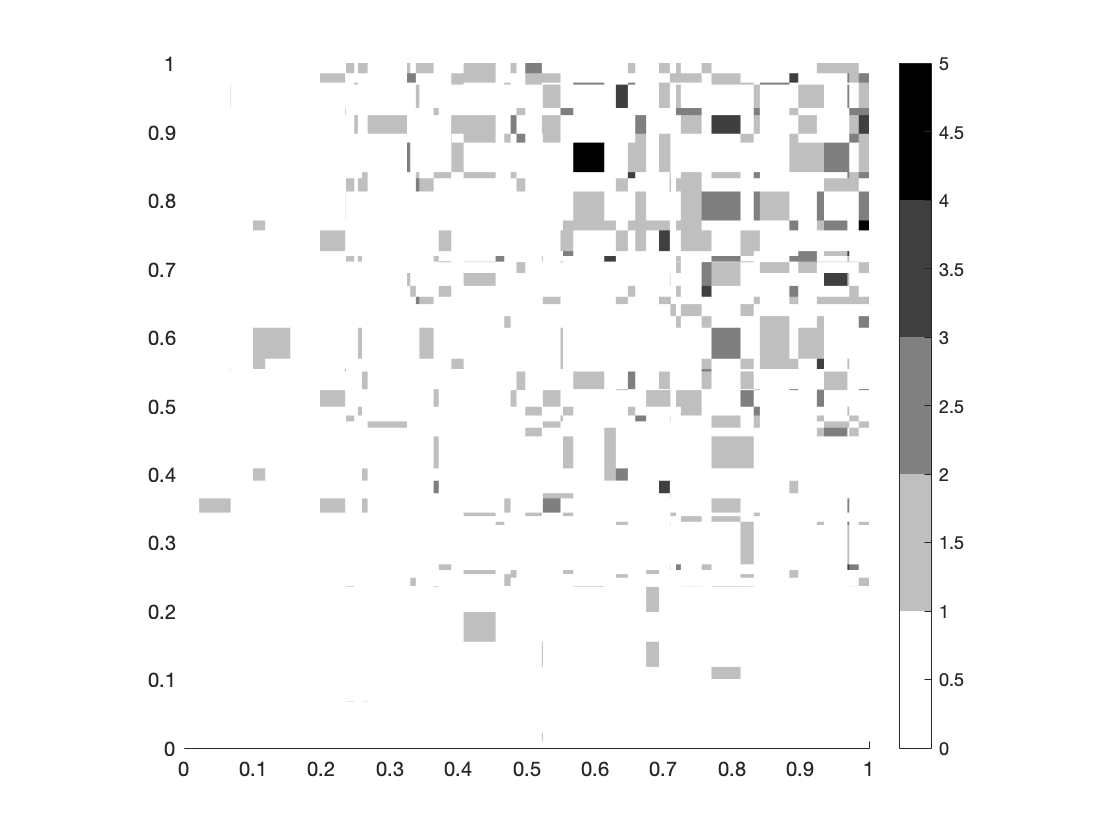}
\includegraphics[width = 0.32\textwidth, trim = 0cm 0cm 0cm 0cm, clip=true]{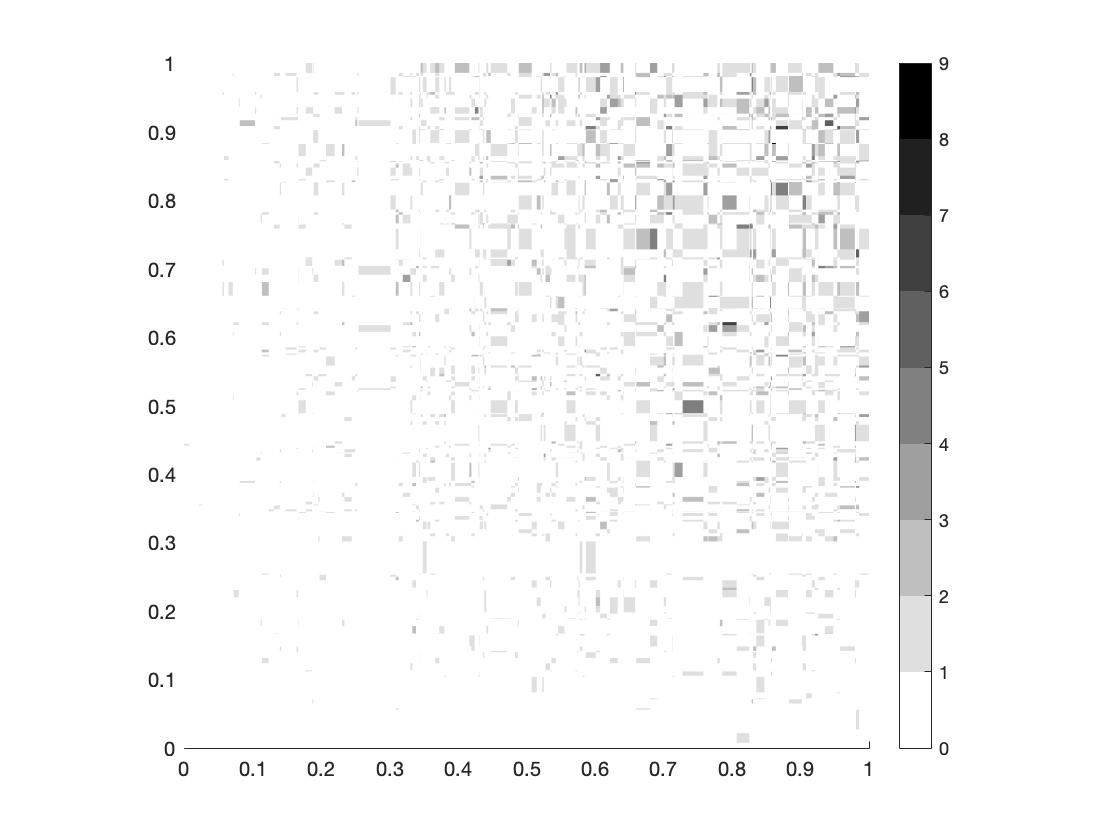}
\includegraphics[width = 0.32\textwidth, trim = 0cm 0cm 0cm 0cm, clip=true]{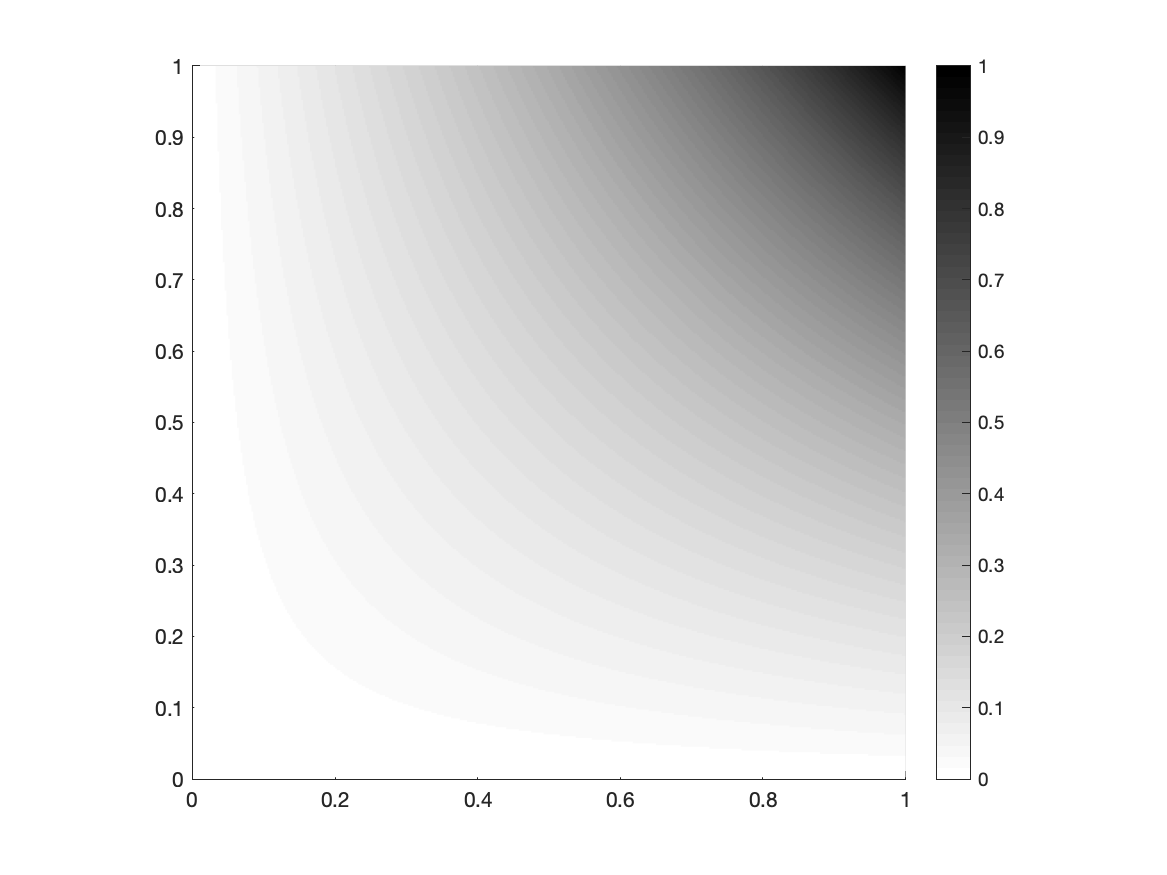}
\caption{Left and Center: Random interaction matrices generated by random sequences for $N=60$ and $N=150$, for the weighted random graph law \eqref{eq:example_Garaschelli_xy}. Right: Corresponding graphon $(x,y)\mapsto \bw(x,y)$. }
\label{Fig:Garlaschelli_xy_matrices_Rand}
\end{figure}

\begin{figure}[h!]
\centering
\includegraphics[width = 0.32\textwidth, trim = 0cm 0cm 0cm 0cm, clip=true]{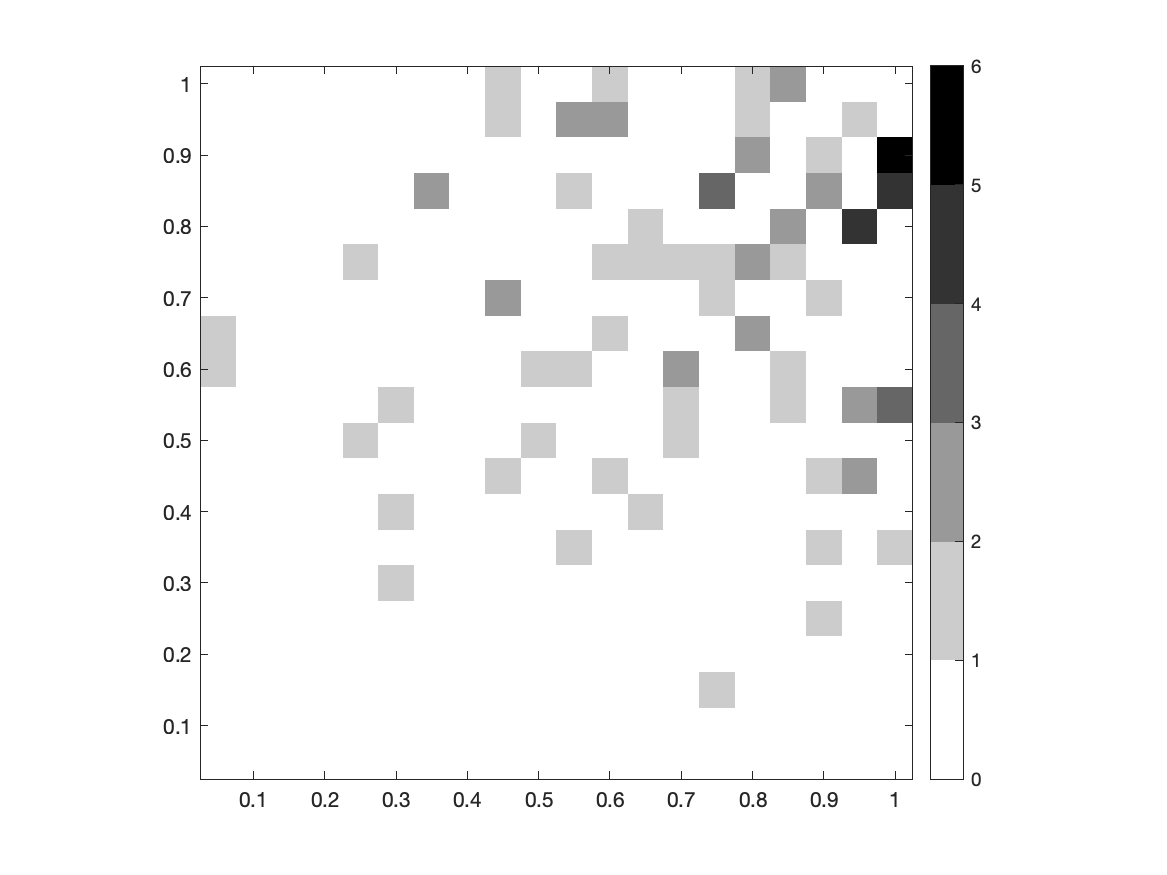}
\includegraphics[width = 0.32\textwidth, trim = 0cm 0cm 0cm 0cm, clip=true]{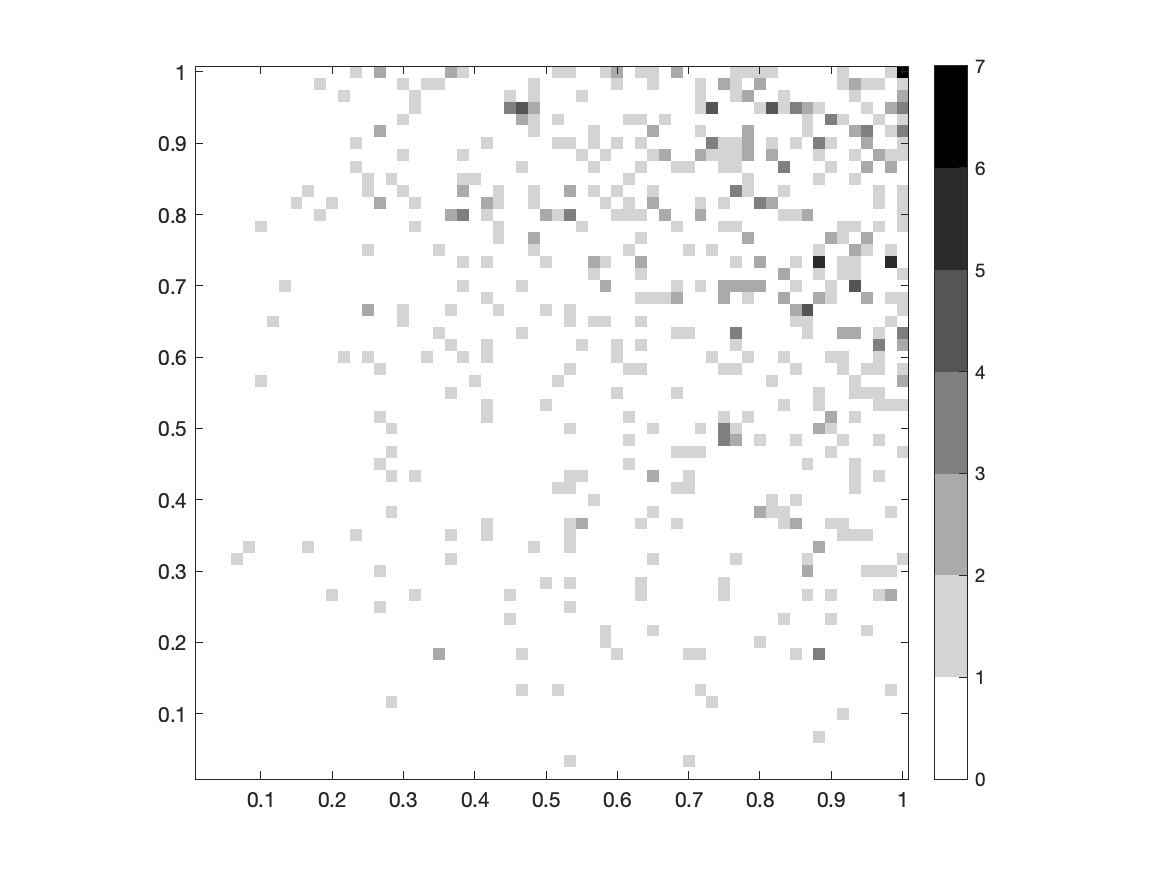}
\includegraphics[width = 0.32\textwidth, trim = 0cm 0cm 0cm 0cm, clip=true]{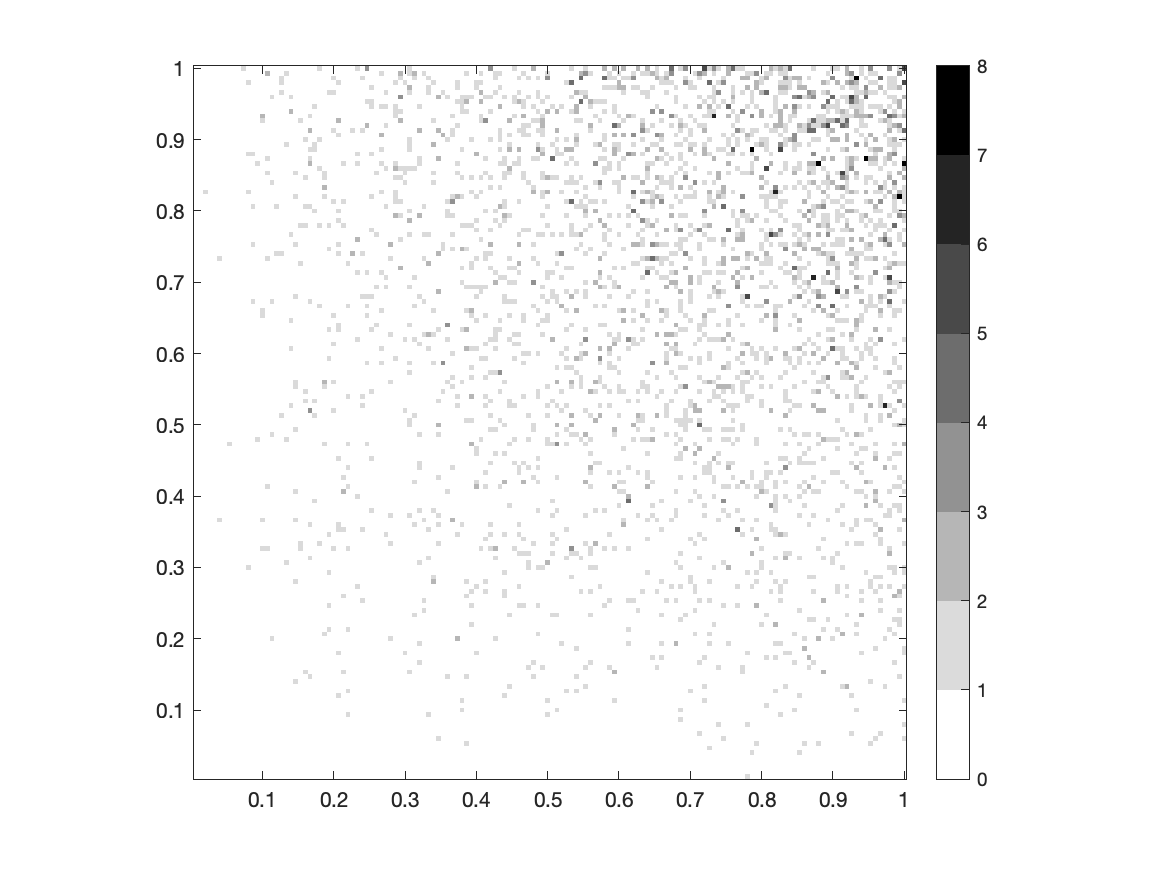}
\caption{Random interaction matrices generated by deterministic sequences for $N=20$, $N=60$ and $N=150$, for the weighted random graph law given by \eqref{eq:example_Garaschelli_xy}.}
\label{Fig:Garlaschelli_xy_matrices_Det}
\end{figure}

Applying Theorem \ref{th:graph_limit}, we expect that the microscopic system \eqref{eq:ODEs} with such random weights will converge to the  graphon $u(t,x)$ solution to the following integro-differential equation:
\begin{equation}
\begin{cases}
\partial_t u(x,t)  = \dsp \int_I \frac{xy}{2-xy} D(u(y,t)-u(x,t))dy\\
u(x,0)   =  g(x),  ~~~x \in I.
\end{cases}
\end{equation}

We illustrate this result numerically with initial data and interaction function respectively given by
\begin{equation}\label{eq:D}
D(z) = \frac{z}{1+\|z\|^2} \text{ and } g:x\mapsto \sin(4x)^2.
\end{equation}


Notice that the random matrices corresponding to the microscopic systems have a large proportion of zero-weight edges (with probability $1-\frac{xy}{2}$, as shown above), whereas the continuous graphon satisfies $\bw(x,y)>0$ as soon as $xy\neq0$.
Moreover, the edge weights $(\xi_{ij})_{i,j\in\elts}$ can theoretically take any integer value, with no upper bound, and
the higher the number of agents, the more likely it becomes to randomly generate some edge weights of high value, as seen in Fig.~\ref{Fig:Garlaschelli_xy_matrices_Rand} and \ref{Fig:Garlaschelli_xy_matrices_Det}.  
Conversely, for small values of $N$, an agent can have very few outgoing or ingoing edges, which considerably reduces its interaction with the group, and hence its convergence towards consensus.

As a result, the microscopic system and the projected solution of the graph limit equation \eqref{eq:graph_limit} have noticeably different time evolutions. Fig.~\ref{Fig:Garlaschelli_xy_N60}  shows the evolution of the agents' positions $(u_i(t))_{i\in\elts}$, for the microscopic system \eqref{eq:ODEs} generated by a random sequence $(X_i)_{i\in\elts}$  (left plots), compared to the projected graph limit solution to \eqref{eq:graph_limit} $(u(X_i,t))_{i\in\elts}$ (right plots).

Interestingly, notice that in the microscopic system, some agents' positions stay constant in time: their corresponding nodes have no ingoing edges (which is increasingly likely as $i$ is small), and they do not feel any influence from the other agents. This behavior is not observed in the projected graph limit's evolution, as in the continuous graphon $(x,y)\mapsto \bw(x,y)=\frac{xy}{2-xy}$, the only zero-weight edges are found when $X_i X_j$ is exactly equal to zero, which happens with probability zero.

\begin{figure}[h!]
\centering
\includegraphics[width = 0.45\textwidth]{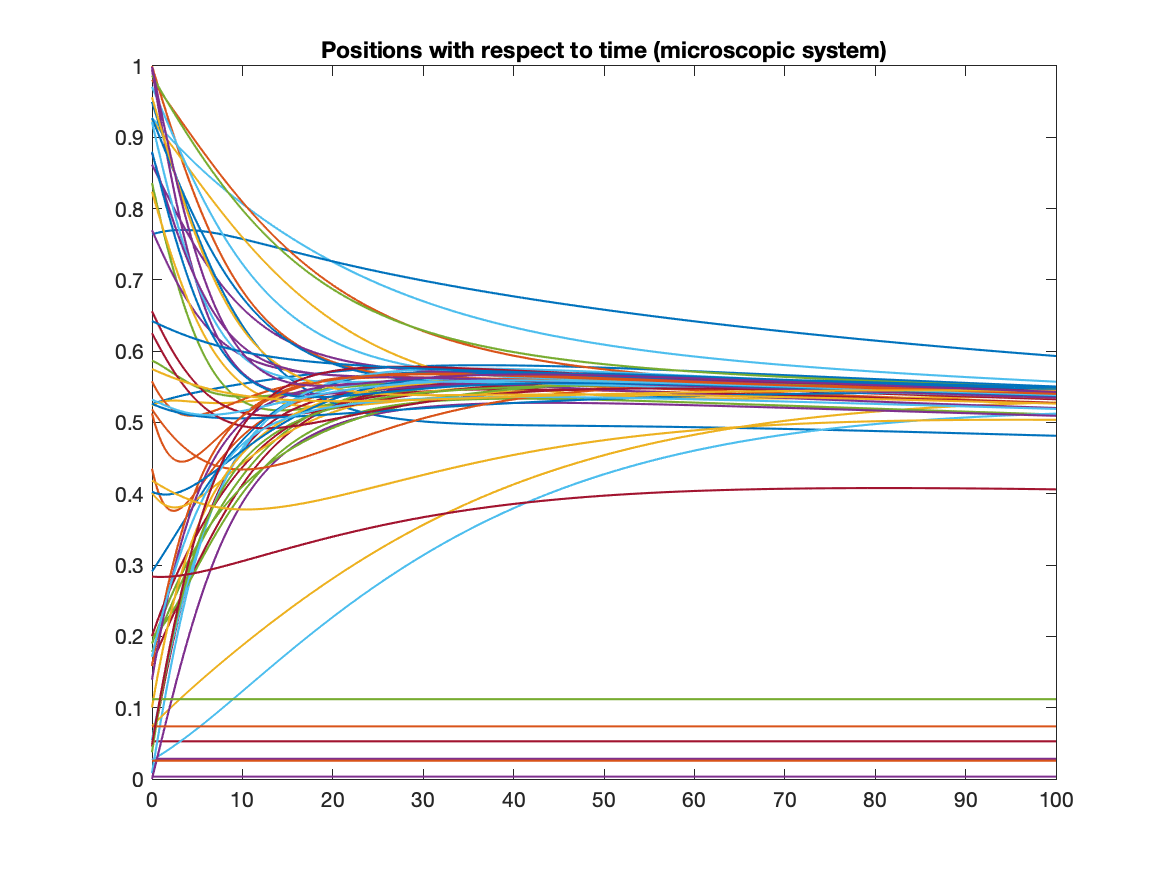}
\includegraphics[width = 0.45\textwidth]{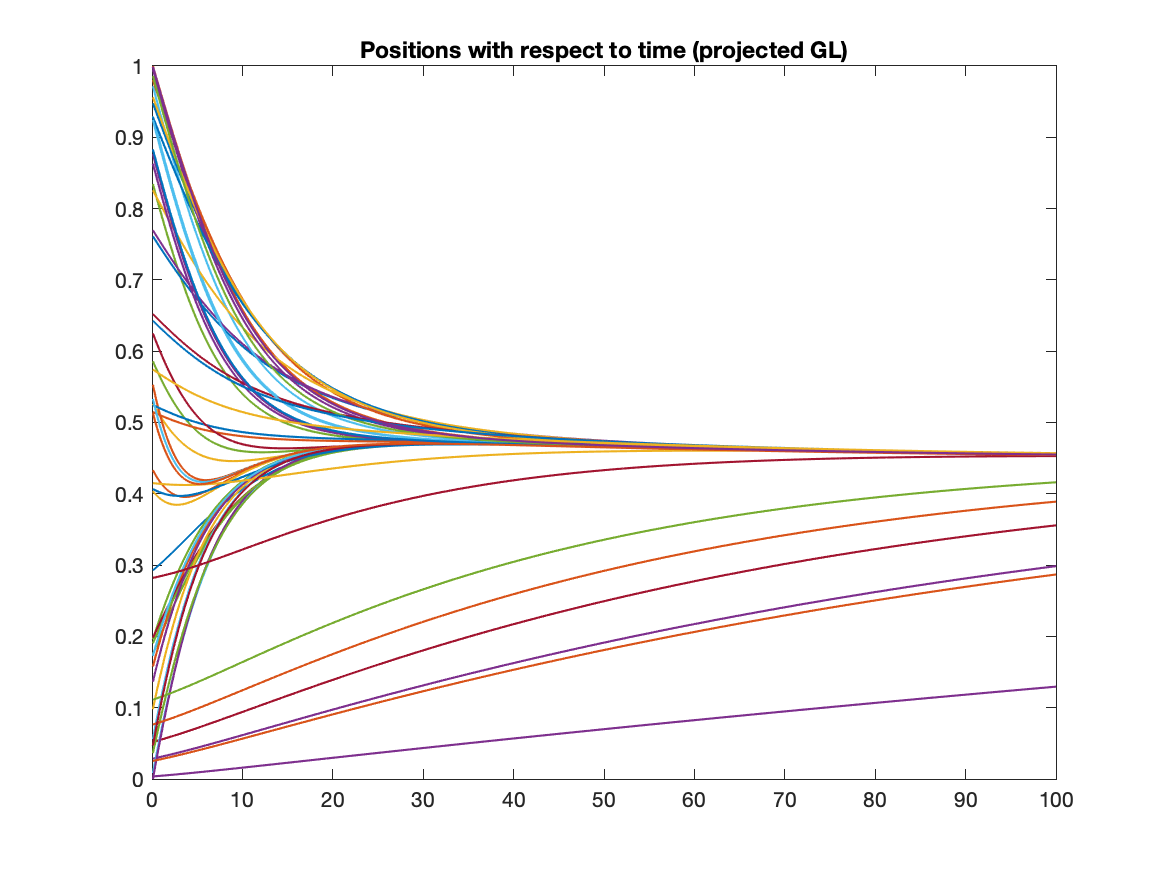}
\caption{Time evolution of the microscopic system \eqref{eq:ODEs} for $N=60$ (left), and of the corresponding projection of the graph limit \eqref{eq:graph_limit} (right), for the weighted random graph law \eqref{eq:example_Garaschelli_xy}.}
\label{Fig:Garlaschelli_xy_N60}
\end{figure}


Fig.~\ref{Fig:Garlaschelli_Det_xy_N60} 
shows the evolution of the graph limit $u(\cdot,t)$ (in red) for $t=0$, $t=6$ and $t=40$  and of  the projection $u_N(\cdot,t)$ of the microscopic system generated by a deterministic sequence \eqref{ODE22}, for $N=60$ (in black).
Observe that while $u(\cdot,t)$ converges to consensus in $L^2$, it does not in $L^\infty$, as the graphon is not strongly connected (as proven in \cite{BonnetPouradierDuteilSigalotti21}).

\begin{figure}[h!]
\centering
\includegraphics[width = 0.32\textwidth]{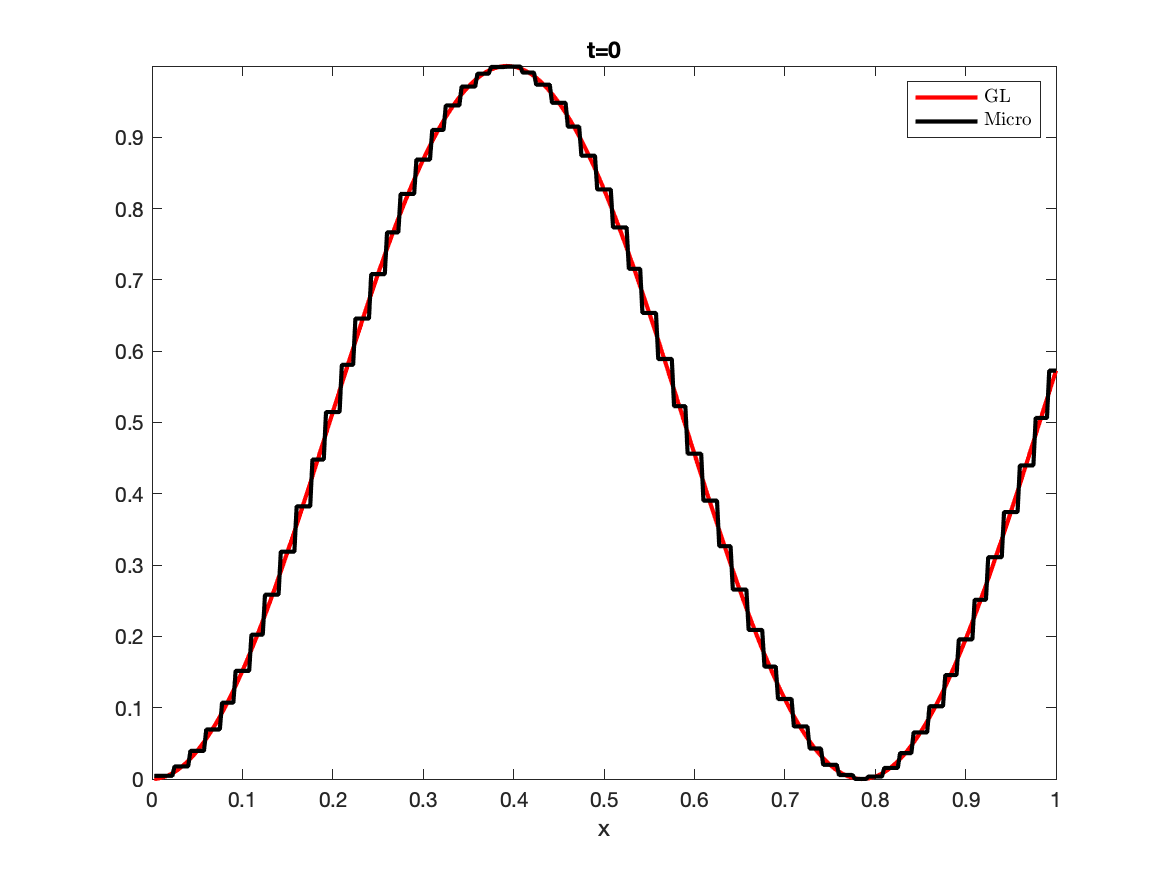}
\includegraphics[width = 0.32\textwidth]{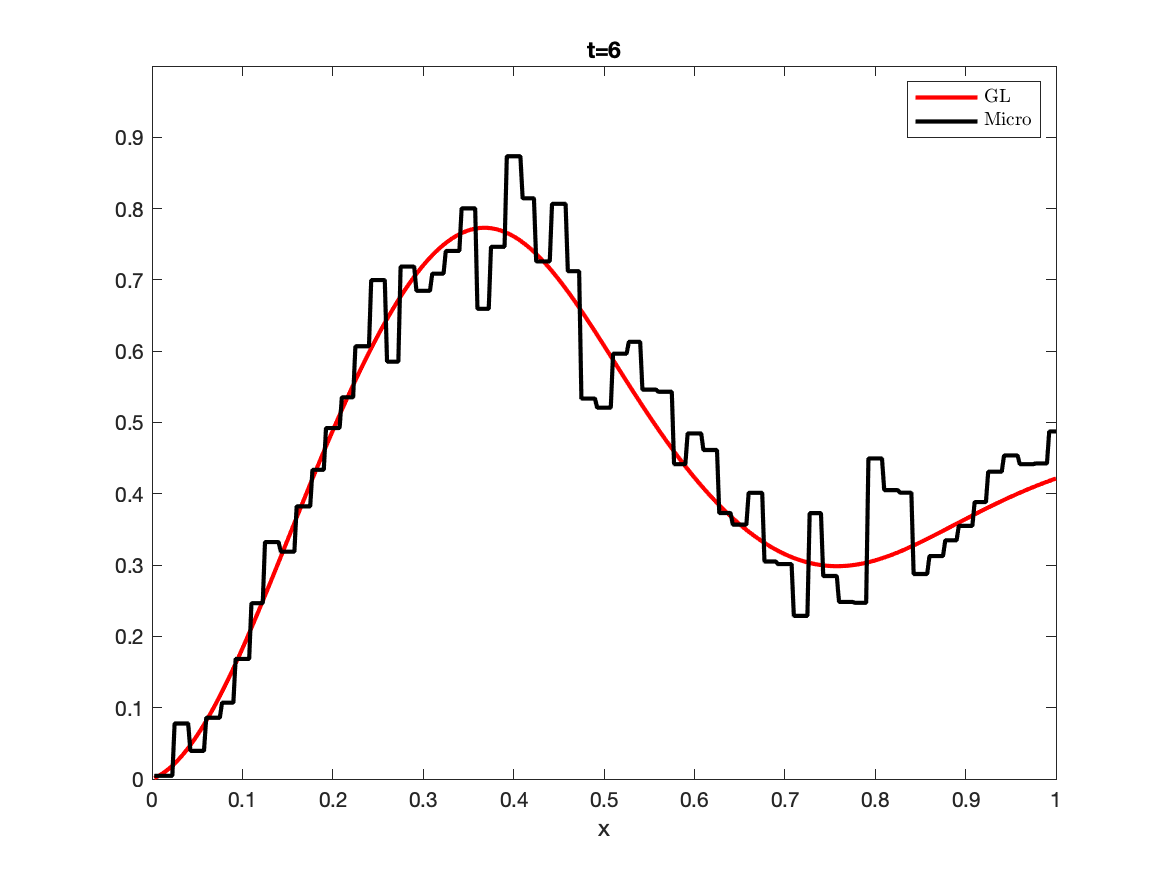}
\includegraphics[width = 0.32\textwidth]{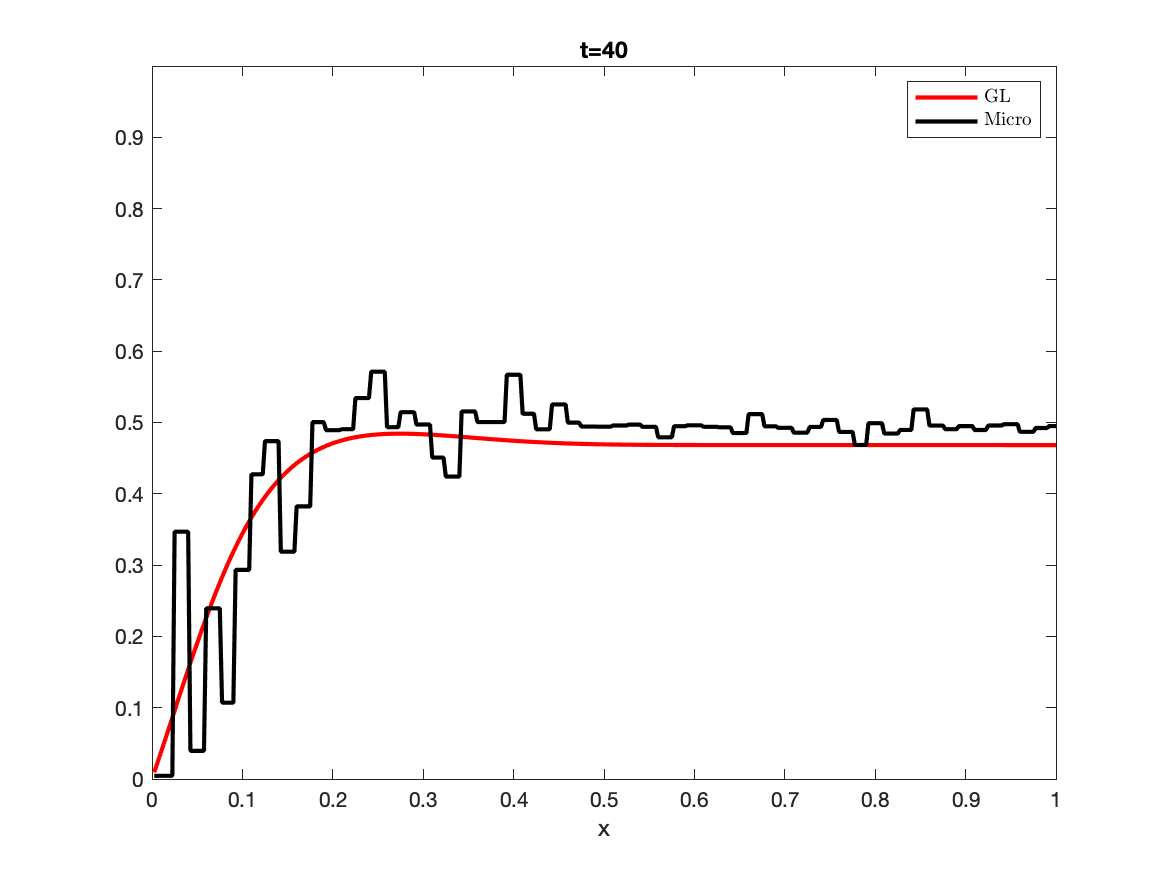}
\caption{Evolution of the graph limit $u(\cdot,t)$ solution to \eqref{eq:graph_limit} (red) and of $u_N(\cdot,t)$ constructed from the solution to \eqref{ODE22} (black) for $N=60$, at $t=0$, $t=6$ and $t=40$, for the weighted random graph law \eqref{eq:example_Garaschelli_xy}.}
\label{Fig:Garlaschelli_Det_xy_N60}
\end{figure}


The convergence of both microscopic systems to the graph limit solution is illustrated in Fig.~\ref{Fig:Garlaschelli_xy_L2norm}.

\begin{figure}[h!]
\centering
\includegraphics[width = 0.45\textwidth]{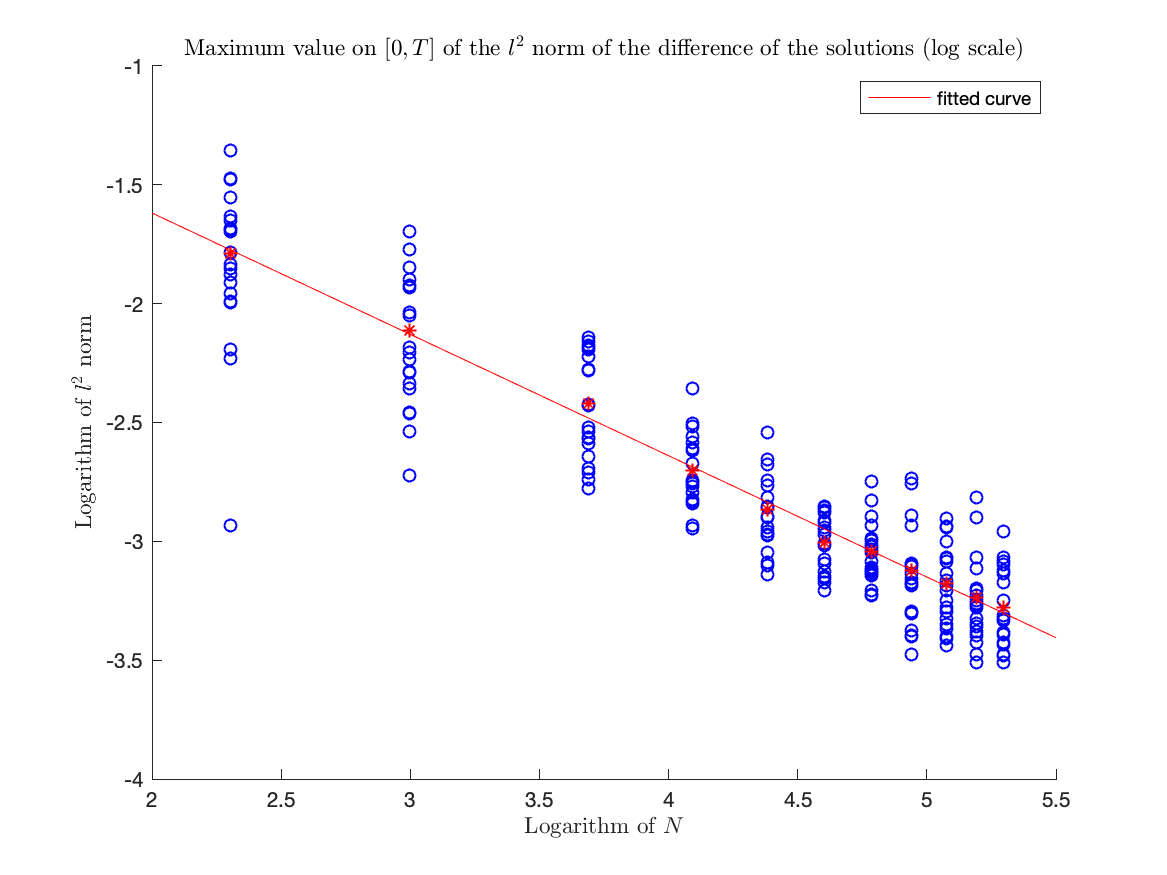}
\includegraphics[width = 0.45\textwidth]{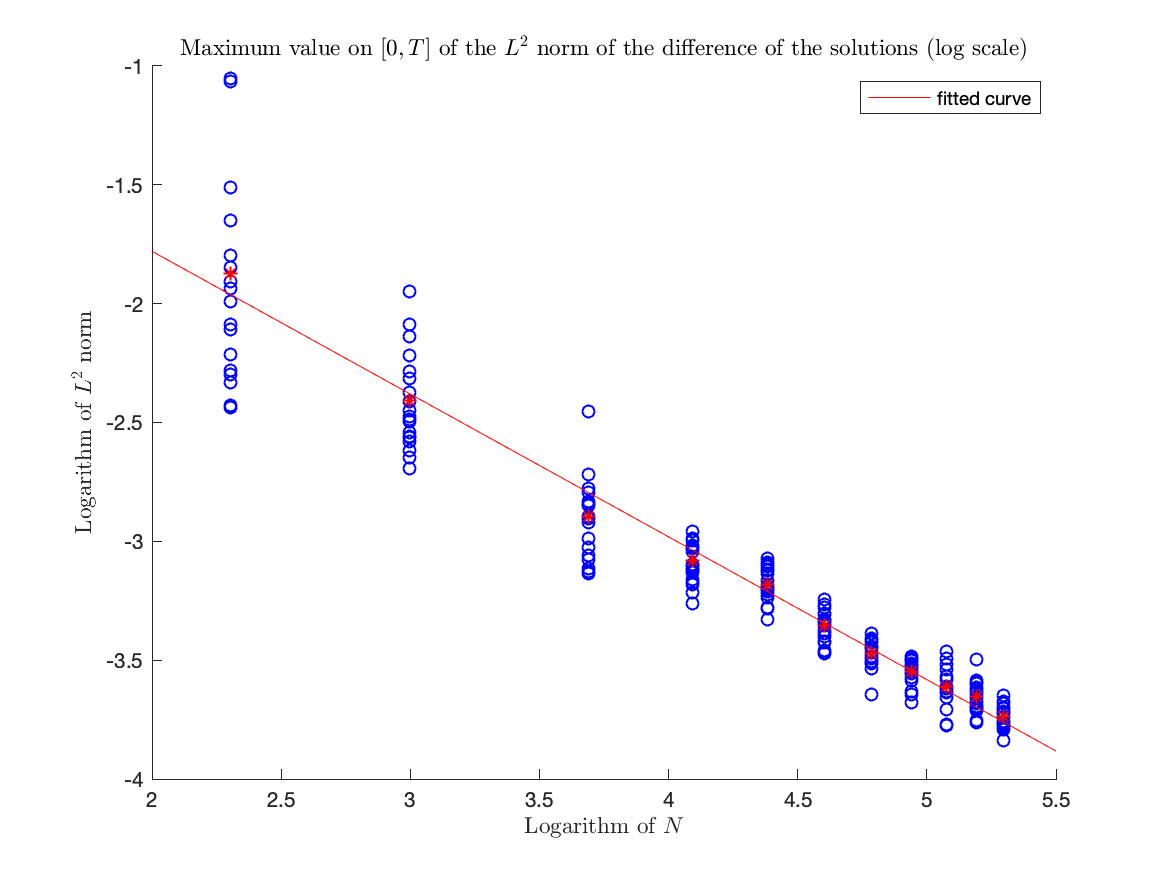}
\caption{Convergence of \eqref{eq:ODEs} quantified by $\sup_{t \in [0,T]}  \|u^N(t) - \mathbf{P}_{\tilde{X}_N} u(\cdot,t) \|_{2,N}$ (left) and of \eqref{ODE22} quantified by $\sup_{t \in [0,T]}  \|u_N(\cdot,t) - u(\cdot,t) \|_{L^2}$ (right) for different values of $N$, with 20 runs for each value of $N$ (logarithmic scale), for the weighted random graph law \eqref{eq:example_Garaschelli_xy}.}
\label{Fig:Garlaschelli_xy_L2norm}
\end{figure}

\subsection{Weighted ``small world'' network}

In \cite{WS98}, Watts and Strogatz introduced a model for a ``small-world'' network, to interpolate between regular and random networks.
The construction procedure for a finite set of $N$ nodes is as follows. Connect each node with its $k$ closest neighbors to form a ring lattice (this is the deterministic underlying structure of the network). Then, rewire each edge at random with probability $p$.
The constructed network reflects the well-known ``small-world'' property according to which each individual has a small probability to be connected with another individual supposedly outside its circle. 

We can refine this model by considering weighted edges. Given two nodes of the graph, we connect them with an edge of weight $1$ if they are among each other's closest $k$ neighbors, i.e. if $|X_i-X_j|\leq r$, where $r:=\frac{k}{2N}$. Then, with probability $p = \frac{|X_i-X_j|}{r}$, rewire each of these edges at random, giving the new edge a weight drawn uniformly in the interval $[0,1]$. The weighted random graph law 
 giving the edge weight distribution for each $(i,j)$ is given by 

\begin{equation}
q(x,y;dw) = 
\begin{cases}
\frac{\rho(x,y)}{r} d\lambda_{[0,1]} + (1-\frac{\rho(x,y)}{r}) \delta_1 \qquad \text{ if } \rho(x-y)\leq r\\
d\lambda_{[0,1]} \qquad \text{ otherwise}
\end{cases}
\label{eq:Example_SmallWorld}
\end{equation}
where $d\lambda_{[0,1]}$ represents the Lebesgue measure restricted to the interval $[0,1]$
and $\rho(x,y) = \min\{|x-y|,|x-y-1|,|y-x-1|\}$.

Note that for all $x,y\in [0,1]$, the probability measure $q(x,y;\cdot)$ is supported in $[0,1]$, hence all its moments are finite. 
Its first moment satisfies
\[
\bw(x,y)=\int_{\R^+} w q(x,y;dw) = 
\begin{cases}
 (1-\frac{\rho(x,y)}{2r}) \qquad \text{ if } \rho(x,y)\leq r\\
\frac{1}{2} \qquad \text{ otherwise.}
\end{cases}
\]

Fig.~\ref{Fig:SmallWorld_matrices} depicts examples of random interaction matrices $(\xi_{ij})_{i,j\in\elts}$ generated by a random sequence (left) or by a deterministic sequence (center), for $r=0.3$. 
Unlike in the previous example, here, the limiting graphon $(x,y)\mapsto \bw(x,y)$ is bounded away from zero.

\begin{figure}[h!]
\centering
\includegraphics[width = 0.32\textwidth, trim = 0cm 0cm 0cm 0cm, clip=true]{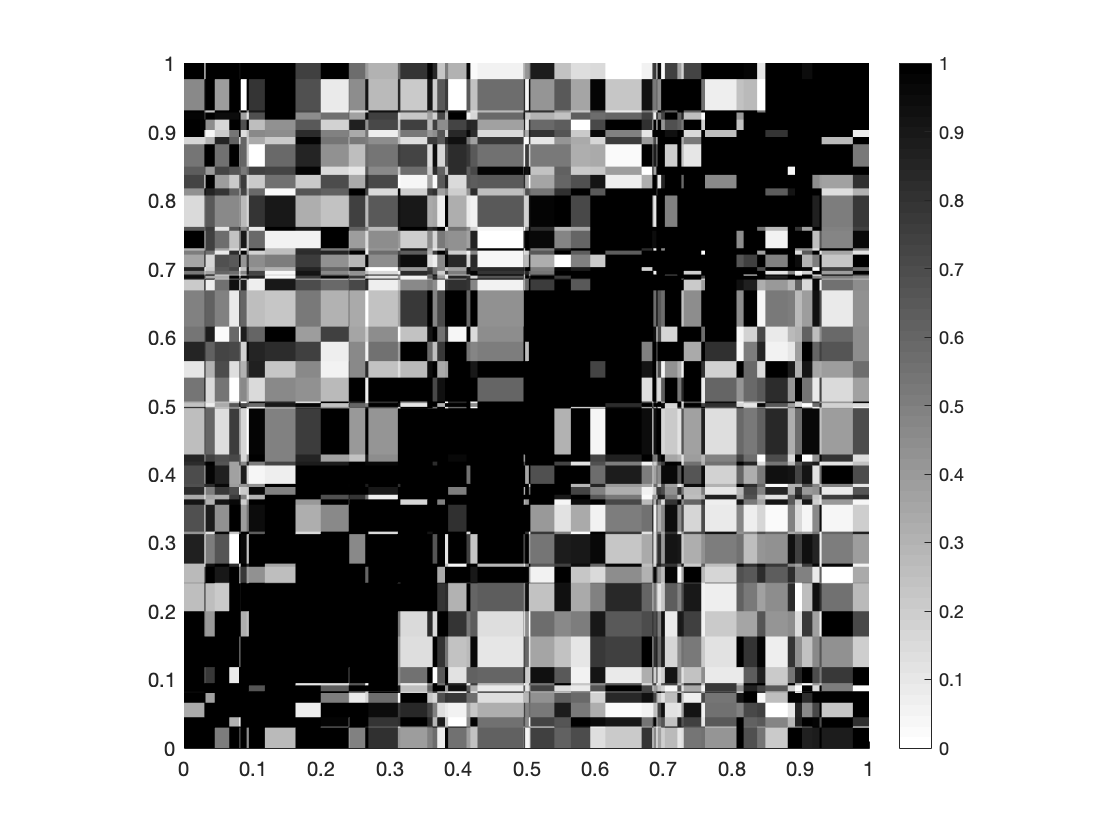}
\includegraphics[width = 0.32\textwidth, trim = 0cm 0cm 0cm 0cm, clip=true]{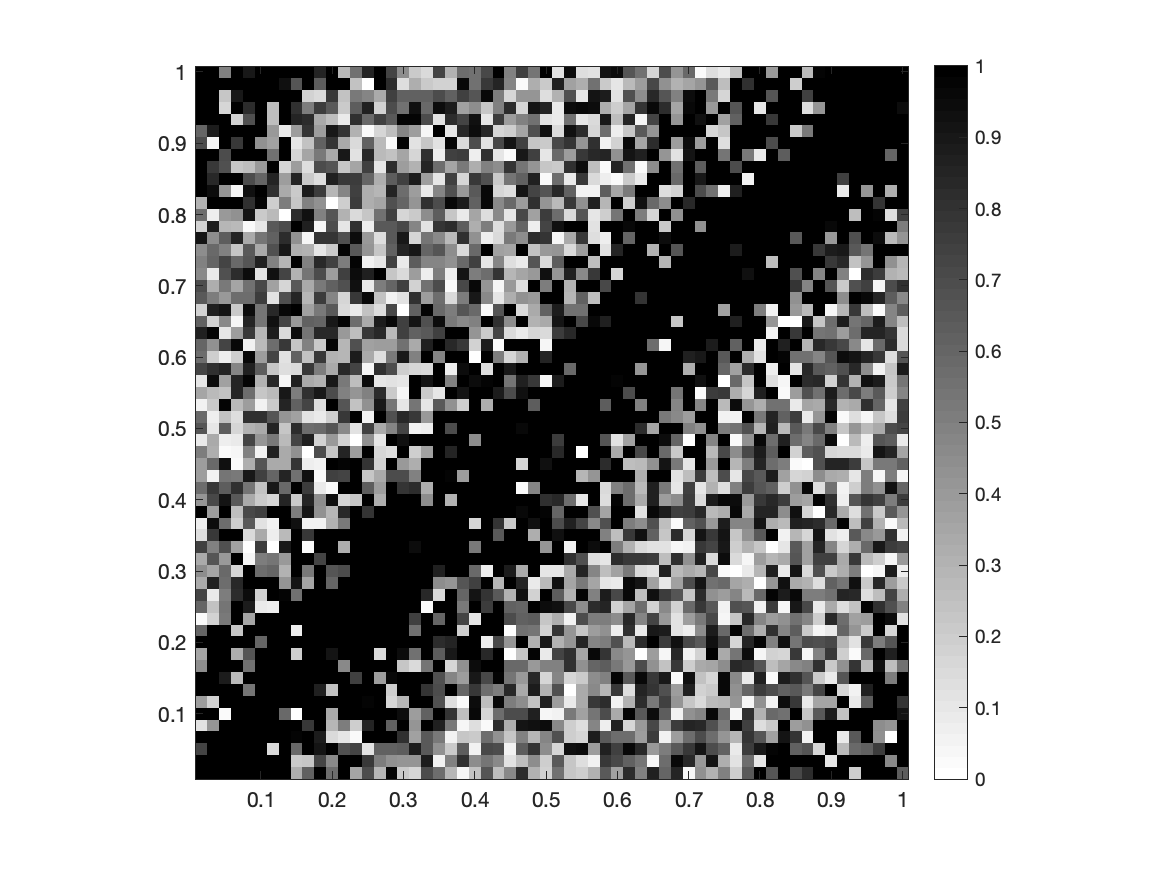}
\includegraphics[width = 0.32\textwidth, trim = 0cm 0cm 0cm 0cm, clip=true]{Graphon_xy_4.png}
\caption{Values of the random interaction matrices generated from a random sequence (left) and a deterministic sequence (right) according to the weighted random graph law \eqref{eq:Example_SmallWorld} for $N=60$. Right: Corresponding continuous graphon $(x,y)\mapsto \bw(x,y)$. }
\label{Fig:SmallWorld_matrices}
\end{figure}

Also unlike the previous example, the probability that an edge's weight is exactly $0$ is zero, which means that the graph is fully connected, even for finite values of $N$, as seen in Fig.~\ref{Fig:SmallWorld_matrices}.  
Fig.~\ref{Fig:SmallWorld_N60} shows the evolution of the microscopic system \eqref{eq:ODEs} generated by a random sequence (left) and that of the corresponding projected graph limit solution $(u(X_i,\cdot))_{i\in\elts}$, for $N=60$. The interaction function and the initial data are given by \eqref{eq:D}. Convergence to consensus seems to happen at similar rates in both cases. 

\begin{figure}[h!]
\centering
\includegraphics[width = 0.45\textwidth]{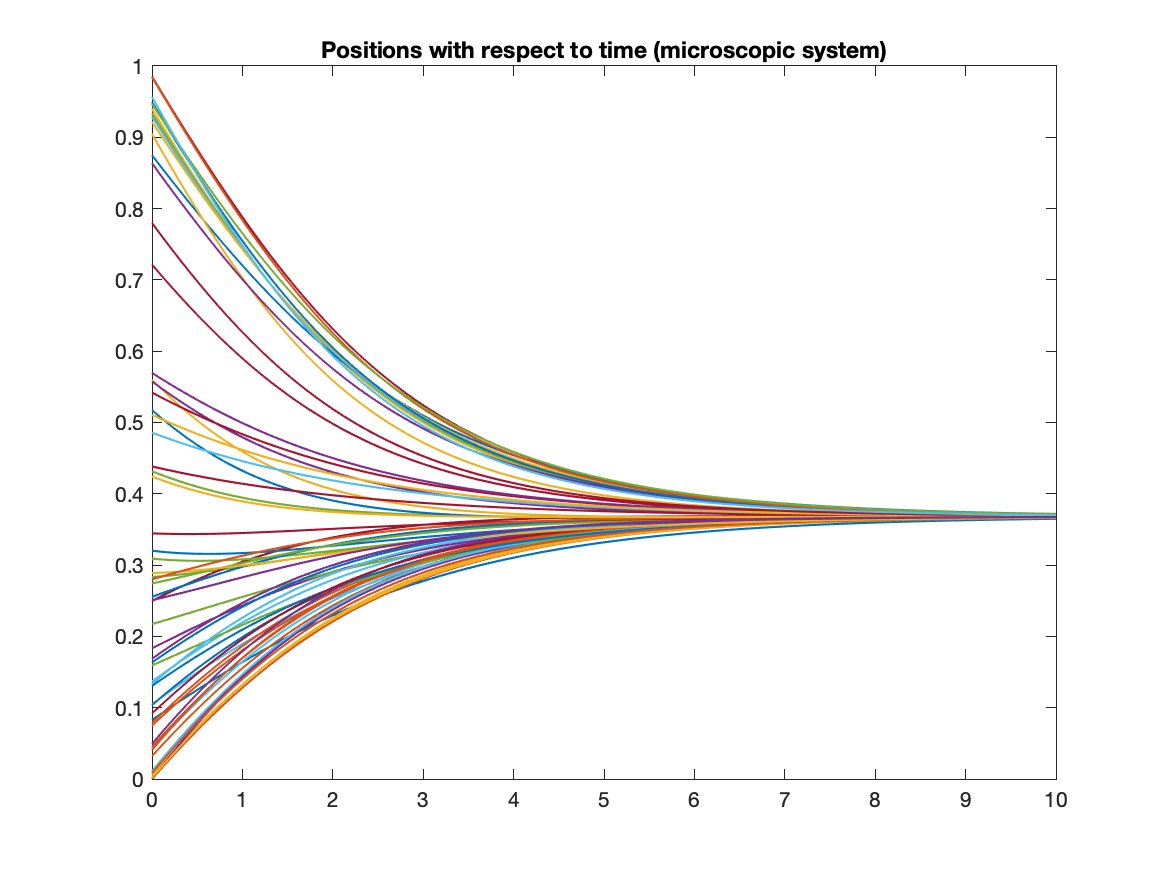}
\includegraphics[width = 0.45\textwidth]{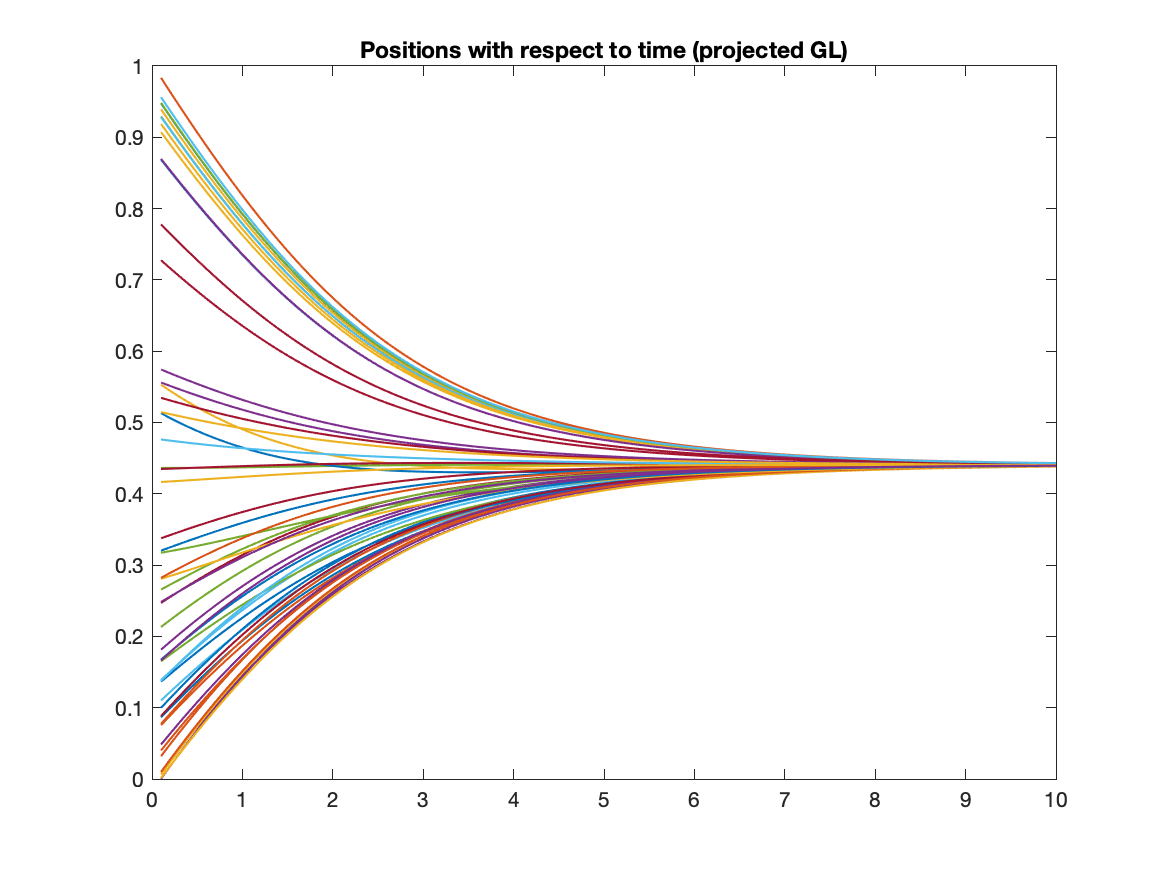}
\caption{Time evolution of the microscopic system \eqref{eq:ODEs} for $N=60$ (left), and of the corresponding projection of the graph limit \eqref{eq:graph_limit} (right), for the random weighed graph law \eqref{eq:Example_SmallWorld}.}
\label{Fig:SmallWorld_N60}
\end{figure}

Fig.~\ref{Fig:SmallWorld_Det_N60} shows the evolution of the graph limit $t\mapsto u(\cdot,t)$ solution to \eqref{eq:graph_limit} (black) and that of the solution to the microscopic system \eqref{ODE22} generated by a deterministic sequence $t\mapsto u_N(\cdot,t)$. Since the graphon $\bw$ is strongly connected, convergence to consensus can be observed both in $L^2$ and $L^\infty$ norms. 

\begin{figure}[h!]
\centering
\includegraphics[width = 0.32\textwidth]{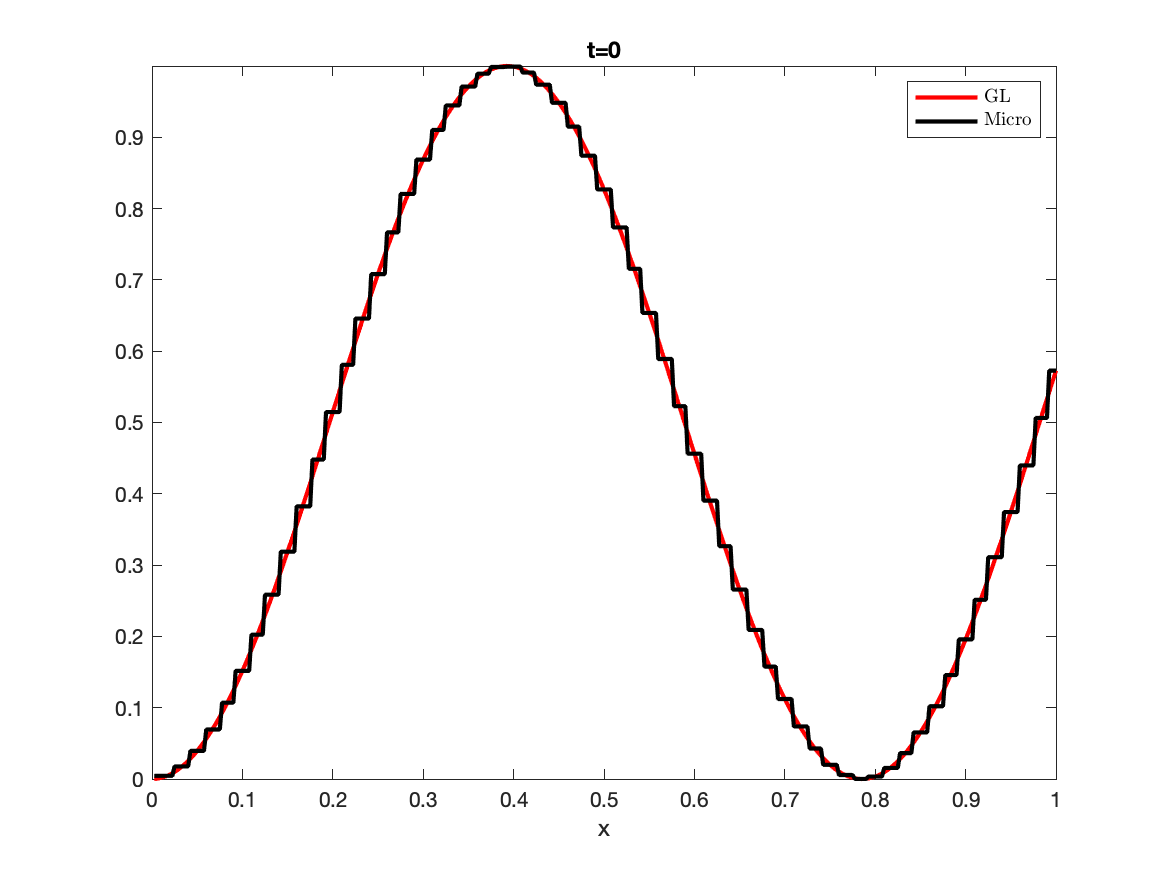}
\includegraphics[width = 0.32\textwidth]{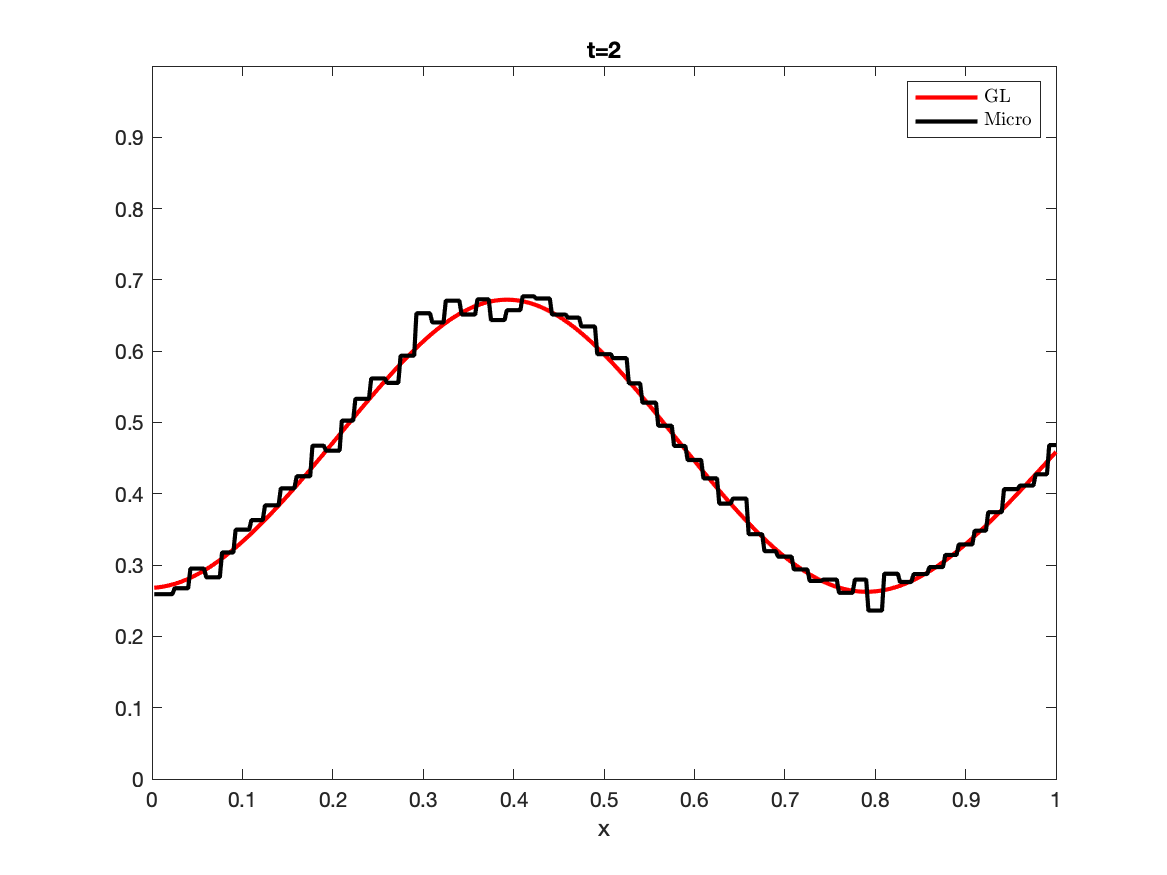}
\includegraphics[width = 0.32\textwidth]{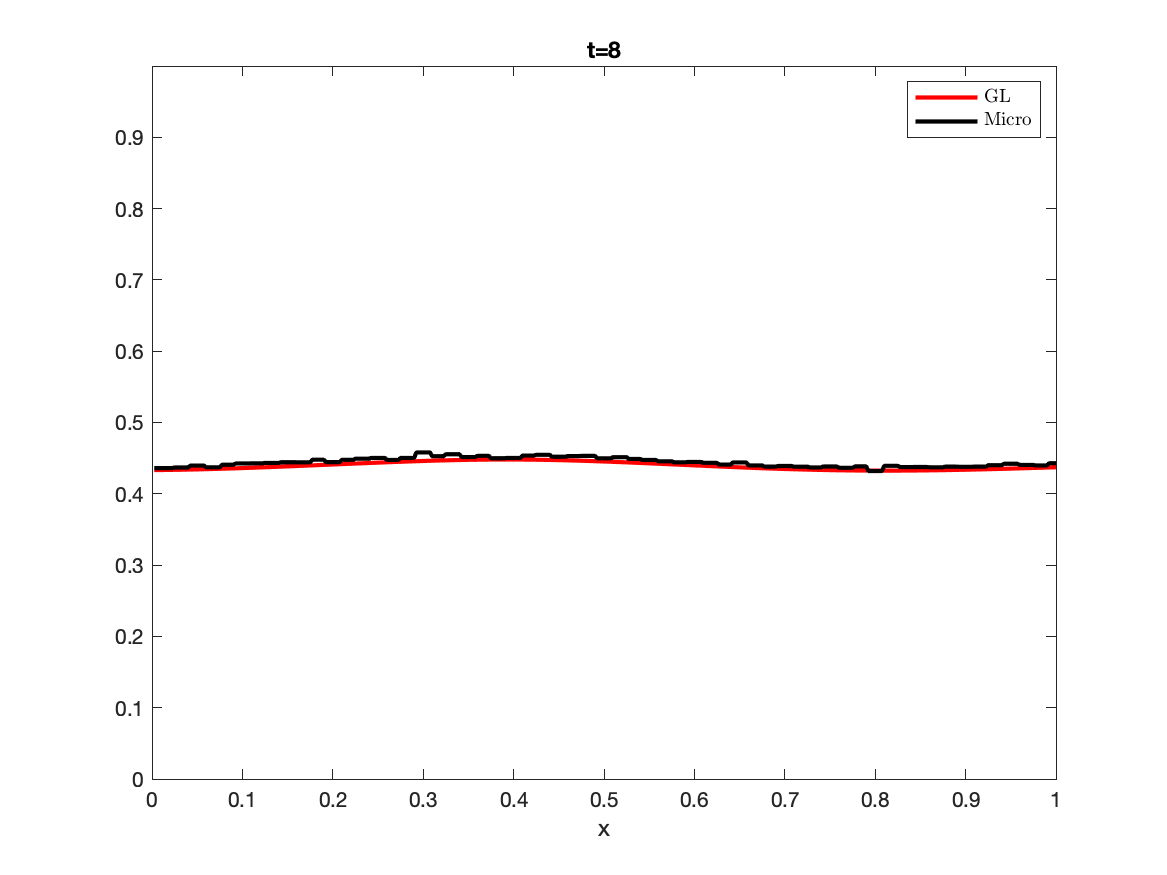}
\caption{Evolution of the graph limit $u(\cdot,t)$ solution to \eqref{eq:graph_limit} (red) and of $u_N(\cdot,t)$ constructed from the solution to \eqref{ODE22} (black) for $N=150$, at $t=0$, $t=6$ and $t=40$, for the weighted random graph law \eqref{eq:Example_SmallWorld}.}
\label{Fig:SmallWorld_Det_N60}
\end{figure}

Convergence towards the graph limit in both cases is quantified in Fig.~\ref{Fig:SmallWorld_L2norm}.

\begin{figure}[h!]
\centering
\includegraphics[width = 0.45\textwidth]{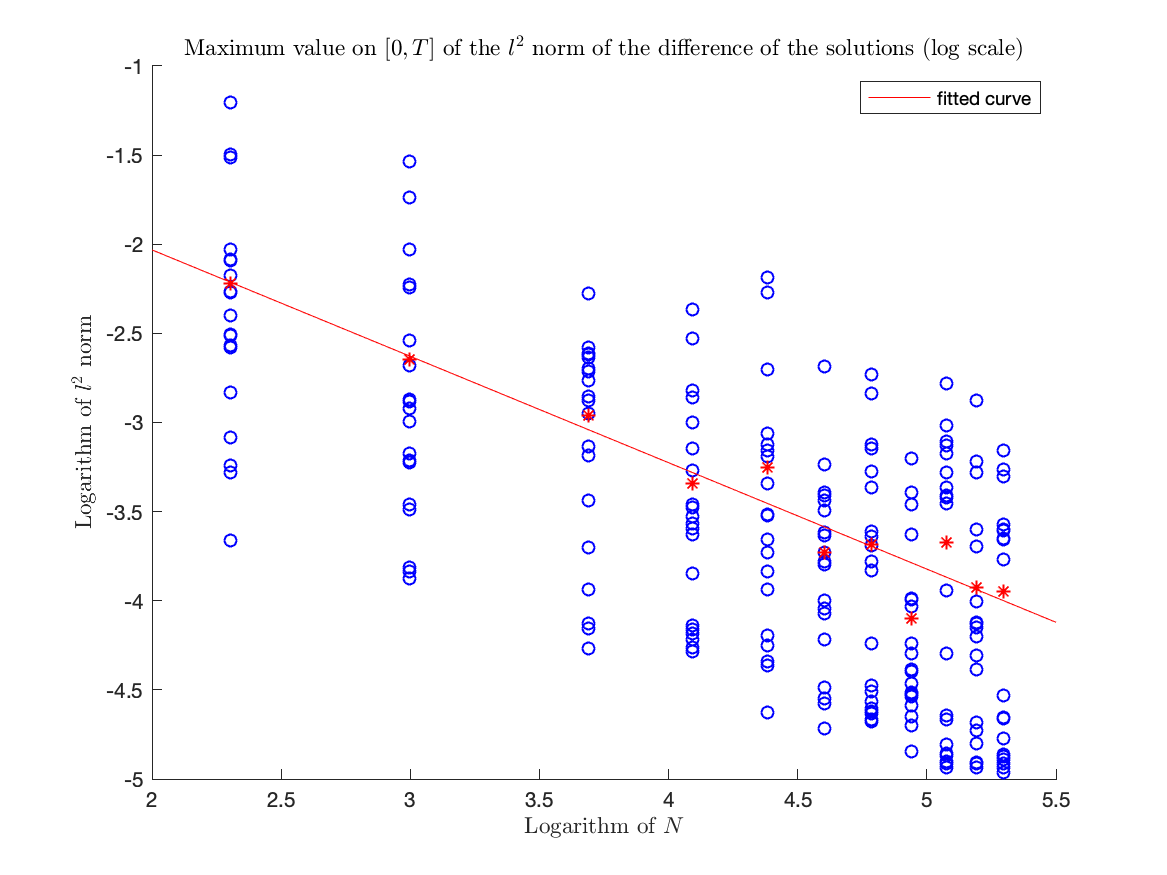}
\includegraphics[width = 0.45\textwidth]{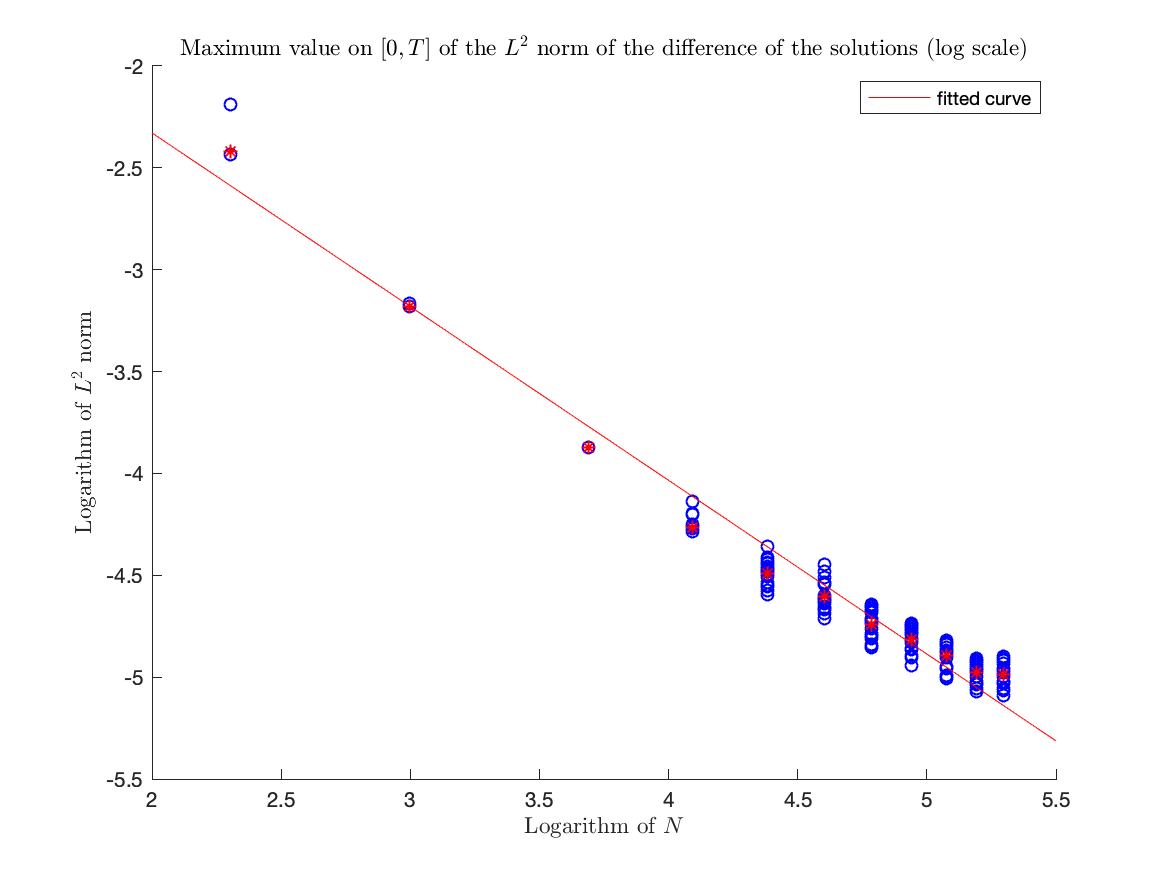}
\caption{Quantification of the convergence of the microscopic systems \eqref{eq:ODEs} and \eqref{ODE22} respectively given by $\sup_{t \in [0,T]}  \|u^N(t) - \mathbf{P}_{\tilde{X}_N} u(\cdot,t) \|_{2,N}$ (left) and $\sup_{t \in [0,T]}  \|u_N(\cdot,t) - u(\cdot,t) \|_{L^2}$ (right) for different values of $N$, with 20 runs for each value of $N$ (logarithmic scale), for the weighted random graph law \eqref{eq:Example_SmallWorld}. }
\label{Fig:SmallWorld_L2norm}
\end{figure}


\subsection{Blinking systems}

We now illustrate the results presented in Section \ref{Sec:blinking}.
Figure \eqref{Fig:Blinking_Garlaschelli_xy} shows the time evolution of the solution to the blinking system \eqref{eq:ODEbl}, with blinking period respectively $\varepsilon=1$ (left) and  $\varepsilon=0.1$ (right), and weighted random graph law $q$ given by \eqref{eq:example_Garaschelli_xy}. As previously, the interaction function and the initial data were chosen as given in equations \eqref{eq:D}.
In comparison, in Fig.~\ref{Fig:Blinking_Garlaschelli_xy_Convergence}, we show the solution to the averaged system \eqref{eq:ODEav}, as well as the projected solution to the graph limit equation \eqref{eq:graph_limit}.
Notice that the solution to the averaged system \eqref{eq:ODEav} is almost undistinguishable from that of the projected graph limit \eqref{eq:graph_limit}.
As expected, for a fixed $N$, the smaller $\varepsilon$ is, the closer the solution to the blinking system is to both the solution to the averaged system \eqref{eq:ODEav} and to the graph limit \eqref{eq:graph_limit}.
Convergence of the blinking system \eqref{eq:ODEbl} towards the graph limit \eqref{eq:graph_limit} for a fixed $\varepsilon = 0.1$ as $N$ goes to infinity is shown in Fig. \ref{Fig:Blinking_Garlaschelli_L2norm}.

\begin{figure}[h!]
\centering
\includegraphics[width = 0.45\textwidth]{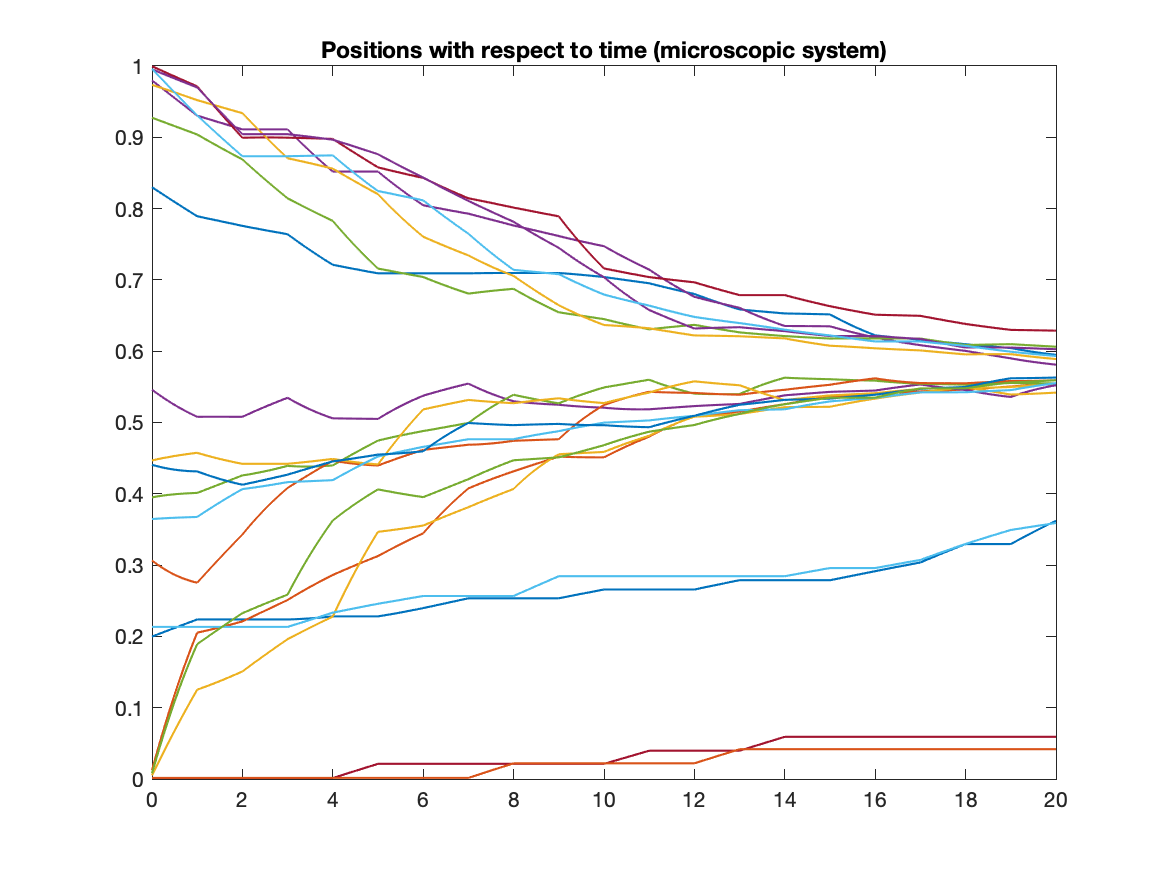}
\includegraphics[width = 0.45\textwidth]{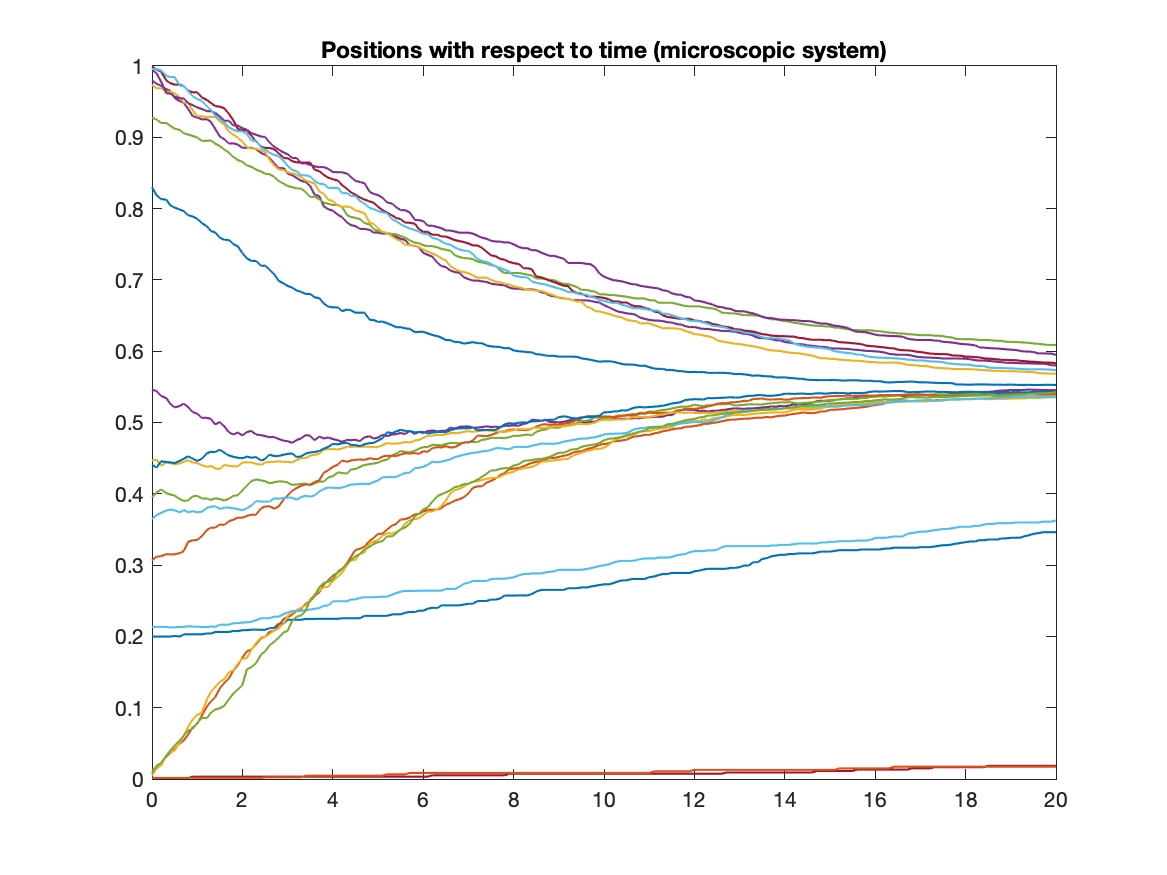}
\caption{Time evolution of blinking system \eqref{eq:ODEbl} for $N=20$, and $\varepsilon=1$ (left) and $\varepsilon=0.1$ (right) for the random weighed graph given by \eqref{eq:example_Garaschelli_xy}.}
\label{Fig:Blinking_Garlaschelli_xy}
\end{figure}

\begin{figure}[h!]
\centering
\includegraphics[width = 0.45\textwidth]{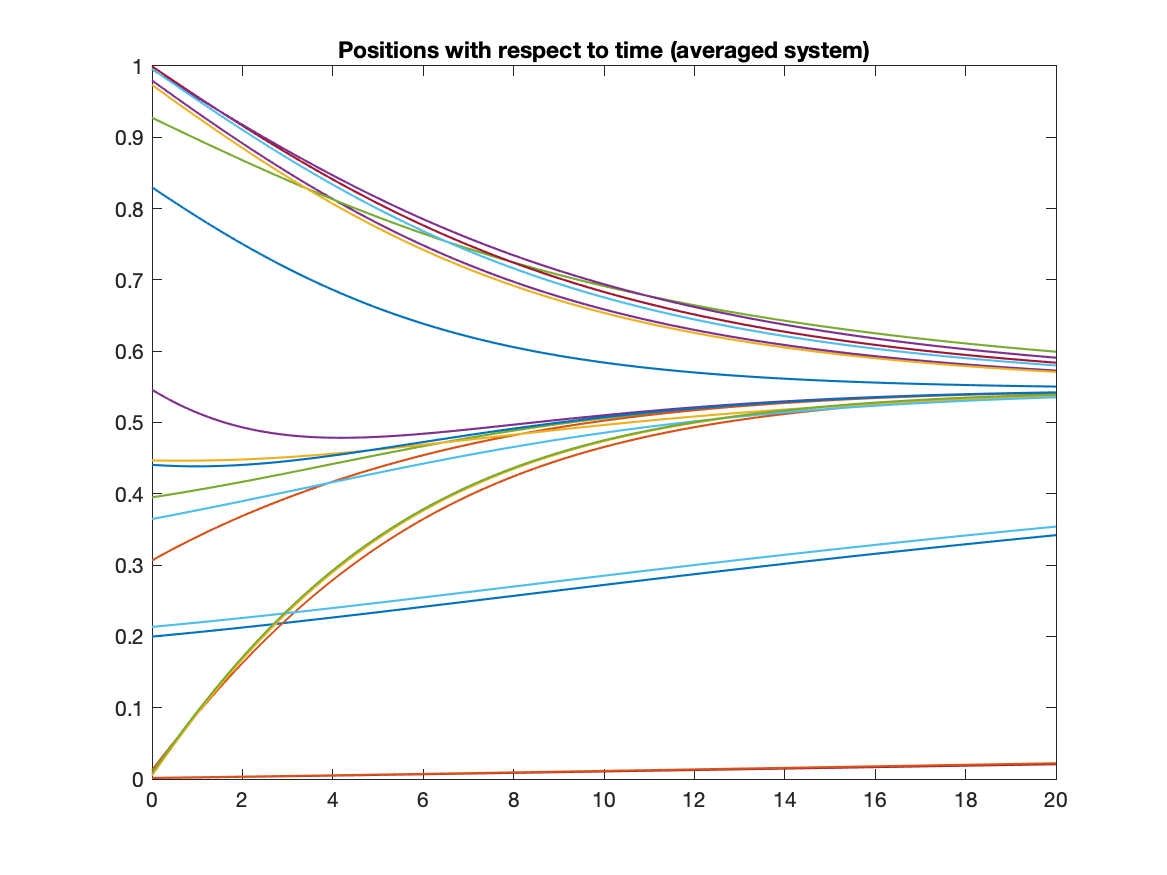}
\includegraphics[width = 0.45\textwidth]{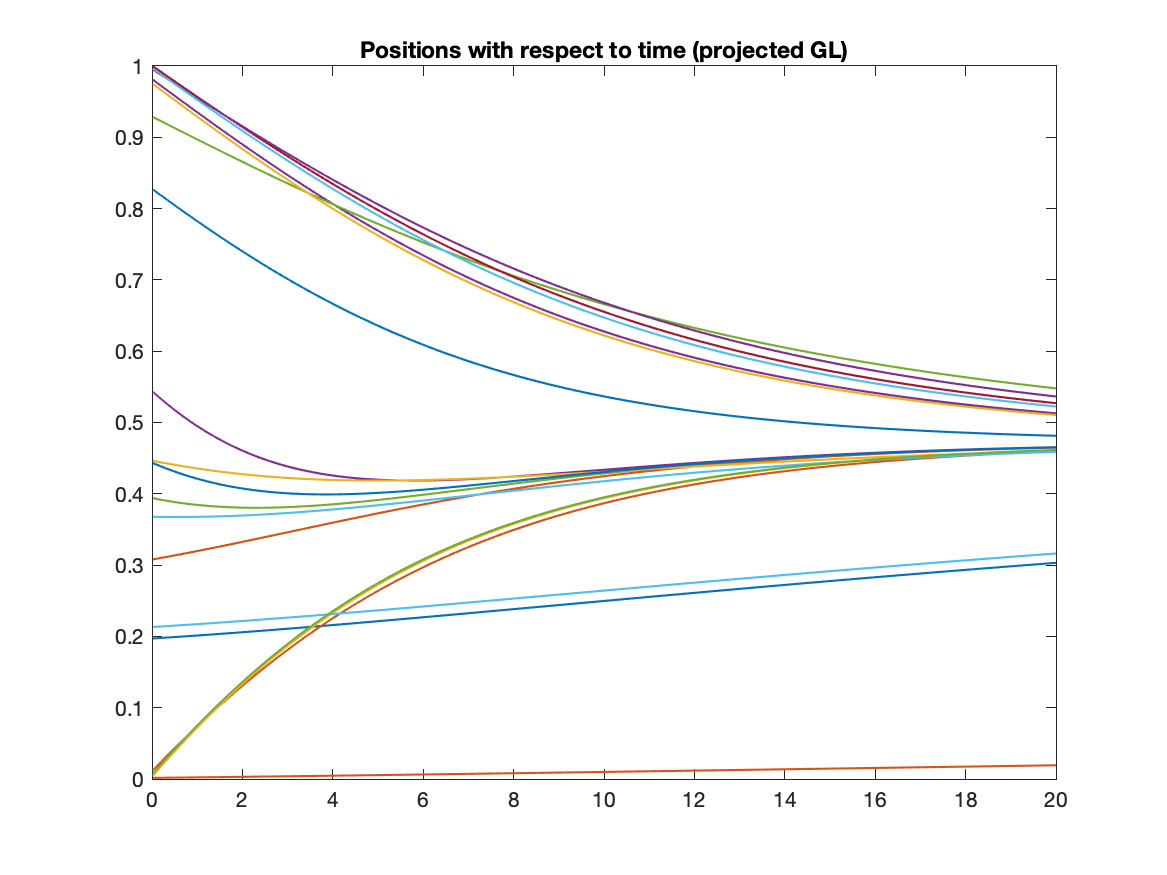}
\caption{Time evolution of averaged system \eqref{eq:ODEarv} and of the corresponding graph limit solution $(u(X_i,\cdot))_{i\in\elts}$ to \eqref{eq:graph_limit} for $N=20$, for the random weighed graph law given by \eqref{eq:example_Garaschelli_xy}.}
\label{Fig:Blinking_Garlaschelli_xy_Convergence}
\end{figure}

\begin{figure}[h!]
\centering
\includegraphics[width = 0.45\textwidth]{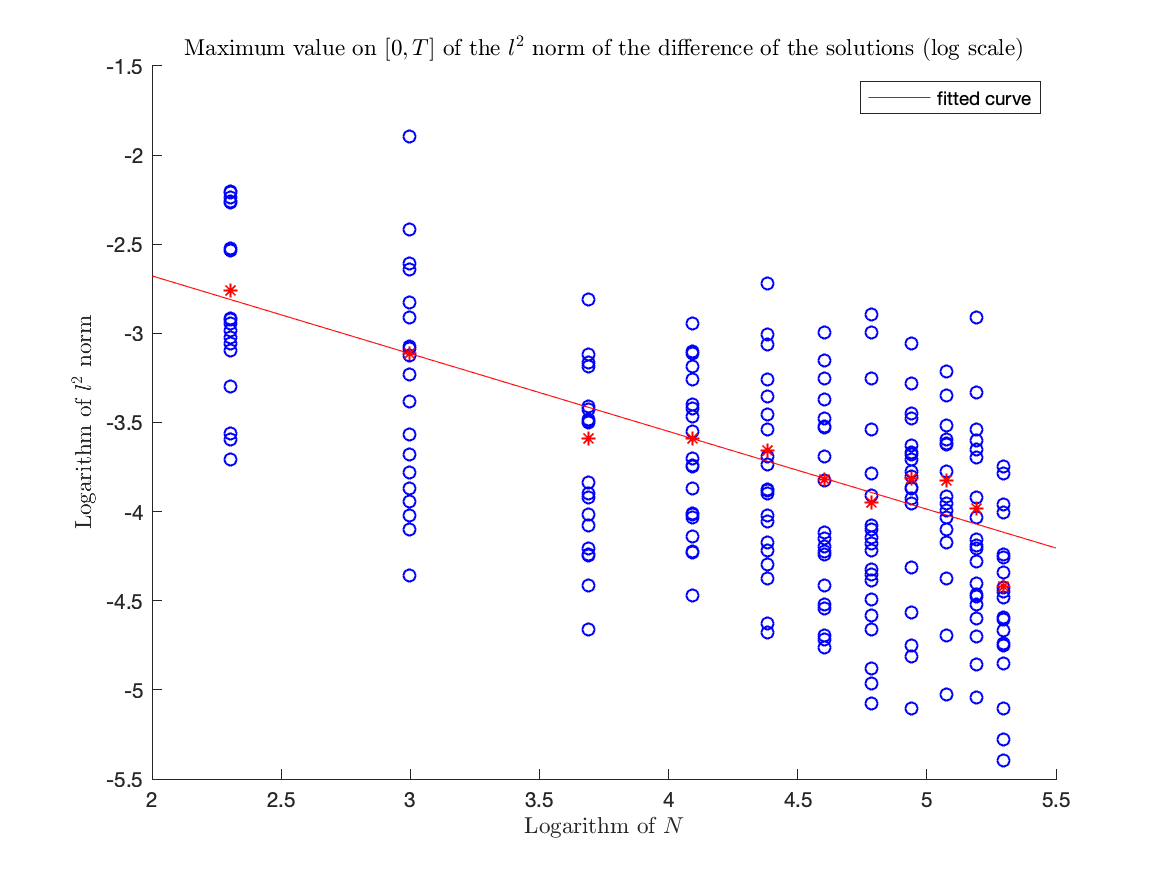}
\caption{Quantification of the convergence of the microscopic systems \eqref{eq:ODEbl} given by $\sup_{t \in [0,T]}  \|u^{N,\varepsilon}(t) - \mathbf{P}_{\tilde{X}_N} u(\cdot,t) \|_{2,N}$ for a fixed $\varepsilon=0.1$, with 20 runs for each value of $N$, for the weighted random graph law \eqref{eq:example_Garaschelli_xy} (logarithmic scale). }
\label{Fig:Blinking_Garlaschelli_L2norm}
\end{figure}

\newpage

\section{Appendix}

\subsection{Computations and bounds of variances in the random-random case } \label{App:1}

We provide the proof of Lemma \ref{Lemma:sigmaY}, which we recall for completeness.
\begin{lemma_appendix}{2}
Given a solution $u$ to the integro-differential equation \eqref{eq:graph_limit}, we consider the collection of random variables $(Y^N_i)_{i\in\elts}$, defined for all $i\in\elts$ by $Y^N_i(t):=\sqrt{N} Z^N_i(t)$, where $Z^N_i(t)$  is given by \eqref{eq:Z}.
Then for all $i\in\elts$, its variance satisfies:
$
\V\croch{(Y_i^N)} = \E\croch{(Y_i^N)^2} = \sz^2(t),
$
where
\[
\sz^2(t) := \iint_{I^2} \bw(x,y) D(u(y,t)-u(x,t))^2 dx dy 
-  \int_I \left( \int_I  \bw(x,y) D(u(y,t) - u(x,t)) \right)^2 dx.
\]
For the random variables $(Y^N_i)^2$, it holds for all $i\in\elts$,
\[
\E\croch{(Y_i^N)^4} = 3\sz^4(t) + O(\frac{1}{N}) \qquad \text{ and }  \qquad
\V\croch{(Y_i^N)^2} = 2\sz^4(t) + O(\frac{1}{N}).
\]
\end{lemma_appendix}

\begin{proof}
Let us start with the computations of $\V\croch{Y_i^N}$. Denoting $\eta_{ij} := \xi_{ij}D(u(t,X_j)-u(t,X_i))$ and
$
\mu(X_i) := \mathbb{E}\left[ \eta_{ij} | X_i\right] = \int_I \bw(X_i,y) D(u(t,y)-u(t,X_i)) dy,
$
the random variable $Y_i^2(t)$ can be rewritten as
\begin{equation*}
Y_i^2(t) =\frac{1}{N} \left(\sum_{j=1}^N (\eta_{ij}-\mu(X_i))\right)^2.
\end{equation*}
Then the term  $\E \croch{(Y_i^N)^2}$ can then be computed as follows:
\[
\begin{array}{rcl}
\E\croch{(Y_i^N)^2}&  =&  \E \croch{\E\croch{(Y_i^N)^2\,|\,X_i} } \\
& =& \dsp \frac{1}{N} \E \croch{\E\croch{\sum_{1\leq j,k\leq N} (\eta_{ij} - \mu(X_i)) (\eta_{ik}  - \mu(X_i)) \,|\,X_i}}\\
& =& \dsp\E \croch{\E\croch{ \sum_{j=1}^N( \eta_{ij} - \mu(X_i))^2\,|\,X_i}}  + \frac{2}{N}   \E \croch{\E\croch{ \sum_{1 \leq j <k \leq N }    (\eta_{i{j}} - \mu(X_i))(\eta_{i{k}} - \mu(X_i)) \,|\,X_i }}.
\end{array}
\]
By independence, we deduce that 
\[
\begin{array}{rcl}
\E\croch{(Y_i^N)^2}&  =&  \dsp\E \croch{\E\croch{ \sum_{j=1}^N( \eta_{ij} - \mu(X_i))^2\,|\,X_i}}.
\end{array}
\]
 Since, we have 
 \[
\begin{split}
\E \croch{\pare{\eta_{ij}-\mu(X_i)}^2 | X_i} =
\E \croch{\eta_{ij}^2 |X_i} - 2 \mu(X_i) \E \croch{\eta_{ij} |X_i} +\mu(X_i)^2 =\E \croch{\eta_{ij}^2 |X_i} - \mu(X_i)^2,
\end{split}
\]
we deduce that $\E\croch{(Y_i^N)^2}=  \sigma_Y^2$, where
$$\sigma_Y^2(t) :=   \iint_{I^2} \int_{\R_+} w^2  D(u(y,t)-u(x,t))^2 q(x,y;dw) dx dy  -  \int_I \left(\int_{\R_+} \int_I \bw(x,y) D(u(y,t) - u(x,t))  \right)^2. $$

We continue with the computations of $\V\croch{(Y_i^N)^2}$. We have
\begin{equation}\label{eq:sigmaY}
\V\croch{(Y_i^N)^2} = \E\croch{((Y_i^N)^2-\sigma_Y^2)^2} = \E \croch{(Y_i^N)^4} - \sigma_Y^4.
\end{equation}
The term  $\E \croch{(Y_i^N)^4}$ can then be computed as follows:
\[
\begin{split}
&\E\croch{(Y_i^N)^4\,|\,X_i}  = \E\croch{N^2 \big(\frac{1}{N} \sum_{j=1}^N \eta_{ij} - \E(\eta_{ij}\,|\,X_i) \big)^4\,|\,X_i} 
=\E\croch{\frac{1}{N^2}\big( \sum_{j=1}^N( \eta_{ij} - \mu(X_i)) \big)^4\,|\,X_i}  \\
 = & \E\croch{\frac{1}{N^2} \sum_{j_1=1}^N  \sum_{j_2=1}^N  \sum_{j_3=1}^N  \sum_{j_4=1}^N(\eta_{i{j_1}} - \mu(X_i))(\eta_{i{j_2}} - \mu(X_i))(\eta_{i{j_3}} - \mu(X_i))(\eta_{i{j_4}} - \mu(X_i)) \,|\,X_i } \\
 = & \E\croch{\frac{1}{N^2} \sum_{j=1}^N (\eta_{i{j}} - \mu(X_i))^4 \,|\,X_i } + 3 \E\croch{\frac{1}{N^2} \sum_{j_1=1}^N  \sum_{j_2\neq j_1}  (\eta_{i{j_1}} - \mu(X_i))^2(\eta_{i{j_2}} - \mu(X_i))^2 \,|\,X_i }.
\end{split}
\]
For the first term, it holds
\[
\begin{split}
\E \croch{\pare{\eta_{ij}-\mu(X_i)}^4 | X_i} =
\E \croch{\eta_{ij}^4 |X_i} - 4 \mu(X_i) \E \croch{\eta_{ij}^3 |X_i} + 6 \mu(X_i)^2 \E \croch{\eta_{ij}^2 |X_i} -3 \mu(X_i)^4 .
\end{split}
\]
Using Hypothesis \ref{hypo_moment}, each of these four quantities is bounded by $(M \|D\|_{L^\infty})^4$, so 
\[
{A_N:=\E}\croch{\E\croch{\frac{1}{N^2} \sum_{j=1}^N (\eta_{i{j}} - \mu(X_i))^4 \,|\,X_i }} = {\frac{1}{N} \E\croch{\E\croch{ (\eta_{i{j}} - \mu(X_i))^4 \,|\,X_i }} := \frac{\mu_4}{N}} = O\left(\frac{1}{N}\right).
\]
Secondly, by independence of the random variables that we consider, it holds
\[
\begin{split}
  {B_N:=} &\; \E\E\croch{\frac{1}{N^2} \sum_{j_1=1}^N  \sum_{j_2\neq j_1}  (\eta_{i{j_1}} - \mu(X_i))^2(\eta_{i{j_2}} - \mu(X_i))^2 \,|\,X_i }  
 = \frac{N(N-1)}{N^2} \E\E\croch{(\eta_{i{j}} - \mu(X_i))^2\,|\,X_i }^2 \\
  = &\;  \frac{N(N-1)}{N^2} \E\croch{ (\E\croch{\eta_{i{j}}^2\,|\,X_i } - \mu(X_i)^2)}^2 =  {(1-\frac{1}{N})\sigma_Y^4}.
\end{split}
\]
We then have 
\begin{equation}
\E\croch{(Y_i^N)^4} = \E\E\croch{(Y_i^N)^4\,|\,X_i} = {A_N+3B_N = 3 (1-\frac{1}{N})\sigma_Y^4 + \frac{\mu_4}{N} }.
\end{equation}
\end{proof}

\subsection{Computations and bounds of variances in the random-deterministic case } \label{App:2}

Denoting $W_{ij}^N:= \bw(x_i^N,x_j^N)$, we study the convergence of the random variable
\[
\dsp \tZ_i^N := \frac{1}{N}   \sum_{j=1}^N \pare{ \xi_{ij}  - W_{ij}^N} D(\vjn - \vin).
\]
We denote $\tY_i^N := \sqrt{N} \tZ_i^N$, which allows us to write  $\|\tZ^N\|^2 = \frac{1}{N}\sumi (\tY_i^N)^2$.
Then applying the Bienaymé-Chebyshev inequality, it holds
\[
\PP\croch{ \left| \frac{1}{N}\sumi (\tY^N_i)^2 - \E\croch{\frac{1}{N}\sumi (\tY^N_i)^2} \right| \geq 1}\leq \V\croch{\frac{1}{N}\sumi (\tY^N_i)^2}
\]
We compute the expectation and variance of $(\tY^N_i)^2$.
Denoting $\gamma_{ij}^N := \xi_{ij} f_{ij}^N$ where $f_{ij}^N := D(v_j^N -v_i^N)$, it holds
 $$\begin{array}{rcl}
  \E \croch{(\tY_i^N)^2 } & =& \dsp \frac{1}{N} \E \croch{\sum_{1\leq j,k \leq N} \pare{(\xi_{ij} - \wijn)\fijn}  \pare{(\xi_{ik} - \wijk)\fijk}}  \\
&  =& \dsp \frac{1}{N} \E \croch{\sum_{1\leq j \leq N} \pare{(\xi_{ij} - \wijn)\fijn}^2} + \frac{2}{N} \E \croch{\sum_{1\leq j<k \leq N} \pare{(\xi_{ij} - \wijn)\fijn}  \pare{(\xi_{ik} - \wijk)\fijk}}\\
 & =& (\sigma_i^N)^2
 \end{array}$$
 where $$\begin{array}{rcl} 
(\sigma_i^N)^2 & := & \dsp \frac{1}{N} \E \croch{\sum_{1\leq j \leq N} \pare{(\xi_{ij} - \wijn)\fijn}^2} \\
 & = & \dsp \frac{1}{N} \sum_{j=1}^N (\fijn)^2 \croch{\int_{\R_+} w^2 q(x_i^N; x_j^N;dw) - \pare{\int_{\R_+} w q(x_i^N; x_j^N;dw) }^2} 
 \leq K^2 M^2.
  \end{array}$$
  
Similarly, 
   $$\begin{array}{rcl}
 \E \croch{(\tY_i^N)^4}& =& \dsp  \frac{1}{N^2} \E \croch{\sum_{1\leq j_1,j_2,j_3,j_4 \leq N} ((\xi_{ij_1} - W_{ij_1}^N)f_{ij_1}^N) \dots ((\xi_{ij_4}^N - W_{ij_4}^N)f_{ij_4}^N ) }\\
  & = &\dsp \frac{6}{N^2}\sum_{1\leq j < k \leq N} \E \croch{((\xi{ij}^N - W_{ij}^N)f_{ij}^N)^2} \E \croch{((\xi_{ik}^N - W_{ik}^N)f_{ik}^N)^2} \\
  & & +\dsp \frac{1}{N^2} \sum_{1\leq j \leq N}  \E \croch{((\xi_{ij}^N -  W_{ij}^N)f_{ij}^N)^4} \\
  & = & \dsp \frac{6 N (N-1)}{2 N^2} (\sigma_i^N)^4 + O(N^{-1}) =  3(\sigma_i^N)^4 + O(N^{-1}) .
 \end{array}$$
 Thus, $\V \croch{(\tY_i^N)^2}= 2(\sigma_i^N)^4 + O(N^{-1}) \leq 3 K^4M^4$ 
 for $N$ large enough.\\


\textit{Acknowledgements.} The authors are grateful to Pierre Monmarché and Thierry Paul for helpful remarks and suggestions  on how to improve our paper. 


\bibliographystyle{abbrv}
\bibliography{biblio}

\end{document}